\documentclass[11pt]{amsart}
\usepackage{amsmath, amssymb}
 \usepackage{epsfig}

\theoremstyle{plain}
\usepackage{amsmath}
\usepackage{amssymb}
\usepackage{amsfonts}
\usepackage{xcolor}
\newtheorem{theorem}{Theorem}
\newtheorem{proposition}{Proposition}
\newtheorem{lemma}{Lemma}
\newtheorem{remark}{Remark}

\numberwithin{equation}{section}

\newcommand{\ignore}[1]{}{}

\def \R {\mathbb{R}}

\def \E {\mathbb{E}}
\def \P {\mathbb{P}}

\def \F  {\mathbb{F}}

\def \< {\langle}
\def \> {\rangle}
\def \^ {\widehat}

\def \tr {{\rm tr}}

\def \centre {{\rm center}}

\newcommand{\bbA}{{\bf A}}

\newcommand{\bbB}{{\bf B}}
\newcommand{\bbC}{{\bf C}}

\newcommand{\bbD}{{\bf D}}

\newcommand{\bbe}{{\bf e}}

\newcommand{\bbG}{{\bf G}}

\newcommand{\bbI}{{\bf I}}

\newcommand{\bbs}{{\bf s}}

\newcommand{\bbT}{{\bf T}}

\newcommand{\bbX}{{\bf X}}

\newcommand{\bbx}{{\bf x}}

\newcommand{\beq}{\begin{equation}}
\newcommand{\eeq}{\end{equation}}
\newcommand{\bqa}{\begin{eqnarray}}
\newcommand{\eqa}{\end{eqnarray}}
\newcommand{\bqn}{\begin{eqnarray*}}
\newcommand{\eqn}{\end{eqnarray*}}
\newcommand{\non}{\nonumber \\}
\newcommand{\bdes}{\begin{description}}
\newcommand{\edes}{\end{description}}

\def\underwiggle 1{\ifmmode\setbox\TempBox=\hbox{$ 1$}\else\setbox\TempBox=\hbox{1}\fi
\setbox\TempBoxA=\hbox to \wd\TempBox{\hss\char'176\hss}
\rlap{\copy\TempBox}\smash{\lower9pt\hbox{\copy\TempBoxA}} }
\begin{document}

\title[ CLT of the smoothed empirical spectral distribution]{Universality of sample covariance matrices: CLT of the smoothed empirical spectral distribution}

\author{ Guangming Pan, \  \ Qi-Man Shao, \ \ Wang Zhou \\
}
\thanks{ G.M. Pan was partially supported by a grant M58110052 at the Nanyang Technological University;
Q.M. Shao was partially supported by Hong Kong RGC CERG 602608 and 603710;   W. Zhou was partially supported by a grant
R-155-000-106-112 at the National University of Singapore}

\address{Division of Mathematical Sciences, School of Physical and Mathematical Sciences, Nanyang Technological University, Singapore 637371}
\email{gmpan@ntu.edu.sg}

\address{Department of Mathematics, Hong Kong University of Science and Technology, Clear Water Bay, Kowloon, Hong Kong}
\email{maqmshao@ust.hk}

\address{Department of Statistics and Applied Probability, National University of
 Singapore, Singapore 117546}
\email{stazw@nus.edu.sg} 

\maketitle

\begin{abstract}
A central limit theorem  (CLT) for the smoothed empirical spectral distribution of sample covariance matrices is established.
 Moreover, the CLTs for the smoothed quantiles of  Marcenko and Pastur's law have been also developed.
 \end{abstract}

\def\theequation{\thesection.\arabic{equation}}

\section{Introduction}

\setcounter {equation}{0}
The sample covariance matrix is defined by
$$
\bbA_n=\frac{1}{n}\bbX_n\bbX_n^T,  \ \text{where} \
\bbX_n=(X_{ij})_{p\times n}
$$
with \{$X_{ij}$\}, $i,j=\cdots,$
being a double array of independent and identically distributed (i.i.d.) real random variables (r.v.'s) with
$\E X_{11}=0 $ and $ \E X_{11}^2=1$. In the large dimensional random matrix theory, the
sample covariance matrix is a prominent model. One reason is that its eigenvalues are
 not only interesting in its own right, but also play important roles in many other areas
  of mathematics and engineering, such as combinatorics \cite{kks},
   mathematical physics \cite{me}, probability \cite{agz},  statistics \cite{john}
   and wireless communications \cite{PZ}. Its study dates back to
   the work of Wishart \cite{Wishart}, who considered the case where all $X_{ij}$ are Gaussian r.v.'s.
In this particular model, the joint distribution of the eigenvalues of $\bbA_n$ can be explicitly computed (as a special case of the Laguerre orthogonal ensemble). One can use this explicit formula to directly obtain the law of
 local eigenvalues in large dimensions, such as the distribution of the largest one \cite{Johan00, john}, the smallest
one \cite{Edel}, and the bulk ones \cite{Su}. Also  it is widely conjectured that these limiting behaviors hold for a much larger class of
sample covariance matrices. For recent progress in this direction, we refer to \cite{ESYJ}. By comparison, the universality of the Wigner matrices and $\beta$-ensembles is well developed, see  \cite{VV} and \cite{ESY}.

However, in order to capture the whole picture of the eigenvalues of sample covariance matrices,
it is necessary to study the behavior of all eigenvalues.
A good candidate for this purpose is the empirical spectral distribution (ESD) defined by
$$
F^{\bbA_n}(x)=\frac{1}{p}\sum\limits_{k=1}^pI(\lambda_k\leq x),
$$
where $\lambda_k,k=1,\cdots,p$ are the eigenvalues of $\bbA_n$.
It is equivalent to consider
$$
\bbB_n=\frac{1}{n}\bbX_n^T\bbX_n,
$$
because the eigenvalues of $\bbA_n$ and $\bbB_n$ differ by $|n-p|$
zero eigenvalues.  The almost sure convergence of $F^{\bbA_n}$ to the
famous Marcenko-Pastur law (MP law) is fully understood under the 2nd moment condition of
 $X_{11}$ when the dimension $p$ is of the same order as the sample size $n$.
 There has been a vast literature on this topic. One can refer to
 %%We only cite
 the pioneer work \cite{MP} and the recent book \cite{b7}.

After establishing the strong law of large numbers (SLLN),
one may wish to prove the central limit theorem (CLT).
However, as far as we know, even for the Wishart ensemble, there is no CLT available in the literature
about $F^{\bbA_n}(\cdot)$ due to the shortage of powerful tools.
Hence it is also impossible to make inference based on the individual
eigenvalue of the sample covariance matrix when one only has finite moment conditions. These difficulties push one to seek other possible ways to make statistical inference.

Motivated by the ``smoothing" ideas, Jing, Pan, Shao and Zhou \cite{ker}
propose the following  kernel estimators of the distribution function of the MP law,
\bqa
F_n(x) &=&\int_{-\infty}^x\, f_n(y)\,dy, \label{Fn}
\eqa
where
\bqa
f_n(x) &=&\frac{1}{ph}\sum\limits_{i=1}^pK(\frac{x-\lambda_i}{h})=\frac{1}{h}\int
K(\frac{x-y}{h})dF^{\bbA_n}(y)\label{a1},
\eqa
$K(\cdot)$ is a smooth function and $h$ is the bandwidth tending to zero as $n\to \infty$. Intuitively, $F_n(\cdot)$ depicts the global picture of all eigenvalues
 and should have no much difference from $F^{\bbA_n}(\cdot)$.  It was proved in \cite{ker}
that $F_n(\cdot)$ almost surely converges to the MP law under some
regularity conditions.

The main aim of this paper is to establish the CLTs for
$F_n(x)$, the smoothed version of the empirical spectral distribution $F^{\bbA_n}$, and $f_n(x)$. Moreover we develop CLTs for
 the $\alpha$-th quantile of $F_n(\cdot)$, which is a
 smoothed version of the $[p\alpha]$-th largest eigenvalue of $\bbA_n$.

\section{main results}
We first introduce some necessary notation, assumptions and some basic facts about the MP law.

 In this paper, we suppose that the ratio of the dimension and sample size
$c_n=p/n$ tends to a positive constant $c$ as $n\to \infty$.
Then $F^{\bbA_n}(\cdot)$ tends to
the so-called Marcenko and Pastur law with the density function
$$
f_c(x)=\begin{cases}(2\pi cx)^{-1}\sqrt{(b-x)(x-a)} & a\leq x\leq b.\\
0 &\text{otherwise}.
\end{cases}
$$
It has point mass $1-c^{-1}$ at the origin if $c>1$, where
$a=(1-\sqrt c)^2$ and $b=(1+\sqrt c)^2$ (see \cite{b7}). The distribution function of the MP law is denoted by $\F_c(\cdot)$.
The Stieltjes transform of the MP law is
\begin{equation}
m(z)=\frac{1-c-z+\sqrt{(z-1-c)^2-4c}}{2cz},\label{m2*}
\end{equation}
 which satisfies the equation
\begin{equation}\label{m2}
m(z)=\frac{1}{1-c-czm(z)-z}.
\end{equation}
Here the Stieltjes transform $m_F(\cdot)$ for any probability
distribution function $F(\cdot)$ is given by
\begin{equation}\label{b5*}
m_F(z)=\int\frac{1}{x-z}dF(x),\ \ z\in\mathcal{C}^+.
\end{equation}
The relationship between the Stieltjes transform of the limit of  $F^{\bbB_n}(\cdot)$ and $m(\cdot)$ is given by
\begin{equation}\label{m45}
\underline{m}(z)=-\frac{1-c}{z}+cm(z).
\end{equation}
which gives the equation satisfied by $\underline{m}(\cdot)$
 \begin{equation}\label{h31}
z=-\frac{1}{\underline{m}(z)}+\frac{c}{1+\underline{m}(z)}.
\end{equation}

For the kernel function $K(\cdot)$ we assume that
\begin{equation}\label{a25}
\lim\limits_{|x|\rightarrow\infty}|xK(x)|=\lim\limits_{|x|\rightarrow\infty}|xK'(x)|=0,\
\end{equation}
 \begin{equation} \label{a26} \int K(x)dx=1, \ \ \int
|xK'(x)|dx<\infty,  \ \ \int |K''(x)|dx<\infty.
\end{equation}
and
\begin{equation} \label{a27}
\int xK(x)dx=0,  \int x^2|K(x)|dx<\infty.
\end{equation}
Let $z=u+iv$, where $u\in \R$ and $v$ is in a bounded interval, say $[-v_0,v_0]$ with $v_0>0$. Suppose that
\begin{equation}\label{f63}
\int^{+\infty}_{-\infty}|K^{(j)}(z)|du<\infty, \quad j=0,1,2,
\end{equation}
uniformly in $v\in [-v_0,v_0]$, where $K^{(j)}(z)$ denotes the $j$-th derivative of $K(z)$. Also suppose that
\begin{equation}\label{a25*}
\lim\limits_{|x|\rightarrow\infty}|xK(x+iv_0)|=\lim\limits_{|x|\rightarrow\infty}|xK'(x+iv_0)|=0.
\end{equation}

   Our first result is the CLT for $(F_n(x)-\F_{c_n}(x))$.

 \begin{theorem}\label{theo2}
 Suppose that

\begin{itemize}
\item[1)]  $h = h(n)$ is a sequence of positive constants satisfying
$$
\lim\limits_{n\rightarrow\infty}\frac{nh^2}{\sqrt{\ln h^{-1}}}\rightarrow 0, \quad \lim\limits_{n\rightarrow\infty}\frac{1}{nh^2}\rightarrow 0,\quad
$$

\item[2)] $K(x)$ satisfies (\ref{a25})-(\ref{a25*}) and is analytic on
open interval including
$$
[\frac{a-b}{h},\frac{b-a}{h}] \ ;
$$

\item[3)]  $X_{ij}$ are i.i.d. with $\E X_{11}=0$, $Var(X_{11})=1$,
$\E X_{11}^4=3$ and
$\E X_{11}^{32}<\infty$,  $c_n\rightarrow c\in (0,1)$ ;

\end{itemize}
 Then, as
$n\rightarrow\infty$, for any fixed positive integer $d$ and different points $x_1,\cdots,x_d$ in $(a,b)$, the joint limiting distribution  of
\begin{equation}\label{g50}
\frac{\sqrt{2}\pi n}{\sqrt{\ln n}}\Big(F_n(x_j)-\F_{c_n}(x_j)\Big),\quad  \ j=1,\cdots,d
\end{equation}
is multivariate normal with mean zero and
covariance matrix $I$, the $d\times d$ identity matrix.
\end{theorem}

\begin{remark}
The convergence rate $n/\sqrt{\ln n}$ is consistent with the
conjectured convergence rate $n/\sqrt{\ln n}$ of the ESD of sample covariance matrices to the MP law.
\end{remark}

\begin{remark}
It is easy to check that the Gaussian kernel $(2\pi)^{-1/2}e^{-x^2/2}$ satisfies all conditions specified in Theorem \ref{theo2}.
\end{remark}

Based on Theorem \ref{theo2} we may further develop the smoothed quantile estimators of the MP law. For $0<\alpha<1$, define the $\alpha$-quantile of the MP law by
\begin{equation}\label{m39}
x_\alpha=\inf \{x,\F_{c_n}(x)\geq \alpha\}
\end{equation}
and its estimator by
\begin{equation}\label{m40}
x_{n,\alpha}=\inf \{x,F_{n}(x)\geq \alpha\}.
\end{equation}

\begin{theorem}\label{theo3} Under the assumptions of Theorem \ref{theo2},
$$
\frac{n}{\sqrt{\ln n}}(x_{n,\alpha}-x_\alpha) \rightarrow N(0,\frac{1}{2\pi^2f_c^2(x_\alpha)}),\quad x_\alpha\in (a,b).
$$

\end{theorem}

The next theorem is the CLT for $f_n(x)$.

\begin{theorem}\label{rem2}
Suppose that
\begin{itemize}
\item[1)]  $h = h(n)$ is a sequence of positive constants satisfying
\begin{equation}\label{band}
\lim_{n\to \infty}\frac{\ln h^{-1}}{nh^2}\rightarrow 0, \lim_{n\to \infty}nh^3=0 \ ;
\end{equation}

\item[2)] $K(x)$ satisfies (\ref{a25})-(\ref{a25*}) and is analytic on
open interval including
$$
[\frac{a-b}{h},\frac{b-a}{h}] \ ;
$$

\item[3)]  $X_{ij}$ are i.i.d. with $\E X_{11}=0$, $Var(X_{11})=1$,
$\E X_{11}^4=3$ and
$\E X_{11}^{32}<\infty$,  $c_n\rightarrow c\in (0,1)$ ;

\end{itemize}
Then, as $n\rightarrow\infty$, for any fixed positive integer $d$ and different points $x_1,\cdots,x_d$ in $(a,b)$,  the joint limiting  distribution
of
\begin{equation}\label{g49}
 nh\Big(f_n(x_j)-f_{c_n}(x_j)\Big), \quad  \ j=1,\cdots,d
\end{equation}
is  multivariate normal with mean zero and covariance matrix $\sigma^2 \, I$,
where
 $$
\sigma^2=-\frac{1}{2\pi^2}\int_{-\infty}^{+\infty}
\int_{-\infty}^{+\infty}K'(u_1) K'(u_2)\ln (u_1-u_2)^2du_1du_2.
$$
\end{theorem}

Note that the Gaussian kernel $(2\pi)^{-1/2}e^{-x^2/2}$ also satisfies all conditions specified in Theorem \ref{rem2}.
 Theorem \ref{rem2} is actually a corollary of the following theorem.

\begin{theorem} \label{theo1}
 When the condition $\lim_{n\to \infty}nh^3=0$ in Theorem \ref{rem2}
  is replaced by $$\lim_{n\to \infty} h=0$$ while the remaining conditions are unchanged, Theorem \ref{rem2} holds as well if the
random variables (\ref{g49}) are replaced by
$$nh\Big[f_n(x_j)-\frac{1}{h}\int^b_a
K(\frac{x_j-y}{h})d\F_{c_n}(y)\Big], \ \  \quad x_j\in (a,b), \ j=1,\cdots,d
$$
\end{theorem}

The paper is organized as follows.
Theorem \ref{theo1} is proved in Section \ref{finite-dim}, and some calculations involved in the proof are deferred to Appendix 2.
In Section \ref{con-Em}, we establish the optimal orders for $\E(\Gamma(z))^2$ and $\E(\Gamma(z))^3$
 where $\Gamma(z)=n^{-1}\tr\bbA^{-1}(z)-n^{-1}\E\tr\bbA^{-1}(z)$ with $\bbA^{-1}(z)=(\bbA_n-z\bbI)^{-1}$, $z=u+iv$ and $v\geq M/\sqrt{n}$ for some constant $M$. It is the most difficult and important result of this paper.
In Section \ref{lim-mean}, we derive the limit of $ (2\pi
i)^{-1}\oint K((x-z)/h)n(\E m_n(z)-m_n^0(z))dz $, which is essential
to Theorem \ref{theo2}. The proof of Theorems \ref{rem2} and
\ref{theo2} is completed in Section \ref{proof-31}. Section
\ref{proof-2} handles Theorem \ref{theo3}. Some technical lemmas are
given in Appendix 1.

Before concluding this section, let us say a few words about the proof of Theorem \ref{theo2}.
Key breakthroughs are to establish optimal orders for $\E(\Gamma(z))^2$ and $\E(\Gamma(z))^3$ with $z=u+iv$ and $v\geq M/\sqrt{n}$ for some constant $M$. This turns out to be quite challenging when $v$ is of the order $n^{-1/2}$. The best order obtained so far is
$
\E|\Gamma(z)|^2\leq M(nv)^{-2}|z+c-1+2zcm(z)|^{-2},
$
or $M/(n^2v^3)$ (see Proposition 6.1 in \cite{g2}). Roughly speaking we establish
$$
\frac{1}{|z+c-1+zcm(z)+zc\E
m_n(z)|}|\E(\Gamma(z))^2|\leq\frac{M}{n^2v^2|z+c-1+2zcm(z)|}
$$
and
$$
|\E(\Gamma(z))^3|\leq M/(n^2v^2).
$$
To this end, we develop a precise order of
$n^{-1}\E\tr\bbA^{-2}(z)$, which is of $v^{-1/2}$. Also, some sharp
bound for $n^{-1}\E\tr\bbA_k^{-2}(z)\underline{\bbA}_k^{-1}(z)$ is
established (the definition of $\underline{\bbA}_k^{-1}(z)$ is given
right before section 3.1). Indeed, these results also imply that the
convergence rate of $\E m_n(z)$ to $m(z)$ is $M/(nv)$ (see
Proposition
 \ref{prop1}), which further implies
 $$
|\E(\Gamma(z))^2|\leq M(nv)^{-2}.
 $$
We expect that this inequality could be used to solve the universality problem of the largest eigenvalue of sample covariance matrices as Johanssan does in \cite{Joh09}.

\section{Finite dimensional convergence of the processes}\label{finite-dim}

Throughout the paper, to save notation, $M$ may stand for different
constants on different occasions. This section deals with Theorem \ref{theo1}.

 Following the truncation steps in \cite{b2} we may
truncate and re-normalize the random variables so that
\begin{equation}\label{f46}
|X_{ij}|\leq \tau_nn^{1/2},\ \E X_{ij}=0,\  \E X_{ij}^2=1,
\end{equation}
where $\tau_nn^{1/3}\rightarrow\infty$ and $\tau_n\rightarrow 0$. Based on this one may then
verify that
\begin{equation}\label{f49}
\E X_{11}^4=3+O(\frac{1}{n}).
\end{equation}

Let $m_n^0(z)$ denote the one obtained from $m(z)$ with $c$ replaced by $c_n$.  For %any finite constants $l_1,\cdots,l_r$ and
$x\in [a,b]$, by Cauchy's formula, with probability one for sufficiently large $n$,
\begin{eqnarray}
nh\Big(\big(f_n(x)-\frac{1}{h}\int
K(\frac{x-y}{h})d\F_{c_n}(y)\big)\Big)=-\frac{1}{2\pi
i}\oint_{\mathcal{C}_1}
K(\frac{x-z}{h})X_n(z)dz,\label{a3*}
\end{eqnarray}
where $X_n(z)=\tr(\bbA_n-z\bbI)^{-1}-nm_n^0(z)$ and the contour $\mathcal{C}_1$ is the union of four segments
$\gamma_j,j=1,2,3,4$. Here
$$\gamma_1=u-iv_0h,u\in [a_l,a_r],\ \gamma_2=u+iv_0h,u\in [a_l,a_r],
$$$$\gamma_3=a_l+iv,v\in
[-v_0h,v_0h],\ \gamma_4=a_r+iv,v\in [-v_0h,v_0h],$$ where $a_l$ is
any positive value smaller than $a$,
$a_r$ any value larger than $b$, and
$v_0$ is a constant specified in (\ref{f63}).

For the sake of simplicity, write $\bbA=\bbA_n$. We now introduce some notation and present some basic facts frequently used in this paper..
 Define $\bbA(z)=\bbA-z\bbI$,
$\bbA_k(z)=\bbA(z)-\bbs_k\bbs_k^T$ with $n^{1/2}\bbs_k$ being the $k$th column of $\bbX_n$. Let
$\E_k=\E(\cdot|\bbs_1,\cdots,\bbs_k)$ and $\E_0$ denote the
expectation. Let $v=\Im(z)$. Set
$$
\beta_k(z)=\frac{1}{1+\bbs_k^T\bbA_k^{-1}(z)\bbs_k},\
\eta_k(z)=
\bbs_k^T\bbA_k^{-1}(z)\bbs_k-\frac{1}{n}\tr\bbA_k^{-1}(z),
$$
$$
b_1(z)=\frac{1}{1+n^{-1}\E\tr\bbA_1^{-1}(z)},\
\beta_k^{\tr}(z)=\frac{1}{1+n^{-1}\tr\bbA_k^{-1}(z)}.
$$
$$
\Gamma_{k}=n^{-1}\tr\bbA^{-1}_{k}(z)-n^{-1}\E\tr\bbA^{-1}_{k}(z),\ \Gamma_{k}^{(2)}=n^{-1}\tr\bbA^{-2}_{k}(z)-n^{-1}\E \tr\bbA^{-2}_{k}(z)
$$
and
$$
\eta_k^{(2)}(z)=
\bbs_k^T\bbA_k^{-2}(z)\bbs_k-n^{-1}\tr\bbA_k^{-2}(z).
$$
We frequently use the following equalities:
\begin{equation}
\label{b23}\bbA^{-1}(z)-\bbA_k^{-1}(z)=-\beta_k(z)\bbA_k^{-1}(z)\bbs_k\bbs_k^T\bbA_k^{-1}(z);
\end{equation}
\begin{equation}
\label{b22}\beta_k=b_1-b_1\beta_k\xi_k(z)=b_1-b_1^2\xi_k(z)+b_1^2\beta_k\xi_k^2(z)
\end{equation}
where
$\xi_k(z)=\bbs_k^T\bbA_k^{-1}(z)\bbs_k-n^{-1}\E\tr\bbA_k^{-1}(z)$.
At this moment, we would point out that the length of the vertical
lines of the contour of integral in (\ref{a3*}) converges to zero.
As a consequence, except $|b_1(z)|$ we can expect neither
$|\beta_k(z)|$ nor $|\beta_k^{\tr}(z)|$ to be bounded above by
constants independent of $v$ although they are bounded by $|z|/|v|$
(see \cite{b4}) (of course $v\neq 0$ in the cases of interest).
Instead, the absolute moments of $\beta_k(z)$ and $\beta_k^{\tr}(z)$
are proved to be bounded. We summarize such estimates in Lemma
\ref{lem1} in Appendix 1. Sometimes we deal with the terms
$\beta_k^{\tr}(z)$ and $\beta_k(z)$ in the following way: One may
verify that
$$
\Im(1+n^{-1}\tr\bbA_k^{-1}(z))\geq v
n^{-1}\tr\bbA_k^{-1}(z)\bbA_k^{-1}(\bar z),
$$
which implies that
\begin{equation}\label{f44}
|\beta^{\tr}_k(z)n^{-1}\tr\bbA_k^{-1}(z)\bbA_k^{-1}(\bar
z)|\leq M|v|^{-1}.
\end{equation}
Similarly,
\begin{equation}
\label{m17}
|\beta_k\bbs_k^T\bbA_k^{-2}(z)\bbs_k|\leq |v|^{-1}.
\end{equation}
We shall also use the simple fact that
\begin{equation}\|\bbA_k^{-1}(z)\|\leq 1/|v|.\label{h25}\end{equation}
Throughout the paper the variable $z$ sometimes will be dropped from their corresponding expressions when there is no confusion.

Here is the famous martingale decomposition in the random matrix theory,
$$
\tr\bbA^{-1}(z)-\E\tr\bbA^{-1}(z)=\sum\limits_{k=1}^n\big(\E_k\tr \bbA^{-1}(z)-\E_{k-1}\tr \bbA^{-1}(z)\big)
$$
$$=\sum\limits_{k=1}^n(\E_k-\E_{k-1})\tr\big[\bbA^{-1}(z)-\bbA_k^{-1}(z)\big]
=-\sum\limits_{k=1}^n (\E_k-\E_{k-1} )\big[\beta_k(z)\bbs_k^T\bbA_k^{-2}(z)\bbs_k\big]
$$
\begin{equation}
=-\sum\limits_{k=1}^n (\E_k-\E_{k-1} )\big[\ln\beta_k(z)\big]',\label{g39}
\end{equation}
where the third step uses (\ref{b23}) and the derivative in the last equality is with respect to $z$.
We then obtain from integration by parts that
\begin{equation}\label{b12}
\frac{1}{2\pi i}\oint
K(\frac{x-z}{h})(\tr\bbA^{-1}(z)-\E\tr\bbA^{-1}(z))dz
\end{equation}
$$
=-\frac{1}{2\pi i}\sum\limits_{k=1}^n(\E_k-\E_{k-1})\oint
K(\frac{x-z}{h})\Big[\ln\beta_k(z)\Big]'dz
$$
\begin{equation}\label{h6}
=\frac{1}{h}\frac{1}{2\pi i}\sum\limits_{k=1}^n(\E_k-\E_{k-1})\oint
K'(\frac{x-z}{h})\ln \Big(\frac{\beta^{\tr}_k(z)}{\beta_k(z)}\Big)dz
\end{equation}
$$
=\frac{1}{h}\frac{1}{2\pi i}\sum\limits_{k=1}^n(\E_k-\E_{k-1})\oint
K'(\frac{x-z}{h})\ln \big(1+\beta^{\tr}_k(z)\eta_k(z)\big)dz
$$
\begin{equation}
=\frac{1}{h}\frac{1}{2\pi i}\sum\limits_{k=1}^n(\E_k-\E_{k-1})\oint
K'(\frac{x-z}{h})\big(\beta^{\tr}_k(z)\eta_k(z)+e_k(z)\big)dz\label{f41}
\end{equation}
where the complex logarithm functions can be selected as their respective principal value branches by Cauchy's theorem and
$$
e_k(z)=\ln (1+\beta^{\tr}_k(z)\eta_k(z))-\beta^{\tr}_k(z)\eta_k(z).
$$

Below, consider $z\in\gamma_2$, the top horizontal line of the contour, unless it is further specified. We remind readers that $v=v_0h$ on $\gamma_2$. The next aim is to prove that
\begin{equation}\label{f2}
\frac{1}{h}\sum\limits_{k=1}^n(\E_k-\E_{k-1})\int
K'(\frac{x-z}{h})e_k(z)dz\stackrel{i.p.}\longrightarrow0,
\end{equation}
where $i.p.$ means ``in probability". By Lemma \ref{lem8}, we have for $m=2,4,6,8$
\begin{equation}\label{g1}
\E\big(|\eta_k(z)|^m|\bbA_k^{-1}(z)\big)\leq  M n^{-m/2}
\big[n^{-1}\tr\bbA_k^{-1}(z)\bbA_k^{-1}(\bar
z)\big]^{m/2}.
\end{equation}
This, together with Lemma \ref{lem1} in Appendix 1 and (\ref{f44}), gives
\begin{equation}\label{h1}
\E|\beta^{\tr}_k(z)\eta_k(z)|^8=\E\big(|\beta^{\tr}_k(z)|^8\E(|\eta_k(z)|^8|\bbA_k^{-1}(z))\big)\leq M(nv)^{-4}.
\end{equation}
 Via (\ref{f63}), (\ref{h1}) and the inequality
  \begin{equation}|\label{h9}
  \ln (1+x)-x|\leq
M|x|^2, \ \mbox{for}\ |x|\leq 1/2,
\end{equation}  we obtain
$$
h^{-2}\E\Big|\sum\limits_{k=1}^n(\E_k-\E_{k-1})\int
K'(\frac{x-z}{h})e_k(z)I(|\beta^{\tr}_k(z)\eta_k(z)|< 1/2)du\Big|^2
$$
\begin{equation}\label{g2}
\leq Mh^{-2}\sum\limits_{k=1}^n\E\Big|\int
K'(\frac{x-z}{h})e_k(z)I(|\beta^{\tr}_k(z)\eta_k(z)|< 1/2)du\Big|^2
\end{equation}
$$
\leq Mh^{-2}\sum\limits_{k=1}^n\Big[\int\int
|K'(\frac{x-z_1}{h})K'(\frac{x-z_2}{h})|\Big(\E|(\beta^{\tr}_k(z_1)\eta_k(z_1))|^4
$$$$\times \E|(\beta^{\tr}_k(z_2)\eta_k(z_2))|^4\Big)^{1/2}du_1du_2\Big]
\leq M/(nv^{2}).
$$
Note that
$$\ln (1+\beta^{\tr}_k(z)\eta_k(z))=\ln \beta_k(z)- \ln\beta^{\tr}_k(z).
$$
Moreover $|\beta_k(z)|\leq |z|/v$ and
$$
|\beta_k(z)|\geq (1+v^{-1}\bbs_k^T\bbs_k)^{-1}\geq (1+v^{-1}n\tau_n)^{-1}.
$$
It follows that
$$|\ln \beta_k(z)|\leq M\max \big(\ln v^{-1},\ \ln (n/v)\big).
$$ Likewise $|\ln\beta^{tr}_k(z)|\leq M\ln v^{-1}$. Hence
$$
|\ln (1+\beta^{\tr}_k(z)\eta_k(z))|\leq M\max \big(\ln v^{-1},\ \ln (n\tau_n/v)\big).
$$
 This, together with (\ref{h1}), ensures that
$$
\frac{1}{h}\E\Big|\sum\limits_{k=1}^n(\E_k-\E_{k-1})\int
K'(\frac{x-z}{h})e_k(z)I(|\beta^{\tr}_k(z)\eta_k(z)|\geq 1/2)du\Big|
\leq M/(nv^2).
$$
Thus, (\ref{f2}) is proven. Similarly, by (\ref{g1}) and Lemma \ref{lem1} we have
$$
\E\Big|\sum\limits_{k=1}^n\E_k\Big[\frac{1}{h}\int
K'(\frac{x-z}{h})\Big((\beta^{\tr}_k(z)-b_1(z))\eta_k(z)\Big)dz\Big]\Big|^2\leq M/(nv^2).
$$
 Therefore on $\gamma_2$
\begin{equation}\label{g41}
(\ref{b12})=\frac{1}{2\pi i}\sum\limits_{k=1}^nY_k(x)+o_p(1),
\end{equation}
where
$$
Y_k(x)=b_1(z)\E_k\Big[\frac{1}{h}\int
K'(\frac{x-z}{h})\eta_k(z)dz\Big].
$$

Apparently, $Y_k(x)$ is a martingale difference so that we may
resort to the CLT for martingales (see Theorem 35.12 in \cite{bili}).
Here we consider only one point $x$. But from the late proof, one can see that we actually prove the finite dimensional convergence.

As in (\ref{g2}), by (\ref{f63}) and (\ref{h1})
we have
$$
\sum\limits_{k=1}^n\E|Y_k(x)|^4\leq M/(nv^2).
$$
which ensures that the Lyapunov condition in the CLT is
satisfied.

Thus, it is sufficient to investigate the limit of the following
covariance function
\begin{eqnarray}
&&-\frac{1}{4\pi^2}\sum\limits_{k=1}^n\E_{k-1}[Y_k(x_1)Y_k(x_2)] \non
&=&-\frac{1}{4h^2\pi^2}\int\int K'(\frac{x_1-z_1}{h})
K'(\frac{x_2-z_2}{h})\mathcal{C}_{n1}(z_1,z_2)dz_1dz_2,\label{f48}
\end{eqnarray}
where
$$
\mathcal{C}_{n1}(z_1,z_2)=b_1(z_1)b_1(z_2)\sum\limits_{k=1}^n\E_{k-1}\big[\E_k\big(\eta_k(z_1)\big)\E_k\big(\eta_k(z_2)\big)\big].
$$

Note that for any non-random matrices $\bbB$ and $\bbC$
\begin{eqnarray} \label{i1}
&\E(\bbs_1^T\bbC \bbs_1-\tr\bbC)(\bbs_1^T\bbB\bbs_1-\tr\bbB) \\ & =
n^{-2}(\E X_{11}^4-|\E X_{11}^2|^2-2)\sum\limits_{i=1}^p(\bbC)_{ii}(\bbB)_{ii}+|\E X_{11}^2|^2n^{-2}\tr\bbC\bbB^T+n^{-2}\tr\bbC\bbB\nonumber.
\end{eqnarray}
This implies that
\begin{eqnarray}
&&b_1(z_1)b_1(z_2)\sum\limits_{k=1}^n\E_{k-1}\big(\E_k\eta_k(z_1)\E_k\eta_k(z_2)\big)
\non
&=&(\E X_{11}^4-3)b_1(z_1)b_1(z_2)\mathcal{C}_{n1}^{(1)}(z_1,z_2)+2b_1(z_1)b_1(z_2)\mathcal{C}_{n2}(z_1,z_2)\label{g7}\\
&=&2b_1(z_1)b_1(z_2)\mathcal{C}_{n2}(z_1,z_2)+O\big(1/(nv^2)\big),\label{g6}
\end{eqnarray}
where
\begin{eqnarray*}
\mathcal{C}_{n1}^{(1)}(z_1,z_2)&=&\frac{1}{n^2}\sum\limits_{k=1}^n\sum\limits_{j=1}^p \E_k(\bbA_k^{-1}(z_1))_{jj}\E_k(\bbA_k^{-1}(z_2))_{jj},\\
\mathcal{C}_{n2}(z_1,z_2)&=&\frac{1}{n^2}\sum\limits_{k=1}^n\tr\E_k(\bbA_k^{-1}(z_1))
\E_k(\bbA_k^{-1}(z_2))\\
&=&\frac{1}{n^2}\sum\limits_{k=1}^n\tr\E_k\big(\bbA_k^{-1}(z_1)
\underline{\bbA}_k^{-1}(z_2)\big)
\end{eqnarray*}
and the last step uses (\ref{f49}) and (\ref{h25}). Here
$\underline{\bbA}_{k}^{-1}(z)$ is defined by $\bbs_1,\cdots,\bbs_{k-1},$ $\underline{\bbs}_{k+1},\cdots,\underline{\bbs}_n$
as $\bbA_{k}^{-1}(z)$ is defined by
$\bbs_1,\cdots,\bbs_{k-1},\bbs_{k+1},\cdots,\bbs_n$ with $\underline{\bbs}_{1},\cdots,\underline{\bbs}_{n}$ being i.i.d.
copies of $\bbs_1$ and independent of $\{\bbs_j,j=1,\cdots,n\}$.

\subsection{The limit of $\mathcal{C}_{n2}(z_1,z_2)$ }

The next aim is to develop the limit of
$\mathcal{C}_{n2}(z_1,z_2)$. To
this end, we introduce more notation and estimates. Let
$$\bbA_{kj}(z)=\bbA(z)-\bbs_k\bbs_k^T-\bbs_j\bbs_j^T,
\beta_{kj}(z)=\frac{1}{1+\bbs_j^T\bbA^{-1}_{kj}(z)\bbs_j},
$$
 $$
b_{12}(z)=\frac{1}{1+n^{-1}\E\tr\bbA^{-1}_{12}(z)}, \quad
\beta_{kj}^{tr}(z)=\frac{1}{1+n^{-1}\tr\bbA^{-1}_{kj}(z)},
$$
$$
\Gamma_{kj}=n^{-1}\tr\bbA^{-1}_{kj}(z)-\E n^{-1}\tr\bbA^{-1}_{kj}(z),\ \Gamma_{kj}^{(2)}=n^{-1}\tr\bbA^{-2}_{kj}(z)-\E n^{-1}\tr\bbA^{-2}_{kj}(z)
$$
and
$$
\xi_{kj}(z)=\bbs_j^T\bbA^{-1}_{kj}(z)\bbs_j-\E n^{-1}\tr\bbA^{-1}_{kj}(z),\quad
\eta_{kj}(z)=\bbs_j^T\bbA^{-1}_{kj}(z)\bbs_j-n^{-1}\tr\bbA^{-1}_{kj}(z).
$$
Actually, they are similar to $\bbA_k(z), \beta_k(z), \cdots$. Note that
\begin{equation}\label{f6}
\bbA_{k}^{-1}(z)-\bbA_{kj}^{-1}(z)=-\beta_{kj}(z)\bbA_{kj}^{-1}(z)\bbs_j\bbs_j^T\bbA_{kj}^{-1}(z)
\end{equation}
which is similar to \eqref{b23},  and (see Lemma 2.10 of \cite{b4}) for any $p\times p$ matrix $\bbD$
\begin{equation}\label{f10}
|\tr(\bbA_k^{-1}(z)-\bbA_{kj}^{-1}(z))\bbD|\leq \|\bbD\|v^{-1}.
\end{equation}
Also, we have
\begin{equation}\label{f7}
|\beta_{kj}|\|\bbs_j^T\bbA_{kj}^{-1}(z)\|^2=|\beta_{kj}\bbs_j^T\bbA_{kj}^{-1}(z)\bbA_{kj}^{-1}(\bar
z)\bbs_j|\leq v^{-1}.
\end{equation}
Write
\begin{equation}\label{f13}
\beta_{kj}(z)=b_{12}(z)-\beta_{kj}(z)b_{12}(z)\xi_{kj}(z)=b_{12}(z)-b_{12}^2(z)\xi_{kj}(z)+\beta_{kj}(z)b_{12}^2(z)\xi^2_{kj}(z).
\end{equation}
By Lemma \ref{lem1} we have
\begin{equation}\label{f14}
\E n^{-1}\tr\bbA^{-1}(z)\bbA^{-1}(\bar
z)=v^{-1}\Im(\E n^{-1}\tr\bbA^{-1}(z))\leq Mv^{-1},
\end{equation}
which, together with (\ref{f10}), implies that
\begin{equation}\label{f15}
\E n^{-1}\tr\bbA_{kj}^{-1}(z)\bbA_{kj}^{-1}(\bar
z)\leq Mv^{-1}.
\end{equation}
 By Lemma \ref{lem1} in Appendix 1 and (\ref{f10}) we then have
\begin{equation}\label{h12}
\E| n^{-1}\tr\bbA_{kj}^{-1}(z)|^4\leq M.
\end{equation}
By (\ref{f10}) and the fact that $b_1(z)$ is bounded, given in Lemma
\ref{lem1}, it is straightforward to verify that
$|b_1(z)-b_{12}(z)|\leq (nv^2)^{-1}$ and hence
\begin{equation}\label{f3*}
|b_{12}(z)|\leq M.
\end{equation}

In the following, we will use $\E^j$ to denote the conditional
expectation given $\bbs_1, \bbs_2, \cdots$ except $\bbs_j$. It is
indeed the expectation taken with respect to $\bbs_j$. And write
\begin{equation}\label{h24}
\centre^j (x)=x-\E^j(x),
\end{equation}
where $x$ is some random variable. We claim that
\begin{equation}\label{h13}
\E\Big|n^{-1}\sum\limits_{j>k}\centre^j \big(\bbs_j^T\bbA^{-1}_{kj}(z_1)\underline{\bbA}_k^{-1}(z_2)\bbs_j\big)\Big|^2=O\big(1/(nv^2)\big),
\end{equation}
\begin{equation}\label{h13*}
\E\Big|n^{-1}\sum\limits_{j<k}\centre^j \big(\bbs_j^T
\bbA^{-1}_{kj}(z_1)\underline{\bbA}_{kj}^{-1}(z_2)\bbs_j\big)\Big|^2=O\big(1/(nv^2)\big)
\end{equation}
and
\begin{equation}\label{h13**}
\E\Big|n^{-1}\sum\limits_{j<k}\centre^j\big(\bbs_j^T
\bbA^{-1}_{kj}(z_1)\bbs_j\big)\Big|^2=O\big(1/(n^2v^2)\big).
\end{equation}
Consider (\ref{h13}) first.  Apparently by Lemma \ref{lem1} we have
\begin{equation}\label{h14}
n^{-2}\sum\limits_{j>k}\E\Big|\centre^j\big(\bbs_j^T
\bbA^{-1}_{kj}(z_1)\underline{\bbA}_k^{-1}(z_2)\bbs_j\big)\Big|^2=O\big(1/(n^2v^3)\big).
\end{equation}
Second, we also obtain
$$
n^{-2}\sum\limits_{j_1\neq j_2>k}\E\Big[\centre^{j_1}\big(\bbs_{j_1}^T
\bbA^{-1}_{kj_1}(z_1)\underline{\bbA}_k^{-1}(z_2)\bbs_{j_1}\big)
$$\begin{equation}\label{h15}\times\centre^{j_2}\big(\bbs_{j_2}^T
\bbA^{-1}_{kj_2}(\bar z_1)\underline{\bbA}_k^{-1}(\bar z_2)\bbs_{j_2}\big)\Big]=O\big(1/(n^2v^4)\big),
\end{equation}
which was ensured by the following estimates:
\begin{equation}\label{h16}
\E\big[\centre^{j_1}\big(\bbs_{j_1}^T
\bbA^{-1}_{kj_1}\underline{\bbA}_k^{-1}\bbs_{j_1}\big)\times
\centre^{j_2}\big(\bbs_{j_2}^T
\bbA^{-1}_{kj_1j_2}\underline{\bbA}_k^{-1}\bbs_{j_2}\big)\big]=0;
\end{equation}
(remember the convention that $z$ and $\bar z$ are dropped from the
corresponding expressions) and via (\ref{f7}), H\"older's inequality
$$
\E\big| \centre^{j_1}\big(\bbs_{j_1}^T
\bbA(k,j_2,j_1)\underline{\bbA}_k^{-1}(z_2)\bbs_{j_1}\big)
\times \centre^{j_2}\big(\bbs_{j_2}^T
\bbA(k,j_1,j_2)\underline{\bbA}_k^{-1}(\bar z_2)\bbs_{j_2}\big)\big|\leq M/(n^2v^4),
$$
where
$$
\bbA(k,j_2,j_1)=\bbA^{-1}_{kj_1j_2}(z_1)\bbs_{j_2}\bbs_{j_2}^T\bbA^{-1}_{kj_1j_2}(z_1)\beta_{kj_2j_1}(z),
$$
$$
\bbA^{-1}_{kj_1j_2}(z)=(\bbA-\bbs_k\bbs_k^T-\bbs_{j_1}\bbs_{j_1}^T-\bbs_{j_2}\bbs_{j_2}^T-z\bbI)^{-1}, \ \beta_{kj_2j_1}(z)=\big(1+\bbs_{j_2}^T\bbA^{-1}_{kj_1j_2}(z)\bbs_{j_2}\big)^{-1},
$$
and $\bbA(k,j_1,j_2)$ and $\beta_{kj_1j_2}(z)$ are defined similarly. Thus (\ref{h13}) is true, as claimed.

Consider (\ref{h13*}) next. Note that (\ref{h14}), (\ref{h16}) are still true if $\underline{\bbA}_k^{-1}(z_2)$ is replaced by $\underline{\bbA}_{kj_1j_2}^{-1}(z_2)$. Moreover, by (\ref{f7}), Lemma \ref{lem8}, H\"older's inequality we obtain
$$
\E\Big|\centre^{j_1}\big(\bbs_{j_1}^T
\bbA(k,j_2,j_1)(z_1)\underline{\bbA}(k,j_2,j_1)(z_2)\bbs_{j_1}\big)
$$$$\times \centre^{j_2}\big(\bbs_{j_2}^T
\bbA(k,j_1,j_2)(\bar z_1)\underline{\bbA}(k,j_2,j_1)(\bar z_2)\bbs_{j_2}\big)\Big|
$$
$$
\leq \frac{M}{n^2}\Big[\E\Big|\|\bbs_{j_2}^T\bbA^{-1}_{kj_1j_2}(z)\|^4\|\bbs_{j_2}^T\underline{\bbA}^{-1}_{kj_1j_2}(z)\|^4|\beta^2_{kj_2j_1}(z)\underline{\beta}^2_{kj_2j_1}(z)|\Big|
$$$$\times \E\Big|\|\bbs_{j_1}^T\bbA^{-1}_{kj_1j_2}(z)\|^4\|\bbs_{j_1}^T\underline{\bbA}^{-1}_{kj_2j_1}(z)\|^4|\beta^2_{kj_1j_2}(z)\underline{\beta}^2_{kj_1j_2}(z)|\Big|\Big]^{1/2}\leq M/(n^2v^4),
$$
where $\underline{\bbA}^{-1}_{kj_1j_2}(z)$ is obtained from  $\bbA^{-1}_{kj_1j_2}(z)$ with $\bbs_{k+1},\cdots,\bbs_{n}$ being replaced by $\underline{\bbs}_{k+1},\cdots,\underline{\bbs}_{n}$ and the remaining $\bbs_1,\cdots,\bbs_{k-1}$ unchanged, $\underline{\beta}_{kj_1j_2}(z)$ is obtained from $\beta_{kj_1j_2}(z)$ with $\bbA^{-1}_{kj_2j_1}(z)$ replaced by $\underline{\bbA}^{-1}_{kj_2j_1}(z)$ and $\underline{\bbA}(k,j_2,j_1)(z)$ from $\bbA(k,j_2,j_1)(z)$ with $\bbA^{-1}_{kj_2j_1}(z)$ and $\beta_{kj_1j_2}(z)$, respectively, replaced by $\underline{\bbA}^{-1}_{kj_2j_1}(z)$ and $\underline{\beta}_{kj_1j_2}(z)$. Here $\underline{\bbA}(k,j_1,j_2)$ and $\underline{\beta}_{kj_2j_1}(z)$ can be similarly defined. These estimates imply (\ref{h13*}).

Replacing $\underline{\bbA}_k^{-1}(z_2)$ by the identity matrix from (\ref{h14})-(\ref{h16}) yields (\ref{h13**}).

Let $u_n(z)=\big(z-(1-n^{-1})b_{12}(z)\big)^{-1}$.  We now state the equality (2.9) in
\cite{b2}
\begin{equation}\label{f21}
\bbA_k^{-1}(z)=-u_n(z)\bbI+b_{12}(z)B(z)+C(z)+D(z),
\end{equation}
where
$$
B(z)=\sum\limits_{j\neq
k}u_n(z)(\bbs_j\bbs_j^T-n^{-1}\bbI)\bbA^{-1}_{kj}(z),
$$
$$
C(z)=\sum\limits_{j\neq
k}(\beta_{kj}(z)-b_{12}(z))u_n(z)\bbs_j\bbs_j^T\bbA^{-1}_{kj}(z)
$$
and
$$
D(z)=n^{-1}b_{12}(z)u_n(z)\sum\limits_{j\neq
k}(\bbA^{-1}_{kj}(z)-\bbA^{-1}_{k}(z)).
$$

Applying the definition of $C(z_1)$ and (\ref{f13})
gives
\begin{equation}\label{h5}
n^{-1}\E_k\big[\tr C(z_1)
\underline{\bbA}_k^{-1}(z_2)\big]=C_1(z_1)+C_2(z_1),\end{equation}
 where
$$
C_1(z_1)=-b_{12}^2(z_1)n^{-1}\sum\limits_{j\neq
k}\E_k\big[\xi_{kj}(z_1)\bbs_j^T\hat{\bbA}^{-1}_{kjk}(z_1,z_2)\bbs_j\big]
$$
and
$$
C_2(z_1)=b_{12}^2(z_1)n^{-1}\sum\limits_{j\neq
k}\E_k\big[\beta_{kj}(z_1)\xi^2_{kj}(z)\bbs_j^T\hat{\bbA}^{-1}_{kjk}(z_1,z_2)\bbs_j\big].
$$
Here
\begin{equation}\label{h4}
 \hat{\bbA}^{-1}_{kjk}(z_1,z_2)=\bbA^{-1}_{kj}(z_1)
\underline{\bbA}_k^{-1}(z_2)u_n(z_1).
\end{equation}
 Define
 $
\zeta_{kj3}=\centre^j \big(\bbs_j^T\hat{\bbA}^{-1}_{kjk}(z_1,z_2)\bbs_j\big).
 $

We claim that the contribution from $C(z_1)$ is negligible. To
verify it we distinguish two cases: $j>k$ and $j<k$. Consider $j>k$
first. From (\ref{h12}) and an estimate similar to (\ref{h12}) we
have
\begin{eqnarray}
&\E|n^{-1}\tr\bbA_{kj}^{-1}(z_1)
\underline{\bbA}_{k}^{-1}(z_2)|^4\leq \non & M\big(\E|n^{-1}\tr\bbA_{kj}^{-1}(z_1)\bbA_{kj}^{-1}(\bar
z_1)|^4\E|n^{-1}\tr\underline{\bbA}_{k}^{-1}(z_2)\underline{\bbA}_{k}^{-1}(\bar
z_2)|^4\big)^{1/2}\leq M/v^{4}.\label{g8}
\end{eqnarray}
where we use (\ref{f10}) as well.
 It follows from Lemma \ref{lem1} and (\ref{g8}) that
$$
\E|C_2(z_1)|\leq Mn^{-1}\sum\limits_{j\neq
k}(\E|\xi_{kj}(z_1)|^4)^{1/2}\big[\E|\beta_{kj}(z_1)|^4
$$$$\qquad\qquad\times\big(\E|\zeta_{kj3}|^4+\E|n^{-1}\tr\hat{\bbA}^{-1}_{kjk}(z_1,z_2)|^4\big)\big]^{1/4}\leq M(nv^{2})^{-1}.
$$
As for $C_1(z_1)$, write
$$
\E_k\big[\xi_{kj}(z_1)\bbs_j^T\hat{\bbA}^{-1}_{kjk}(z_1,z_2)\bbs_j\big]=\E_k\big(\eta_{kj}(z_1)\zeta_{kj3}\big)
$$
\begin{equation}+\E_k\big(n^{-1}\Gamma_{kj}(\tr\hat{\bbA}^{-1}_{kjk}(z_1,z_2)-
\E\tr\hat{\bbA}^{-1}_{kjk}(z_1,z_2))\big)
+\E_k(\Gamma_{kj})n^{-1}\E\tr(\hat{\bbA}^{-1}_{kjk}(z_1,z_2)).\label{m48}
\end{equation}
We conclude from (\ref{f10}), Lemmas \ref{lem7} and \ref{lem1} that
the absolute moments of the first two terms above on the right hand
have an order of $1/(nv^2)$. As for the last term, it was proved in
Proposition 6.1 of \cite{g2} that
\begin{equation}\label{h28}
\E|n^{-1}\tr\bbA^{-1}(z)-\E n^{-1}\tr\bbA^{-1}(z)|^2\leq\frac{M}{n^2v^2|z+c_n-1+2c_nzm_n(z)|^2}
\end{equation}
In view of (\ref{m2*}), we have
\begin{equation}\label{g18}
z+c_n-1+2c_nzm_n^0(z)=\sqrt{(a_n-z)(b_n-z)},
\end{equation}
where $m_n^0(z)$ is obtained from $m(z)$ with $c$ replaced by $c_n$,
$a_n=(1-\sqrt{c_n})^2$ and $b_n=(1+\sqrt{c_n})^2$. From (\ref{a25*})
we claim that
\begin{equation}\label{h19}
\frac{1}{h}\int^{a_r}_{a_l}
\frac{|K\big((x-z)/h\big)|}{|z+c_n-1+2c_nzm_n^0(z)|}du\leq M
\frac{1}{h}\int^{a_r}_{a_l}\frac{|K\big((x-z)/h\big)|}{\sqrt{|(u-a_n)(b_n-u)|}}du\leq
M.
\end{equation}
Indeed, by a change of variables and dividing the integration region
into $|q|\leq\delta$ and $|q|>\delta$ with $0<\delta<\min
\{\frac{b_n-x}{2},\frac{x-a_n}{2}\}$, we have
$$
\frac{1}{h}\int^{a_r}_{a_l}\frac{|K\big((x-z)/h\big)|}{\sqrt{|(a_n-u)(b_n-u)|}}du=
\frac{1}{h}\int^{x-a_l}_{x-a_r}I(|q|\leq\delta)\frac{|K\big(q/h+iv_0\big)|}{\sqrt{|(x-q-a_n)(b_n-(x-q))|}}dq
$$
$$
+\frac{1}{h}\int^{x-a_l}_{x-a_r}I(|q|>\delta)\frac{|K\big(q/h+iv_0\big)|}{\sqrt{|(x-q-a_n)(b_n-(x-q))|}}dq
$$
$$
\leq
M_x\frac{1}{h}\int^{x-a_l}_{x-a_r}|K\big(q/h+iv_0\big)|dq+\sup\limits_{|q|>\delta}|\frac{q}{h}K(\frac{d}{h}+iv_0)|
\frac{1}{\delta}\int^{a_r}_{a_l}\frac{1}{\sqrt{|(u-a_n)(b_n-u)|}}du\leq
M,
$$
where $M_x$ denotes some positive constant which depends on $x$. It
follows from (\ref{h28}), (\ref{h19}), (\ref{f10}) and an inequality
similar to (\ref{f10}) that
\begin{equation}\label{h26}
\E\Big|\frac{1}{h}\int^{a_r}_{a_l}
K(\frac{x-z_1}{h})\E_k\big(\Gamma_{kj}
n^{-1}\E\tr(\hat{\bbA}^{-1}_{kjk}(z_1,z_2))\big)du_1\Big| \leq
M/(nv^2),
\end{equation}
where we also use the fact that
$|n^{-1}\E\tr(\hat{\bbA}^{-1}_{kjk}(z_1,z_2))|\leq M/v$.

For handling the case $j<k$, we define
$\underline{\bbA}_{kj}^{-1}(z),\underline{\beta}_{kj}(z)$ and
$\underline{\xi}_{kj}(z)$ by
$$\bbs_1,\cdots,\bbs_{j-1},\bbs_{j+1},\cdots,\bbs_{k-1},\underline{\bbs}_{k+1},\cdots,\underline{\bbs}_n$$
as $\bbA_{kj}^{-1}(z),\beta_{kj}(z)$ and $\xi_{kj}(z)$ are defined
by
$$\bbs_1,\cdots,\bbs_{j-1},\bbs_{j+1},\cdots,\bbs_{k-1},\bbs_{k+1},\cdots,\bbs_n.$$
When $j<k$, similar to (\ref{f6}), we then decompose
$\underline{\bbA}_{k}^{-1}(z_2)$ as
\begin{equation}\underline{\bbA}_{kj}^{-1}(z_2)-\underline{\bbA}_{kj}^{-1}(z_2)\bbs_j\bbs_j^T\underline{\bbA}_{k}^{-1}(z_2)\underline{\beta}_{kj}.\label{m49}
\end{equation}

Note that $\bbs_j$ is independent of
$\underline{\bbA}_{kj}^{-1}(z_2)$. Apparently, the preceding
argument for the case $j>k$ also works if we replace
$\underline{\bbA}_k^{-1}(z_2)$ in $\hat{\bbA}^{-1}_{kjk}(z_1,z_2)$
with $\underline{\bbA}_{kj}^{-1}(z_2)$, the first term of
(\ref{m49}), because the preceding argument used the independence
between $\underline{\bbA}_k^{-1}(z_2)$ and $\bbs_j$ when $j<k$. For
another term of $C_2(z_1)$ due to the second term of (\ref{m49}), by
(\ref{f7}), (\ref{h12}) and Lemma \ref{lem1}
$$
\E\Big|\beta_{kj}(z_1)\underline{\beta}_{kj}(z_2)\xi^2_{kj}(z)\bbs_j^T\bbA_{kj}^{-1}(z_1)\underline{\bbA}_{kj}^{-1}(z_2)\bbs_j\bbs_j^T
\underline{\bbA}_{kj}^{-1}(z_2)u_n(z_1)\bbs_j\Big|
$$
$$
\leq M
v^{-1}\big(\E|\beta_{kj}(z_1)|^2\E|\underline{\beta}_{kj}(z_2)|^2\E|\xi_{kj}|^8\E|\bbs_j^T\underline{\bbA}_{kj}^{-1}(z_2)\bbs_ju_n(z_1)|^4\big)^{1/4}
\leq M/(nv^{2}).$$ As for another term of $C_1(z_1)$, it follows
from Holder's inequality, Lemmas \ref{lem1}, \ref{lem7},
(\ref{h13**}), (\ref{f10}), (\ref{f6}) and (\ref{b23}) that yields
$$
n^{-1}\sum\limits_{j<
k}\E_k\big[\underline{\beta}_{kj}(z_2)\xi_{kj}(z_1)\bbs_j^T\bbA^{-1}_{kj}(z_1)
\underline{\bbA}_{kj}^{-1}(z_2)\bbs_j\bbs_j^T\underline{\bbA}_{kj}^{-1}(z_2)\bbs_ju_n(z_1)\big]
$$\begin{equation}\label{h38} =n^{-1}\sum\limits_{j<
k}\E_k\big[\underline{\beta}_{kj}(z_2)\xi_{kj}(z_1)\bbs_j^T\bbA^{-1}_{kj}(z_1)
\underline{\bbA}_{kj}^{-1}(z_2)\bbs_jn^{-1}\tr\underline{\bbA}_{kj}^{-1}(z_2)u_n(z_1)\big]+A_1
\end{equation}$$ =n^{-1}\E\big(n^{-1}\tr\bbA^{-1}_{kj}(z_1)
\underline{\bbA}_{kj}^{-1}(z_2)\big)\sum\limits_{j<
k}\E_k\big[\underline{\beta}_{kj}(z_2)\xi_{kj}(z_1)n^{-1}\tr\underline{\bbA}_{kj}^{-1}(z_2)u_n(z_1)\big]+A_2
$$$$
=n^{-1}u_n(z_1)b_{12}(z_2)\E n^{-1}\tr\underline{\bbA}_{kj}^{-1}(z_2)\E\big(n^{-1}\tr\bbA^{-1}_{kj}(z_1)
\underline{\bbA}_{kj}^{-1}(z_2)\big)
$$$$\times\sum\limits_{j<
k}\E_k\Big[\eta_{kj}(z_1)+\big(n^{-1}\tr\bbA^{-1}(z_1)-\E
n^{-1}\tr\bbA^{-1}(z_1)\big)\Big]+A_3
$$
$$
=u_n(z_1)b_{12}(z_2)\E
n^{-1}\tr\underline{\bbA}_{kj}^{-1}(z_2)\E\big(n^{-1}\tr\bbA^{-1}_{kj}(z_1)
\underline{\bbA}_{kj}^{-1}(z_2)\big)$$$$\times
(1-k/n)E_k\big(n^{-1}\tr\bbA^{-1}(z_1)-\E
n^{-1}\tr\bbA^{-1}(z_1)\big)+A_4,
$$
where each $A_j$ satisfies $E|A_j|\leq M/(nv^2)$, $j=1,2,3,4$ and
the last term can be handled as in (\ref{h26}).
 Summarizing the above we have proved that
\begin{equation}\label{h27}
\E\Big|\frac{1}{h}\int^{a_r}_{a_l} K(\frac{x-z_1}{h}) n^{-1}\tr\E_k\Big[C(z_1)
\underline{\bbA_k}^{-1}(z_2)\Big]du_1\Big|\leq\frac{M}{nv^2}.
\end{equation}

Consider $D(z_1)$ now. When $j>k$ using (\ref{f13}) and recalling the definition of $ \hat{\bbA}^{-1}_{kjk}(z_1,z_2)$ in (\ref{h4}) we obtain
$$
n^{-1}\E_k\Big[\tr D(z_1)\underline{\bbA_k}^{-1}(z_2)\Big]
=n^{-2}b_{12}(z_1)\sum\limits_{j\neq k}[D_1+D_2]
$$
where
$$
D_1=-n^{-1}\E_k\big[\tr\hat{\bbA}^{-1}_{kjk}(z_1,z_2)\bbA^{-1}_{kj}(z_1)\beta_{kj}(z_1)\big]
$$
and
$$
D_2=\E_k\big[\beta_{kj}(z_1)\centre^{j}\big(\bbs_j^T \hat{\bbA}^{-1}_{kjk}(z_1,z_2)\bbA^{-1}_{kj}(z_1)\bbs_j\big)\big].
$$
By Lemmas \ref{lem8}, \ref{lem1}, Holder's inequality and
(\ref{h12}) we have $\E|D_1|\leq M/v^2$ and $\E|D_2|\leq v^{-3/2}$.
These imply that for $j>k$
\begin{equation}
\E| n^{-1}\tr D(z_1)\underline{\bbA_k}^{-1}(z_2)|\leq
M/(nv^2).\label{h3}
\end{equation}
When $j<k$, we resort to (\ref{m49}), the decomposition of
$\underline{\bbA}_{k}^{-1}(z_2)$. As before, the above argument for
the case $j>k$ also works for the term involving
$\underline{\bbA}_{kj}^{-1}(z_2)$ if we replace
$\underline{\bbA}_k^{-1}(z_2)$ with
$\underline{\bbA}_{kj}^{-1}(z_2)$. Another term is
$$
\frac{b_{12}(z_1)}{n^2}\sum\limits_{j\neq
k}\Big[\beta_{kj}(z_1)\underline{\beta}_{kj}(z_2)\bbs_j^T\bbA_{kj}^{-1}(z_1)\underline{\bbA}_{kj}^{-1}(z_2)\bbs_j
\bbs_j^T\underline{\bbA}_{kj}^{-1}(z_2)\bbA_{kj}^{-1}(z_1)\bbs_ju_n(z_1)\Big],
$$
which has, via (\ref{f7}), an order of
$(nv^{2})^{-1}$. Thus, the contribution from $C(z_1)$ and $D(z_1)$ is negligible.

Next consider $B(z_1)$. In view of (\ref{h13}) and (\ref{h13*}) we may write
$$
n^{-1}\tr\E_k\big[ B(z_1)\underline{\bbA}_k^{-1}(z_2)\big]
=B_1(z_1)+B_2(z_1)+A_5,
$$
where
$$
B_1(z_1)=-n^{-1}\sum\limits_{j<k}
\E_k\big[\underline{\beta}_{kj}(z_2)\bbs_j^T\bbA_{kj}^{-1}(z_1)\underline{\bbA}_{kj}^{-1}(z_2)\bbs_j\bbs_j^T
\underline{\bbA}_{kj}^{-1}(z_2)\bbs_ju_n(z_1)\big],
$$
$$
B_2(z_1)=n^{-2}\sum\limits_{j<k}\E_k\big[\underline{\beta}_{kj}(z_2)\bbs_j^T
\underline{\bbA}_{kj}^{-1}(z_2)\bbA_{kj}^{-1}(z_1)\underline{\bbA}_{kj}^{-1}(z_2)\bbs_ju_n(z_1)\big]
$$
and $E|A_5|\leq \frac{M}{nv^2}$. By (\ref{f7}) and Lemma \ref{lem1}
we have $|B_2(z_1)|\leq M/(nv^2)$. With notation
$\hat{\eta}_{kj}=\centre^j\big(\bbs_j^T\bbA_{kj}^{-1}(z_1)$
$\underline{\bbA}_{kj}^{-1}(z_2)\bbs_j\big)$, from Lemma \ref{lem1}
we obtain
$$
\E\big|(\hat{\eta}_{kj})\big(n^{-1}\tr\bbA_{kj}^{-1}(z_1)-\E n^{-1}\tr\bbA_{kj}^{-1}(z_1)\big)\big|
\leq M/(nv^2),
$$
which, together with (\ref{h13*}), implies that
$$
\E\big|n^{-1}\sum\limits_{j<k}\big(\hat{\eta}_{kj}\big)
n^{-1}\tr\bbA_{kj}^{-1}(z_1)\big|=O(n^{-1}v^{-2}).
$$
Moreover, in view of Lemma \ref{lem7} and (\ref{h13**}) we have
$$
\E\big|n^{-2}\sum\limits_{j<k}\tr\bbA_{kj}^{-1}(z_1)\underline{\bbA}_{kj}^{-1}(z_2)\eta_{kj}\big|\leq M/(nv^2).
$$
Apparently by Lemma \ref{lem1} we also have
$$
\E\big|\hat{\eta}_{kj}\eta_{kj}\big|
\leq M/(nv^2).
$$
We then conclude from Lemmas \ref{lem8}, \ref{lem1} and (\ref{f13})
that
\begin{equation}\label{h37}
\big|B_1(z_1)+n^{-3}b_{12}(z_2)u_n(z_1)\sum\limits_{j<k}\E_k\big[\tr\bbA_{kj}^{-1}(z_1)\underline{\bbA}_{kj}^{-1}(z_2)
\tr\underline{\bbA}_{kj}^{-1}(z_2)\big]\big| =A_6,
\end{equation}
where $E|A_6|\leq\frac{M}{nv^2}$.

Furthermore by (\ref{f10}) we obtain
$$
n^{-3}\sum\limits_{j<k}\E_k\Big[\tr\bbA_{kj}^{-1}(z_1)\underline{\bbA}_{kj}^{-1}(z_2)\tr\underline{\bbA}_{kj}^{-1}(z_2)\Big]
$$$$=
\frac{k-1}{n^3}\E_k\Big[\tr\bbA_{k}^{-1}(z_1)\underline{\bbA}_{k}^{-1}(z_2)\tr\underline{\bbA}_{k}^{-1}(z_2)\Big]+O(\frac{1}{nv^2}).
$$
It follows from Lemma \ref{lem1}, (\ref{h28}) and (\ref{h19}) that
$$
\E_k\Big[\frac{k-1}{n^3}\tr\bbA_{k}^{-1}(z_1)\underline{\bbA}_{k}^{-1}(z_2)\tr\underline{\bbA}_{k}^{-1}(z_2)\Big]
$$$$=\frac{1}{n}\E\tr\underline{\bbA}_{k}^{-1}(z_2)\E_k\Big[\frac{k-1}{n^2}\tr\bbA_{k}^{-1}(z_1)\underline{\bbA}_{k}^{-1}(z_2)\Big]+A_7,
$$
where
\begin{equation}\frac{1}{h}\int^{a_r}_{a_l} |K(\frac{x-z_1}{h})|
\E|A_7|du_1\leq\frac{M}{nv^2}.\label{h29}
\end{equation}
We then conclude that
\begin{equation}\label{f20}
B_1(z_1)+b_{12}(z_2)u_n(z_1)n^{-1}\E\tr\underline{\bbA}_{k}^{-1}(z_2)\frac{k-1}{n^2}\E_k\Big[\tr\bbA_{k}^{-1}(z_1)\underline{\bbA}_{k}^{-1}(z_2)\Big]
=:A_8,
\end{equation}
where $A_8$ satisfies (\ref{h29}) with $A_7$ replaced by $A_8$.

Summarizing the argument from (\ref{h5}) to (\ref{f20}) yields
\begin{equation}\label{f30}
n^{-1}\E_k\big[\tr\bbA_k^{-1}(z_1)
\underline{\bbA}_k^{-1}(z_2)\big]=-n^{-1}u_n(z_1) \E_k\big[\tr\underline{\bbA}_{k}^{-1}(z_2)\big]
\end{equation}
\begin{eqnarray*}
-u_n(z_1)b_{12}(z_1)b_{12}(z_2)n^{-1}\E\big(\tr\bbA^{-1}(z_2)\big)\Big[\frac{k-1}{n^2}\E_k
\big(\tr\bbA_{k}^{-1}(z_1)\underline{\bbA}_{k}^{-1}(z_2)\big)\Big]
+A_9,
\end{eqnarray*}
where  $A_9$ satisfies (\ref{h29}) with $A_7$ replaced by $A_9$.

By the formula ( see (2.2) of \cite{s3})
$
\underline{m}_n(z)=-z^{-1}n^{-1}\sum\limits_{k=1}^n\beta_k(z),
$
we have
\begin{equation}\label{d12}
\E\beta_1(z)=-z\E\underline{m}_n(z)
\end{equation}
It follows from (\ref{b22}) and Lemma \ref{lem1} that
$$
|\E\beta_1(z)-b_1(z)|=|b_1(z)^2\E(\beta_1(z)\xi_1^2(z))|\leq M/(nv)
$$
and from (\ref{f10}) that
\begin{equation}\label{u1}
|b_1(z)-b_{12}(z)|\leq M/(nv).
\end{equation}
These, together with (\ref{g14}) below, imply that
\begin{equation}\label{g36}
|b_{12}(z)+z\underline{m}_n^0(z)|\leq M/(nv).
\end{equation}
This, along with (\ref{h31}), ensures that
\begin{equation}
n^{-1}\E\tr\bbA^{-1}(z)=-\frac{c_n}{z+z\underline{m}_n^0(z)}+O(\frac{1}{nv}).\label{g9}
\end{equation}
We then conclude from (\ref{f30}), (\ref{g36}), (\ref{g9}), Lemma \ref{lem1} and (\ref{g14}) that
 \begin{eqnarray}\label{f33}
&& n^{-1}\E_k\Big[\tr \bbA_k^{-1}(z_1)
\E_k(\bbA_k^{-1}(z_2))\Big]\times\Big[1-\frac{k-1}{n}b_n(z_1,z_2)\Big]
\non
&=&\frac{b_n(z_1,z_2)}{z_1z_2\underline{m}_n^0(z_1)\underline{m}_n^0(z_2)}+A_{10},
\end{eqnarray}
where $A_{10}$ satisfies (\ref{h29}) with $A_{7}$ replaced by
$A_{10}$ and
$$
b_n(z_1,z_2)=\frac{c_n\underline{m}_n^0(z_1)\underline{m}_n^0(z_2)}{(1+\underline{m}_n^0(z_1))
(1+\underline{m}_n^0(z_2))}.
$$

From (2.19) in \cite{b2} and the inequality above (6.37) in
\cite{ker} we see that
\begin{equation}
\label{f56} |1-\frac{k-1}{n}b_n(z_1,z_2)|\geq Mv, \quad
|1-tb_n(z_1,z_2)|\geq Mv,\quad\text{for any} \quad t\in [0,1].
\end{equation}
It follows that
\begin{equation}\label{h20}
|n^{-1}\sum\limits_{k=1}^n\big(1-\frac{k-1}{n}b_n(z_1,z_2)\big)^{-1}-\int^1_0\big(1-tb_n(z_1,z_2)\big)^{-1}dt|\leq\frac{M}{nv^2}.
\end{equation}
Similarly we have
$$
|n^{-1}\sum\limits_{k=1}^n\big(1-|\frac{k-1}{n}b_n(z_1,z_2)|\big)^{-1}-\int^1_0\big(1-tb_n(z_1,z_2)\big)^{-1}dt|\leq\frac{M}{nv^2},
$$
Moreover from Lemma \ref{lem1} and (\ref{f58})
$$
\int^1_0\big(1-t|b_n(z_1,z_2)|\big)^{-1}dt=|b_n(z_1,z_2)|^{-1}\ln (1-|b_n(z_1,z_2)|)=O(\ln 1/v).
$$
It follows that
\begin{equation}
n^{-1}\sum\limits_{k=1}^n|1-\frac{k-1}{n}b_n(z_1,z_2)|^{-1}=O(\ln 1/v).\label{h30}
\end{equation}
We conclude from Lemma \ref{lem1}, (\ref{f33}), (\ref{h20}) and (\ref{h30}) that
\begin{eqnarray}
\label{f24} a_{n1}(z_1,z_2)&=&b_n(z_1,z_2) n^{-1}\sum\limits_{k=1}^n
\big(1-\frac{k-1}{n}b_n(z_1,z_2)\big)^{-1}+A_{11}\non
&=&b_n(z_1,z_2)\int^1_0\big(1-tb_n(z_1,z_2)\big)^{-1}dt+A_{12}\non
&=&-\ln (1-b_n(z_1,z_2))+A_6\non &=&-\ln
\Big((z_1-z_2)\underline{m}_n^0(z_1)\underline{m}_n^0(z_2)\Big)-\ln
(\underline{m}_n^0(z_1)-\underline{m}_n^0(z_2))+A_{12},
\end{eqnarray}
where $A_{11}$ and $A_{12}$ satisfy
\begin{equation}\label{g33}
\frac{1}{h}\int^{a_r}_{a_l} |K(\frac{x-z_1}{h})|
\E|A_j|du_1\leq\frac{M\ln 1/v}{nv^2},\ j=5,6
\end{equation}
and in the last step one uses the fact that via (\ref{h31})
$$
z_1-z_2=\frac{\underline{m}_n^0(z_1)-\underline{m}_n^0(z_2)}{\underline{m}_n^0(z_1)\underline{m}_n^0(z_2)}(1-b_n(z_1,z_2)).
$$

So far we have considered $z\in\gamma_2$, the top horizontal line.
The above argument evidently works for the case of $z\in\gamma_1$,
the bottom horizontal line, due to symmetry.

To deal with the cases when $z$ belongs to two vertical lines of the contour, from Fubini's theorem and (\ref{f63}) we obtain for $j=0,1,2$.
$$
\int^{a_r}_{a_l}\Big[\frac{1}{h}\int^{v_0}_{0}|K^{(j)}(\frac{x-u}{h}+iv)|dv\Big]du=\int^{v_0}_{0}\Big[\frac{1}{h}\int^{a_r}_{a_l}|K^{(j)}(\frac{x-u}{h}+iv)|du\Big]dv<\infty.
$$
This implies for $u\in [a_l,a_r]$
\begin{equation}
\label{g3}\frac{1}{h}\int^{v_0}_{0}|K^{(j)}(\frac{x-u}{h}+iv)|dv<\infty,\quad j=0,1,2.\end{equation}
We also need the estimates (1.9a) and (1.9b) of \cite{b2}, which hold under our truncation level.
That is
\begin{equation}\label{m7}
\P(\|\bbA\|\geq \mu_1)=o(n^{-l}),\ \P(\lambda_{\min}^\bbA\leq \mu_2)=o(n^{-l}),
\end{equation}
for any $\mu_1>(1+\sqrt{c})^2$,
$\mu_2<(1-\sqrt{c})^2$ and $l$. This implies that
\begin{equation}
\P(\|\bbA_k\|\geq \mu_1)=o(n^{-l}),\ \P(\lambda_{\min}^{\bbA_k}\leq \mu_2)=o(n^{-l}).\label{h8}
\end{equation}

Select a sequence of positive numbers $\varepsilon_n$
satisfying for some $\beta\in(0,1)$,
\begin{equation}\label{g21}
\varepsilon_n\downarrow0,\ \ \varepsilon_n\geq n^{-\beta}.
\end{equation} Then, as in \cite{b2}, we introduce a truncation version of $X_n(z)$ on the top half parts of the two vertical lines of the contour as follows:
$$
\hat{X}_n(z)=\begin{cases}X_n(z) & \text{for}\ u=a_r, a_l, v\in [n^{-1}\varepsilon_n,v_0h]\\
 X_n(a_r+in^{-1}\varepsilon_n) & \text{for}\ u=a_r,v\in [0,n^{-1}\varepsilon_n]\\
 X_n(a_l+in^{-1}\varepsilon_n) & \text{for}\ u=a_l,v\in [0,n^{-1}\varepsilon_n]
\end{cases}
$$
(one can similarly consider the bottom half parts of the two vertical lines).  It follows that with probability one
$$
\Big|\int K(\frac{x-z}{h})(\hat{X}_n(z)-\hat{X}_n(z))dz \Big|\leq Mh \varepsilon_n (\frac{1}{a_r-\lambda_{\max}}+\frac{1}{\lambda_{\min}-a_l})\rightarrow
0.
$$
Indeed, there is an extra $h$ above on the right hand which is needed in the proof of Theorem \ref{theo2}.

For $\hat{X}_n(z)$ on the two vertical lines $\gamma_2\cup \gamma_4$,
 (\ref{h6}) is still true because  there are at most finite number of points where the derivative of the corresponding truncation version of $\beta_k(z)$ do not exist.
  Moreover, for the truncation versions, the
higher moments of $\bbA^{-1}(z)$, $\bbA_k^{-1}(z)$ and
$\bbA_{kj}^{-1}(z)$ are bounded by (\ref{m7}) and (\ref{h8}) (see
(3.1) in \cite{b2}). As pointed out in the paragraph below (3.2) in
\cite{b2}, the moments of $\beta_1(z),\beta_{12}(z), \beta^{\tr}(z),
s_1^T\bbA_1^{-1}(z_1)\bbT\bbA_1^{-1}(z_2)\bbs_1$ are bounded as
well.  Using these facts, all the estimates holding for
$z\in\gamma_1\cup \gamma_2$ also holds for the case where $z\in
\gamma_r\cup\gamma_l$. Note that the length of the vertical line is
at most $h$. Via these facts, the arguments of the case $z\in
\gamma_r\cup\gamma_l$, two vertical lines, can follow from those of
the case $z\in\gamma_1\cup \gamma_2$ (here we omit the details) and
hence their limits have the same form as (\ref{f24}).

In the mean time, appealing to Cauchy's theorem gives
\begin{equation}\label{g53}
\frac{1}{h^2}\oint_{\mathcal{C}_1}\oint_{\mathcal{C}_2}
K'(\frac{x_1-z_1}{h})K'(\frac{x_2-z_2}{h})\ln\Big((z_1-z_2)\underline{m}_n^0(z_1)\underline{m}_n^0(z_2)\Big)dz_1dz_2=0,
\end{equation}
where the contour $\mathcal{C}_2$ is also a rectangle formed with
four vertices $a_l-\varepsilon\pm 2iv_0h$ and $a_{r}+\varepsilon\pm
2iv_0h$ with $\varepsilon>0$. One should note that the contour
$\mathcal{C}_2$ encloses the contour $\mathcal{C}_1$. Thus, in view
of (\ref{f24}), it remains to find the limit of the following
\begin{equation}\label{f62}
-\frac{1}{2h^2\pi^2}\oint_{\mathcal{C}_1}\oint_{\mathcal{C}_2}
K'(\frac{x_1-z_1}{h})K'(\frac{x_2-z_2}{h})\ln(\underline{m}_n^0(z_1)-\underline{m}_n^0(z_2))dz_1dz_2,
\end{equation}
which is done in Appendix 2.

\section{The convergence rate of $\E m_n(z)$ to $m(z)$}\label{con-Em}

The aim of this section is to develop a sharp order for $\E\Gamma_1^2$ and $\E\Gamma_1^3$
 which are crucial to the establishment of Theorem \ref{theo2} with the stringent bandwidth restriction. Throughout this section, let $z=u+iv$
with $u\in [a,b]$ and $v\geq M_1/\sqrt{n}$ where $M_1$ is a sufficiently large positive constant.

We begin with a series of Lemmas.

\begin{lemma}\label{lem11a}
Let
$$
g(z)=z+c_n-1+zc_nm_n^0(z)+zc_n\E m_n(z),
$$
where $m_n^0(z)$ stands for the one obtained from $m(z)$ with $c$ replaced by $c_n$. Then
\begin{equation}\label{m6}
|g(z)|\geq c_nv\mu_2\E\big(n^{-1}\tr\bbA^{-1}(z)\bbA^{-1}(\bar
z)\big)=c_n\mu_2\Im\big[\E \big(n^{-1}\tr\bbA^{-1}(z)\big)\big]
\end{equation}
and \begin{equation}\label{m10} |g(z)|\geq M\sqrt{v},
\end{equation}
where $0<\mu_2<(1-\sqrt{c_n})^2$.
\end{lemma}

\begin{proof}
It is straightforward to check that
\begin{equation}\label{m5}
\Im \big(z+c_n-1+zc_nm_n^0(z)+zc_nm_n(z)\big)\geq
v+c_nv\lambda_{\min}(\bbA)n^{-1}\tr\bbA^{-1}(z)\bbA^{-1}(\bar z).
\end{equation}
 Write
$$
g(z)=z+c_n-1+zc_nm_n^0(z)+zc_n\E\big[m_n(z)I(D)\big]+zc_n\E\big[m_n(z)I(D^c)\big],
$$
where the event $D=(\lambda_{\min}(\bbA)\leq\mu_2)$. Then (\ref{m6})
follows from (\ref{m5}) and (\ref{m7}). By (\ref{g18})
\begin{equation}\label{m12}
|z+c_n-1+2zc_nm_n^0(z)|\geq M\sqrt{v}.
\end{equation}
On the other hand, it is proved in (6.109) of \cite{g2} that
\begin{equation}\label{m13}
|\E m_n(z)-s(z)|\leq\frac{M}{nv^{3/2}}\leq \rho_1\sqrt{v},
\end{equation}
where $\rho_1$ is sufficiently small. We then conclude from (\ref{m12}) and (\ref{m13}) that (\ref{m10}) holds.
\end{proof}

\begin{lemma}\label{lem14}
\begin{equation}\label{h48*}
|\E n^{-1}\tr\bbA_1^{-2}(z)|\leq M/\sqrt{v},\quad  n^{-1}\sum\limits_{k=1}^n\big|\big(\E n^{-1}\tr\bbA_k^{-1}(z)\underline{\bbA}_k^{-1}(z)\big)^2\big|\leq  M/v,
\end{equation}
\begin{equation}\label{h48}
 n^{-1}\sum\limits_{k=1}^n|\E n^{-1}\tr\bbA_k^{-2}(z)\underline{\bbA}_k^{-1}(z)|\leq\frac{M}{v|z+c_n-1+2zc_nm_n^0(z)|},
 \end{equation}
 \begin{equation}\label{h48a}
 n^{-1}\sum\limits_{k=1}^n|\E n^{-1}\tr\bbA_k^{-1}(z)\underline{\bbA}_k^{-1}(z)\E n^{-1}\tr\bbA_k^{-2}(z)\underline{\bbA}_k^{-1}(z)|\leq M v^{-3/2},
\end{equation}
and
\begin{equation}\label{h48**}
 |(ng(z))^{-1}\sum\limits_{k=1}^n\E n^{-1}\tr\bbA_k^{-2}(z)\underline{\bbA}_k^{-2}(z)|\leq\frac{M}{v^{2}|z+c_n-1+2zc_nm_n^0(z)|}.
 \end{equation}
\end{lemma}
\begin{remark}
From the derivation of (\ref{h48**}) we see that the left side of the inequality of (\ref{h48**}) multiplied  by $g(z)$ is still less than the right side of the inequality.
\end{remark}

\begin{proof}
 Consider $\E n^{-1}\tr\bbA_k^{-2}(z)$ first. When replacing $\underline{\bbA}_k^{-1}(z_2)$ by $\bbA_k^{-1}(z)$, the derivation in the last section for $\E_k(n^{-1}\tr\bbA_k^{-1}(z_1)$\ $\underline{\bbA}_k^{-1}(z_2))$ also
works for $\E n^{-1}\tr\bbA_k^{-2}(z)$ except (\ref{m48}),
(\ref{h38}) and the argument starting from (\ref{h37}). It is
unnecessary to distinguish between the cases $j<k$ and $j>k$ in the
current case and so we need not consider (\ref{m48}). By Lemma
\ref{lem1}, (\ref{f7}) and (\ref{f13}), (\ref{h38}) reduces to
\begin{eqnarray*}
&&n^{-1}\sum\limits_{j}\E\big[\beta_{kj}(z)\xi_{kj}(z)\bbs_j^T\bbA^{-2}_{kj}(z)
\bbs_j\bbs_j^T\bbA_{kj}^{-1}\bbs_ju_n(z)\big]\\
&=&n^{-1}b_{12}u_n(z)\sum\limits_{j<k} \big[\E
\big(\eta_{kj}\bbs_j^T\bbA^{-2}_{kj}\bbs_jn^{-1}\tr\bbA_{kj}^{-1}\big)+\E
\big(\Gamma_{kj}\Gamma^{(2)}_{kj}
\big)\E n^{-1}\tr\bbA_{kj}^{-1}\big]+O(n^{-1}v^{-2})\\
&=&O(n^{-1}v^{-2}).
\end{eqnarray*}
Moreover, from Lemma \ref{lem1}, (\ref{h37}) turns out to be
\begin{eqnarray}
&&B_1(z)+n^{-3}b_{12}(z)u_n(z)\sum\limits_{j}\E\big[\tr\bbA_{kj}^{-2}(z)\tr\bbA_{kj}^{-1}(z)\big] \nonumber \\
&=&B_1(z)+b_{12}(z)u_n(z)n^{-1}\E\tr\bbA_{1}^{-1}(z)n^{-1}\E\tr\bbA_{1}^{-2}(z)+O(n^{-1}v^{-2}).\label{g4}
\end{eqnarray}
Therefore, as in (\ref{f33}), we have
\begin{equation}\label{f33*}
n^{-1}\E\big[\tr \bbA_k^{-2}(z) \big]\times\big[1-b_n(z,z)\big]
=\frac{b_n(z,z)}{(z\underline{m}_n^0(z))^2}+O(n^{-1}v^{-2}),
\end{equation}
where $b_n(z,z)$ is obtained from $b_n(z_1,z_2)$ from (\ref{f33}) with $z_1=z_2=z$.
From (\ref{h31}) and (\ref{m45}) one may verify that
\begin{equation}
1-b_n(z,z)=-(z+c_n-1+2zc_nm_n^0(z))\frac{\underline{m}_n^0(z)}{1+\underline{m}_n^0(z)}. \label{g16}
\end{equation}
It follows from (\ref{f33*}), (\ref{g16}), (\ref{f58}), (\ref{m12}),
(\ref{m46}) and Lemma \ref{lem1} that
\begin{equation}\label{h46}
\big|n^{-1}\E\tr \bbA_1^{-2}(z)\big|\leq M/\sqrt{v}.
\end{equation}

As for the second inequality in (\ref{h48*}), checking the above
proof for $\E n^{-1}\tr\bbA_1^{-2}(z)$ and the last section for
$\E_k(n^{-1}$ $\tr\bbA_k^{-1}(z_1)\underline{\bbA}_k^{-1}(z_2))$ and
referring to (\ref{f33}) we have
\begin{equation}\label{g25}
n^{-1}\E\big[\tr \bbA_k^{-1}\underline{\bbA}_k^{-1}
\big]\times\big[1-\frac{k-1}{n}b_n(z,z)\big]
=\frac{b_n(z,z)}{(z\underline{m}_n^0(z))^2}+O(n^{-1}v^{-2}).
\end{equation}
As in (\ref{h20}) and (\ref{h30}) we obtain
\begin{eqnarray}
n^{-1}\sum\limits_{k=1}^n (1-\frac{k-1}{n}|b_n(z,z)|)^{-2}&=&\int_0^1 (1-t|b_n(z,z)|)^{-2}dt+O(\frac{1}{nv^3})
\non&=&\frac{|b_n(z,z)|}{1-|b_n(z,z)|}+O(\frac{1}{nv^3})=O(\frac{1}{v})\label{g24}
\end{eqnarray}
It follows from (\ref{g25}) and (\ref{g24}) that
\begin{equation}
n^{-1}\sum\limits_{k=1}^n|(\E n^{-1}\tr\bbA_k^{-1}(z)\underline{\bbA}_k^{-1}(z))^2|=O(v^{-1}).\label{g26}
\end{equation}

Consider (\ref{h48}) next. The strategy is to use (\ref{f21}). From (\ref{f7}) and Lemma \ref{lem1} we obtain
$$\E\big[n^{-1}\tr\underline{\bbA}_k^{-1}\bbA_k^{-1}D(z)\big]=\frac{b_{12}u_n(z)}{n^2}\sum\limits_{j\neq k}^n\E\big[\bbs_j^T\bbA^{-1}_{kj}
\underline{\bbA}_{k}^{-1}\bbA_k^{-1}\bbA^{-1}_{kj}\bbs_j\beta_{kj}\big]=O(v^{-1}).
$$

 Apply (\ref{f13}) and (\ref{f6}) to write
\begin{equation}\label{h11}\E\big[n^{-1}\tr\underline{\bbA}_k^{-1}\bbA_k^{-1}C(z)\big]=
n^{-1}u_n(z)b_{12}\sum\limits_{j\neq k}^n
\E\big[(\xi_{kj}+\beta_{12}(\xi_{kj})^2)\bbs_j^T\bbA^{-1}_{kj}\underline{\bbA}_k^{-1}\bbA_k^{-1}\bbs_j\big]\end{equation}
$$
=\frac{u_n(z)b_{12}(z)}{n}\sum\limits_{j<k}^n(C_3+C_4+C_5+C_6+C_7+C_8)+\frac{u_n(z)b_{12}(z)}{n}\sum\limits_{j>k}^n(C_9+C_{10}),
$$
where
$$
C_3=\E\big[\eta_{kj}\centre^j(\bbs_j^T\bbA^{-1}_{kj}\underline{\bbA}_{kj}^{-1}\bbA_{kj}^{-1}\bbs_j\big)
+\Gamma_{kj}(n^{-1}\tr\bbA^{-2}_{kj}\underline{\bbA}_{kj}^{-1}
-\E n^{-1}\tr\bbA^{-2}_{kj}\underline{\bbA}_{kj}^{-1})\big]
$$
$$
C_4=\E\big[\xi_{kj}(\bbs_j^T\bbA^{-1}_{kj}\underline{\bbA}_{kj}^{-1}\bbs_j)^2\underline{\beta}_{kj}\big],\
C_5=\E\big[\xi_{kj}(\bbs_j^T\bbA^{-1}_{kj}\underline{\bbA}_{kj}^{-1}\bbs_j)^2\bbs_j^T\bbA^{-1}_{kj}\bbs_j\beta_{kj}
\underline{\beta}_{kj}\big],
$$
$$
C_6=\E\big[\xi_{kj}\bbs_j^T\bbA^{-1}_{kj}\underline{\bbA}_{kj}^{-1}\bbA^{-1}_{kj}\bbs_j\bbs_j^T\bbA^{-1}_{kj}\bbs_j\beta_{kj}\big],\
C_7=\E\big[\beta_{kj}(\xi_{kj})^2\bbs_j^T\bbA^{-1}_{kj}\underline{\bbA}_k^{-1}\bbA_{kj}^{-1}\bbs_j\big]
$$
$$
C_8=\E\big[\beta_{kj}^2(\xi_{kj})^2\bbs_j^T\bbA^{-1}_{kj}\underline{\bbA}_k^{-1}\bbA_{kj}^{-1}\bbs_j\bbs_j^T\bbA_{kj}^{-1}\bbs_j\big],
$$
and
$$
C_9=\E\big[\beta_{kj}\xi_{kj}\bbs_j^T\bbA^{-1}_{kj}\underline{\bbA}_k^{-1}\bbA^{-1}_{kj}\bbs_j\big],\ C_{10}=\E\big[\beta_{kj}^2\xi_{kj}\bbs_j^T\bbA^{-1}_{kj}\underline{\bbA}_k^{-1}\bbA^{-1}_{kj}\bbs_j\bbs_j^T\bbA^{-1}_{kj}\bbs_j\big].
$$
It follows from Lemmas \ref{lem1}, \ref{lem7}, (\ref{f13}), (\ref{f6}) and (\ref{f7}) that $|C_j|\leq M/v,j=3,6,7,8,9,10$ and
$$
C_j=\big(\E n^{-1}\tr\bbA_k^{-1}(z)\underline{\bbA}_k^{-1}(z)\big)^2C_{j1}+O(v^{-1}),\quad j=4,5,
$$
where $|C_{j1}|\leq M/\sqrt{nv}$. Therefore
\begin{equation}\label{h11*}
\E\big[n^{-1}\tr\underline{\bbA}_k^{-1}\bbA_k^{-1}C(z)\big]=\big(\E n^{-1}\tr\bbA_k^{-1}(z)\underline{\bbA}_k^{-1}(z)\big)^2A_{13}+O(v^{-1}),
\end{equation}
where $|A_{13}|\leq M/\sqrt{nv}$.

Next we use (\ref{f13}) and (\ref{f6}) to write
\begin{eqnarray*}
\E \Big[n^{-1}\tr\underline{\bbA}_k^{-1}\bbA_k^{-1}B(z)\Big]&=&
n^{-1}u_n(z) \sum\limits_{j\neq k}^n\E \Big[\centre^j\big(\bbs_j^T\bbA^{-1}_{kj}\underline{\bbA}_k^{-1}\bbA_k^{-1}\bbs_j\big)\Big]\\
&=&\frac{u_n(z)}{n}\sum\limits_{j<k}^n(B_3+B_4+B_5)+\frac{u_n(z)}{n}\sum\limits_{j>k}^nB_6
\end{eqnarray*}
where
$$
B_3=\E\big[(\bbs_j^T\bbA^{-1}_{kj}\underline{\bbA}_{kj}^{-1}\bbs_j)^2\underline{\beta}_{kj}-
n^{-1}\bbs_j^T\underline{\bbA}_{kj}^{-1}\bbA_{kj}^{-2}\underline{\bbA}_{kj}^{-1}\bbs_j\underline{\beta}_{kj}\big]
$$
$$
B_4=\E\big[(\bbs_j^T\bbA^{-1}_{kj}\underline{\bbA}_{kj}^{-1}\bbs_j)^2\bbs_j^T\bbA_{kj}^{-1}\bbs_j\underline{\beta}_{kj}\beta_{kj}-
n^{-1}\bbs_j^T\bbA_{kj}^{-2}\underline{\bbA}_{kj}^{-1}\bbs_j\bbs_j^T\bbA_{kj}^{-1}
\underline{\bbA}_{kj}^{-1}\bbs_j\underline{\beta}_{kj}\beta_{kj}\big],
$$
$$
B_5=-\E\big[\bbs_j^T\bbA^{-1}_{kj}\underline{\bbA}_{kj}^{-1}\bbA_{kj}^{-1}\bbs_j\bbs_j^T\bbA_{kj}^{-1}\bbs_j\beta_{kj}-
n^{-1}\bbs_j^T\bbA_{kj}^{-2}\underline{\bbA}_{kj}^{-1}\bbA_{kj}^{-1}\bbs_j\beta_{kj}\big]
$$
and
$$
B_6=-\E(\bbs_j^T\bbA^{-1}_{kj}\underline{\bbA}_k^{-1}\bbA_{kj}^{-1}\bbs_j\bbs_j^T
\bbA_{kj}^{-1}\bbs_j\beta_{kj}-n^{-1}\bbs_j^T\bbA_{kj}^{-2}
\underline{\bbA}_k^{-1}\bbA_{kj}^{-1}\bbs_j\beta_{kj}).
$$
In view of Lemmas \ref{lem1}, \ref{lem7}, (\ref{f13}), (\ref{f10}) and (\ref{f7}) we have
$$
B_j=\big(\E n^{-1}\tr\bbA_k^{-1}\underline{\bbA}_k^{-1}\big)^2B_{j1}+O(v^{-1}),\quad j=3,4,
$$
where $|B_{j1}|\leq M$ and
$$
B_j=-b_{12}\E n^{-1}\tr\bbA_k^{-1}\E(n^{-1}\tr\bbA_k^{-2}\underline{\bbA}_k^{-1})+O(v^{-1}),\quad j=5,6.
$$
Thus
\begin{eqnarray}
&&\E\big(n^{-1}\tr\underline{\bbA}_k^{-1}\bbA_k^{-1}B(z)\big)=\big(\E(n^{-1}\tr\bbA_k^{-1}\underline{\bbA}_k^{-1})\big)^2A_{14}\label{g3}\\
&&-b_{12}u_n(z)\E n^{-1}\tr\bbA_k^{-1}\E(n^{-1}\tr\bbA_k^{-2}\underline{\bbA}_k^{-1})+O(v^{-1}), \nonumber
\end{eqnarray}
where $|A_{14}|\leq M$.

Note that the coefficient of $\E(n^{-1}\tr\bbA_k^{-2}\underline{\bbA}_k^{-1})$ in (\ref{g3}) is the same as that of $\E n^{-1}\tr\bbA_k^{-2}$ in (\ref{g4}). Summarizing the above we have thus obtained
\begin{equation}\label{g5}
\E n^{-1}\tr\bbA_k^{-2}\underline{\bbA}_k^{-1}(1-b_n(z,z))=-u_n\E n^{-1}\tr\bbA_k^{-1}\underline{\bbA}_k^{-1}
+\big(\E n^{-1}\tr\bbA_k^{-1}\underline{\bbA}_k^{-1}\big)^2A_{15}+O(v^{-1}),
\end{equation}
where $|A_{15}|\leq M$. This, together with (\ref{h48*}), implies (\ref{h48}).

As for (\ref{h48a}), in view of (\ref{g5}) we have
\begin{eqnarray}
&\qquad (1-b_n(z,z))\E(\frac{1}{n}\tr\bbA_k^{-2}\underline{\bbA}_k^{-1})\E(n^{-1}\tr\bbA_k^{-1}\underline{\bbA}_k^{-1})
=a_1(z)\big[\E(n^{-1}\tr\bbA_k^{-1}\underline{\bbA}_k^{-1})\big]^2\label{g23}\\
&+a_2(z)\big[\E(n^{-1}\tr\bbA_k^{-1}\underline{\bbA}_k^{-1})\big]^3
+a_3(z)\E(n^{-1}\tr\bbA_k^{-1}\underline{\bbA}_k^{-1}),\nonumber
\end{eqnarray}
where $|a_1(z)|\leq M, |a_2(z)|\leq M$ and $|a_3(z)|\leq M/v$.
Moreover we conclude from (\ref{f56}), (\ref{g24}), (\ref{g16}) and Lemma \ref{lem11a} that
\begin{eqnarray}
n^{-1}\sum\limits_{k=1}^n (1-\frac{k-1}{n}|b_n(z,z)|)^{-3}&=&\int_0^1 (1-t|b_n(z,z)|)^{-3}dt+O(v^{-1})
\non&=&(1-|b_n(z,z)|)^{-2}+O( v^{-1})=O(v^{-1})\label{g24*}.
\end{eqnarray}
Then (\ref{h48a}) follows from (\ref{g25}), (\ref{g23}), (\ref{g24*}), and (\ref{g16}).

To establish (\ref{h48**}), we observe that Lemmas \ref{lem11a} and \ref{lem1} imply
\begin{eqnarray}
&&|\E n^{-1}\tr\bbA^{-1}(z)\bbA^{-1}(\bar z)|/|g(z)|\leq M/v,\label{g29}
\\
&&\E| n^{-1}\tr\bbA^{-1}(z)\bbA^{-1}(\bar z)|^k/|g(z)|
 \leq M(v^k|g(z)|)^{-1} \big[E|\Gamma|^k+|\Im \big(\E n^{-1}\tr\bbA^{-1}\big)|^k\big] \nonumber\\
 &&\leq M/v^k, \quad k=2,4,8.\nonumber
\end{eqnarray}
This key fact implies that whenever
$$\E n^{-1}\tr\bbA^{-1}(z)\bbA^{-1}(\bar z) \ \text{and} \ \
\E| n^{-1}\tr\bbA^{-1}(z)\bbA^{-1}(\bar z)|^k,\ k=2,4,8$$
 appear, dividing them by
$|g(z)|$ does not change their original sizes. As consequences of this fact, applying (\ref{g29}) and (\ref{f10}) then ensures that (\ref{a39}) and (\ref{b18}) are still true when we divide the expectations in  them by $|g(z)|$. For example, by (\ref{g29}) and (\ref{f10}) we have for $m=2,4,6,8$
\begin{eqnarray}
|g(z)|^{-1}\E|\eta_{kj}|^m&\leq& M\big(n^{m/2}|g(z)|\big)^{-1}\E|n^{-1}\tr\bbA^{-1}(z)\bbA^{-1}(\bar z)|^\frac{m}{2} \label{g28}\\
&& \quad +M\big(n^{m/2}v^{m/2}\big)^{-1}
\leq M\big(n^{m/2}v^{m/2}\big)^{-1}. \nonumber
\end{eqnarray}
We also provide an argument in Lemma \ref{lem7} for (\ref{a39}) when the expectation in it is divided by $g(z)$. Moreover by (\ref{g29}) and (\ref{f10}) one may verify that the first three conclusions of Lemma \ref{lem7}
are still true when the expectations in them are divided by $|g(z)|$ as in the last claim of Lemma \ref{lem7}. From now on until the end of this lemma we mean the corresponding expressions divided by $|g(z)|$
 whenever we quote (\ref{a39}), (\ref{b18}) and Lemma \ref{lem7}.

Now we resort to use (\ref{f21}) again. From (\ref{g29}), (\ref{f13}) and (\ref{f7}), Lemmas \ref{lem1} and \ref{lem11a} we have
$$|g(z)|^{-1}\big|\E\big[n^{-1}\tr\underline{\bbA}_k^{-2}\bbA_k^{-1}D(z)\big]\big|\leq M\big(v^3n^2|g(z)|\big)^{-1}\sum\limits_{j\neq k}^n\E\big(\|\bbs_j^T\bbA_{kj}^{-1}\|^2|\beta_{kj}|\big)
$$
$$
\leq M\big(v^3n^2|g(z)|\big)^{-1}\sum\limits_{j\neq k}^n\big[\E n^{-1}\tr\bbA_{kj}^{-1}(z)\bbA_{kj}^{-1}(\bar z))|b_{12}|
+\E\big(\|\bbs_j^T\bbA_{kj}^{-1}\|^2||b_{12}\beta_{kj}\xi_{kj}|\big)\big]\leq M/v^2,
$$
where $\|\cdot\|$ denotes the spectral norm or Euclidean norm of matrices or vectors.
As in (\ref{h11*}) and (\ref{g3}), by (\ref{a39}), (\ref{f10}), (\ref{b18}), (\ref{f13}) and Lemmas \ref{lem7}, \ref{lem1} one may verify that
$$g(z)^{-1}\E\big[n^{-1}\tr\underline{\bbA}_k^{-2}\bbA_k^{-1}C(z)\big]=\big(\E n^{-1}\tr\bbA_k^{-1}\underline{\bbA}_k^{-1}\big)^2A_{16}
$$$$+
\E(n^{-1}\tr\bbA_k^{-1}\underline{\bbA}_k^{-1})\E(n^{-1}\tr\bbA_k^{-1}\underline{\bbA}_k^{-2})A_{17}+O(v^{-2}),
$$
and that
$$
g(z)^{-1}\E\big[n^{-1}\tr\underline{\bbA}_k^{-2}\bbA_k^{-1}B(z)\big]
=\big(\E n^{-1}\tr\bbA_k^{-1}\underline{\bbA}_k^{-1}\big)^2A_{18}+
$$$$\E(n^{-1}\tr\bbA_k^{-1}\underline{\bbA}_k^{-1})\E(n^{-1}\tr\bbA_k^{-1}\underline{\bbA}_k^{-2})A_{19}-
b_{12}u_n(z)\E n^{-1}\tr\bbA_k^{-1}\underline{\bbA}_k^{-1}\E(n^{-1}\tr\bbA_k^{-2}\underline{\bbA}_k^{-2})+O(v^{-2}),
$$
where $|A_{16}|\leq M$, $|A_{17}|\leq M/\sqrt{nv}$, $|A_{18}|\leq M/\sqrt{v}$ and $|A_{19}|\leq M$. These imply that
$$
g(z)^{-1}\E(n^{-1}\tr\bbA_k^{-2}\underline{\bbA}_k^{-2})(1-b_n(z,z))=-u_n(z)g(z)^{-1}\E(n^{-1}\tr\bbA_k^{-1}\underline{\bbA}_k^{-2})
$$
$$
+\big(\E n^{-1}\tr\bbA_k^{-1}\underline{\bbA}_k^{-1}\big)^2A_{20}+\E(n^{-1}\tr\bbA_k^{-1}\underline{\bbA}_k^{-1})\E(n^{-1}\tr\bbA_k^{-1}\underline{\bbA}_k^{-2})A_{21}
+O(v^{-2}),
$$
where $|A_{20}|\leq M/\sqrt{v}$ and $|A_{21}|\leq M$. Thus (\ref{h48**}) follows from (\ref{h48*}), (\ref{h48}) and (\ref{h48a}) immediately.
\end{proof}

\begin{lemma}\label{lem15}
\begin{equation}\label{m23}
\E|\bbe_i^T\bbA_1^{-1}(z)\bbe_j-\E\bbe_i^T\bbA_1^{-1}(z)\bbe_j|^2 \leq M/(nv^2),
\end{equation}
\begin{equation}\label{m23*}
\E|\bbe_i^T\bbA_1^{-2}(z)\bbe_j-\E\bbe_i^T\bbA_1^{-2}(z)\bbe_j|^2\leq M/(nv^4),
\end{equation}
\begin{equation}\label{m23**}
\E|\bbe_i^T\bbA_1^{-1}(z)\bbe_i-\E\bbe_i^T\bbA_1^{-1}(z)\bbe_i|^4 \leq M/(nv^2)
\end{equation}
and
\begin{equation}\label{m23***}
\E|\bbe_i^T\bbA_1^{-2}(z)\bbe_i-\E\bbe_i^T\bbA_1^{-2}(z)\bbe_i|^4 \leq M/(nv^6),
\end{equation}
where $\bbe_i$ is the $p$-dimensional vector with the $i$-th coordinate being $1$ and the remaining being zero.
\end{lemma}

\begin{proof}
 Consider $i=j$ first. Let $\Phi_1(z)=\bbe_i^T\bbA_1^{-1}(z)\bbe_i-E\bbe_i^T\bbA_1^{-1}(z)\bbe_i$. By (\ref{f6}) and (\ref{f13}) write
\begin{eqnarray}
\Phi_1(z)&=&\sum\limits_{k=2}^p(\E_k-\E_{k-1})(\bbe_i^T(\bbA_1^{-1}(z)-\bbA_{1k}^{-1}(z))\bbe_i) \nonumber \\
&=&-\sum\limits_{k=2}^p(\E_k-\E_{k-1})(\bbs_k^T\bbA_{1k}^{-1}(z)\bbe_i\bbe_i^T\bbA_{1k}^{-1}(z)\bbs_k\beta_{1k}). \label{h2}
%&=&-\sum\limits_{k=2}^p(\E_k-\E_{k-1})(\gamma_{ke}b_{12}+b_{12}\xi_{1k}\beta_{1k}\bbs_k^T\bbA_{1k}^{-1}(z)\bbe_i\bbe_i^T\bbA_{1k}^{-1}(z)\bbs_k),\label{h10}
\end{eqnarray}
Let
$$
\gamma_{ke}=\gamma_{rke}-n^{-1}\bbe_i^T\bbA_{1k}^{-2}(z)\bbe_i,\ \gamma_{rke}=\bbs_k^T\bbA_{1k}^{-1}(z)\bbe_i\bbe_i^T\bbA_{1k}^{-1}(z)\bbs_k.
$$
From  Lemmas \ref{lem8} and \ref{lem1} we obtain
\begin{eqnarray}
\E|\gamma_{rke}\beta_{1k}|^4&\leq&
M(\E|\beta_{1k}|^8)^{1/2}(\E|\gamma_{ke}|^{8})^{1/2}+M(\E|\beta_{1k}|^8)^{1/2}(\E|n^{-1}\bbe_i^T\bbA_{1k}^{-2}(z)\bbe_i|^8)^{1/2}\non
&\leq &M(n^4v^6)^{-1}(1+ (\E|\Phi_1(z)|^4)^{1/2}),\label{h10}
\end{eqnarray}
where we also use the facts that
$$
\bbe_i^T\bbA_{1k}^{-1}(z)\bbA_{1k}^{-1}(\bar z)\bbe_i=v^{-1}\Im (\bbe_i^T\bbA_{1k}^{-1}(z)\bbe_i)
$$
and that via (\ref{f6}), Lemmas \ref{lem8} and \ref{lem1}
\begin{equation}\label{h7}
\E|\bbe_i^T\bbA_{1k}^{-1}(z)\bbe_i-\bbe_i^T\bbA_{1}^{-1}(z)\bbe_i|^4\leq
M(nv^2)^{-1}.
\end{equation}

By Lemma \ref{lem8}, estimates similar to (\ref{f44}) and (\ref{g1})
we have
$$\E_{k-1}\Big|\gamma_{rke}\Big|^2\leq\frac{M}{n^2v^2}\E_{k-1}|\bbe_i^T\bbA_{1k}^{-1}(z)\bbe_i|^2,$$
$$
\E_{k-1}\Big|\gamma_{rke}(\beta^{tr}_{1k})^2\eta_{1k}\Big|^2
\leq\frac{M}{n^3v^3}\Big(\E_{k-1}|\bbe_i^T\bbA_{1k}^{-1}(z)\bbe_i|^4\Big)^{1/2}\Big(\E_{k-1}|\beta^{tr}_{1k}|^4\Big)^{1/2}$$
and
$$
\E_{k-1}\Big|\gamma_{rke}\beta_{1k}(\beta^{tr}_{1k})^2\eta_{1k}^2\Big|^2
$$
$$\leq\frac{M}{v^2}\E_{k-1}\Big|\gamma_{rke}\beta^{tr}_{1k}\eta_{1k}^2\Big|^2
\leq\frac{M}{n^4v^6}\Big(\E_{k-1}|\bbe_i^T\bbA_{1k}^{-1}(z)\bbe_i|^4\Big)^{1/2}\Big(\E_{k-1}|\beta^{tr}_{1k}|^2\Big)^{1/2}.
$$
These, together with
$\beta_{1k}=\beta^{tr}_{1k}-(\beta^{tr}_{1k})^2\eta_{1k}+\beta_{1k}(\beta^{tr}_{1k})^2\eta_{1k}^2$,
imply that
\begin{equation}\label{h18}
\E_{k-1}\Big|\gamma_{rke}\beta_{1k}\Big|^2
\leq
\frac{M}{n^2v^2}\Big(\E_{k-1}|\bbe_i^T\bbA_{1k}^{-1}(z)\bbe_i|^4\Big)^{1/2}\Big[\Big(\E_{k-1}|\beta^{tr}_{1k}|^4\Big)^{1/2}+1\Big].
\end{equation}
It follows from (\ref{h10}), (\ref{h18}), (\ref{h7}), Burkholder's inequality and Lemma \ref{lem1} that
$$
E|\Phi_1(z)|^4\leq ME\Big(\sum\limits_{k=2}^p\E_{k-1}\Big|\gamma_{rke}\beta_{1k}\Big|^2 \Big)^2+M
\sum\limits_{k=2}^pE|\gamma_{rke}\beta_{1k}|^4
$$
$$
\leq \frac{M}{nv^2}(E|\Phi_1(z)|^4)^{1/2}+\frac{M}{nv^2}.
$$
Solving the inequality yields (\ref{m23**}).

As for $i\neq j$, from (\ref{m23**}), (\ref{h7}) and Burkholder's inequality we obtain
$$
\E|\bbs_k^T\bbA_{1k}^{-1}(z)\bbe_i|^4\leq \frac{M}{n^2}E|\bbe_i^T\bbA_{1k}^{-1}(z)\bbA_{1k}^{-1}(\bar z)\bbe_i|^2\leq \frac{M}{n^2v^2}E|\bbe_i^T\bbA_{1}^{-1}(z)\bbe_i|^2\leq\frac{M}{n^2v^2}
$$
and
$$
\E|\bbs_k^T\bbA_{1k}^{-1}(z)\bbe_i|^8\leq\frac{M}{n^4v^6}.
$$
These ensure that
$$
\E|\bbs_k^T\bbA_{1k}^{-1}(z)\bbe_i\bbe_j^T\bbA_{1k}^{-1}(z)\bbs_k|^2\leq M (\E|\bbs_k^T\bbA_{1k}^{-1}(z)\bbe_i|^4\E|\bbe_j^T\bbA_{1k}^{-1}(z)\bbs_k|^4)^{1/2}\leq M(n^2v^2)^{-1}
$$
and
$$
\E|\xi_{1k}\beta_{1k}\bbs_k^T\bbA_{1k}^{-1}(z)\bbe_i\bbe_j^T\bbA_{1k}^{-1}(z)\bbs_k)|^2
$$$$\leq M(\E|\bbs_k^T\bbA_{1k}^{-1}(z)\bbe_i|^8\E|\bbe_j^T\bbA_{1k}^{-1}(z)\bbs_k|^8)^{1/4}(\E|\xi_{1k}|^8\E|\beta_{1k}|^8)^{1/4}
\leq M(n^3v^4)^{-1}.
$$
By the above two estimates, (\ref{f13}) and replacing $\bbe_i$ of (\ref{h2}) by $\bbe_j$ we obtain (\ref{m23}).

In view of Cauchy's theorem we have
\begin{equation}\label{g19}
\E|\bbe_i^T\bbA_1^{-2}(z)\bbe_j-\E\bbe_i^T\bbA_1^{-2}(z)\bbe_j|^2\leq
M v^{-2}\sup\limits_{\zeta\in \Gamma_\zeta}\E|\Phi_1(\zeta)|^2,
\end{equation}
where  $\Gamma_\zeta=\{\zeta: |\zeta-z|=v/2\}$. This, together with (\ref{m23}), ensures (\ref{m23*}). Similarly from (\ref{m23**}) we can obtain (\ref{m23***}).

\end{proof}

\begin{lemma}\label{lem12}
\begin{equation}\label{m32}
\big|\E(\eta_1)^3\big|\leq M/(n^{\frac{3}{2}}v),
\end{equation}
\begin{equation}\label{m33}
\big|\E\big[\E_k\big(\eta_k^{(2)}\big)\E_k\big(\eta_k^{(2)}\eta_k\big)\big]\big|\leq M/(nv^2),
\end{equation}
\end{lemma}

\begin{proof}
 For $m=1,2,3$, write
$$
\bbs_k^T\bbB_m\bbs_k-n^{-1}\tr\bbB_1=\sum\limits_{i=1}^p(X_{ki}^2-1)(\bbB_m)_{ii}+\sum\limits_{i\neq
j}X_{ki}X_{kj}(\bbB_m)_{ij},
$$
where $\bbB_1,\bbB_2,\bbB_3$ are symmetric, independent of $\bbs_k$.
A direct calculation then yields
\begin{equation}\label{m24}
\E\Big[\prod\limits_{m=1}^3(\bbs_k^T\bbB_m\bbs_k-n^{-1}\tr\bbB_m)\Big]
=n^{-3}\E(X_{11}^2-1)^3\sum\limits_{i=1}^p\E\big[\prod\limits_{m=1}^3(\bbB_m)_{ii}\big]
\end{equation}
\begin{equation}\label{m25}
+2n^{-3}(\E X_{11}^3)^2\sum\limits_{m_1,m_2,m_3}\sum\limits_{i_1\neq
i_2}\E\big[(\bbB_{m_1})_{i_1i_1}(\bbB_{m_2})_{i_2i_2}(\bbB_{m_3})_{i_1i_2}\big]
\end{equation}
\begin{equation}\label{m26*}
+4n^{-3}(\E X_{11}^4-1)\sum\limits_{m_1,m_2,m_3}\sum\limits_{i_1\neq
i_2}\E\big[(\bbB_{m_1})_{i_1i_1}(\bbB_{m_2})_{i_1i_2}(\bbB_{m_3})_{i_1i_2}\big]
\end{equation}
\begin{equation}\label{m26}
+4n^{-3}(\E X_{11}^3)^2\sum\limits_{i_1\neq
i_2}\E\big[(\bbB_1)_{i_1i_2}(\bbB_2)_{i_1i_2}(\bbB_3)_{i_1i_2}\big]
\end{equation}
\begin{equation}\label{m27}
+8n^{-3}\sum\limits_{i_1\neq i_2,i_2\neq i_3,i_1\neq i_3}
\E\big[ (\bbB_1)_{i_1i_2}(\bbB_2)_{i_1i_3}(\bbB_3)_{i_2i_3}\big].
\end{equation}
where each $m_i$ runs over $1,2,3$, $m_i\neq m_j$ for any $i\neq j$.

Consider $\E(\eta_1)^3$ now. In this case
$\bbB_1=\bbB_2=\bbB_3=\bbA_1^{-1}(z)$ in (\ref{m24}). Note that
$\E(\bbA_1^{-1}(z))_{ii}=n^{-1}\E\tr\bbA_1^{-1}(z)$ is bounded. By
Lemma \ref{lem15}
\begin{equation}\label{m28}
n^{-3}\E(X_{11}^2-1)^3\sum\limits_{i=1}^p\E\big[\prod\limits_{m=1}^3(\bbB_m)_{ii}\big]=O(n^{-2}).
\end{equation}
From Lemma \ref{lem15} and H\"older's inequality we also have
$$
n^{-3}\big|\sum\limits_{i_1\neq
i_2}\E\big[(\bbB_{m_1})_{i_1i_1}(\bbB_{m_2})_{i_2i_2}(\bbB_{m_3})_{i_1i_2}\big]\big|\leq
Mn^{-3/2}(n^{-1}\E\tr\bbA_1^{-1}(z)\bbA_1^{-1}(\bar
z))^{1/2},
$$
which implies
\begin{equation}\label{m29}
(\ref{m25})=O(n^{-3/2}v^{-1}).
\end{equation}
Similarly, one may verify that
\begin{equation}\label{m30}
(\ref{m26*})=O(n^{-2}v^{-1}),\quad
(\ref{m26})=O( n^{-2}v^{-2}).
\end{equation}
This, together with Lemma \ref{lem15}, implies that
$$
(\ref{m27})=8 n^{-3}\sum\limits_{i_1=1}^p\sum\limits_{i_2=1}^p\sum\limits_{i_3=1}^p
 \E\big[
(\bbA_1^{-1}(z))_{i_1i_2}(\bbA_1^{-1}(z))_{i_1i_3}(\bbA_1^{-1}(z))_{i_2i_3}\big]+O(n^{-2}v^{-2})
$$
\begin{equation}
=8n^{-3}\sum\limits_{i=1}^p\E(\bbe_i^T\bbA_1^{-3}(z)\bbe_i)+O(n^{-2}v^{-2})=O(n^{-2}v^{-2}).\label{m31}
\end{equation}
Thus (\ref{m32}) follows from (\ref{m28})-(\ref{m31}).

Consider (\ref{m33}) next. Write
$$\E\big[\E_k (\eta_k^{(2)} )\E_k (\eta_k^{(2)}\eta_k)\big]=\E\big[\underline{\eta}_k^{(2)}\eta_k^{(2)}\eta_k\big],$$
where
$\underline{\eta}_k^{(2)}=\bbs_k^T\underline{\bbA}^{-2}_k(z)\bbs_k-n^{-1}\tr\underline{\bbA}^{-2}_k(z)$. In this case $\bbB_1=\underline{\bbA}^{-2}_k(z)$,
$\bbB_2=\bbA^{-2}_k(z)$ and $\bbB_3=\bbA_k^{-1}(z)$ in (\ref{m24}).
Applying Lemmas \ref{lem14} and \ref{lem15} we have
\begin{equation}\label{m34}
n^{-3} \E(X_{ki}^2-1)^3 \sum\limits_{i=1}^p\E\Big[\prod\limits_{m=1}^3(\bbB_m)_{ii}\Big]=O(n^{-2}v^{-2}).
\end{equation}
Similarly one may verify that
\begin{equation}
(\ref{m25})=O(n^{-3/2}v^{-3}),\quad
(\ref{m26*})=O(n^{-2}v^{-3}),\quad
(\ref{m26})=O(n^{-2}v^{-4}),
\end{equation}
where we use the fact that
$\E n^{-1}\tr\bbA^{-2}_k(z)\bbA^{-2}_k(\bar z)\leq1/v^3$. As in
(\ref{m31}) we obtain
$$
(\ref{m27})=8n^{-3}\sum\limits_{i=1}^p\E(\bbe_i^T\underline{\bbA}_1^{-2}(z)\bbA_1^{-3}(z)\bbe_i)+O(n^{-2}v^{-4})
=O(n^{-2}v^{-4}).
$$
These imply (\ref{m33}).
\end{proof}

\begin{lemma}\label{lem11}
$$
n^{-2}\E\big(\tr\bbA^{-1}(z)-\E\tr\bbA^{-1}(z)\big)^2
=n^{-4}b_1^2(z)\sum\limits_{k=1}^n\E\big[\tr\bbA_k^{-2}(z)\underline{\bbA}_k^{-2}(z)\big]+O(n^{-2}v^{-2}).
$$
\end{lemma}
\begin{proof}
Write \begin{equation}
\label{m11}
\beta_k=\beta_k^{\tr}-\beta_k\beta_k^{\tr}\eta_k(z).
\end{equation}
Let
$$
u_k=\big(\E_k-\E_{k-1}\big)(\beta_k\bbs_k^T\bbA_k^{-2}(z)\bbs_k).
$$
Applying (\ref{m11}) twice and referring to (\ref{g39}) we then
have
$$\E\big[\tr\bbA^{-1}(z)-\E\tr\bbA^{-1}(z)\big]^2=\sum\limits_{k=1}^n\E(u_k)^2
=\sum\limits_{k=1}^n\E\big(q_1+q_1+q_3+q_4+q_5+q_6\big),
$$
where
$$
q_1=\big[\E_k\big(\beta_k^{\tr}\eta_k^{(2)}\big)\big]^2,\
q_2=\big[\big(\E_k-\E_{k-1}\big)\big((\beta_k^{\tr})^2\bbs_k^T\bbA_k^{-2}(z)\bbs_k\eta_k\big)
\big]^2,
$$
$$
q_3=\big[\big(\E_k-\E_{k-1}\Big)\big(\beta_k(\beta_k^{\tr})^2\bbs_k^T\bbA_k^{-2}(z)\bbs_k\eta_k^2\big)
\big]^2
$$
$$
q_4=-2\E_k\big(\beta_k^{\tr}\eta_k^{(2)}\big)\big(\E_k-\E_{k-1}\big)\big((\beta_k^{\tr})^2\bbs_k^T\bbA_k^{-2}(z)\bbs_k\eta_k\big)
,
$$
$$
q_5=-2\big(\E_k-\E_{k-1}\big)\big((\beta_k^{\tr})^2\bbs_k^T\bbA_k^{-2}(z)\bbs_k\eta_k\big)
\big(\E_k-\E_{k-1}\big)\big(\beta_k(\beta_k^{\tr})^2\bbs_k^T\bbA_k^{-2}(z)\bbs_k\eta_k^2\big),
$$
and
$$
q_6=2\E_k\big(\beta_k^{\tr}\eta_k^{(2)}\big)\big(\E_k-\E_{k-1}\big)\big(\beta_k(\beta_k^{\tr})^2
\bbs_k^T\bbA_k^{-2}(z)\bbs_k\eta_k^2\big).
$$
It follows from  (\ref{m17}) and (\ref{h1}) that
$$
\big|n^{-2}\sum\limits_{k=1}^n\E q_3\big|\leq
Mn^{-2}v^{-2}\sum\limits_{k=1}^n
\E|(\beta_k^{\tr})^2\eta_k^2(z)|^2\leq M/(n^2v^2).
$$
Similar to (\ref{h1}), one may verify that
\begin{equation}
\E|\beta_k^{\tr}\eta_k^{(2)}|^8\leq M/(n^4v^{12}).\label{h32}
\end{equation}
We then conclude from (\ref{g1}), (\ref{h1}) (\ref{h32}), Lemmas
\ref{lem14} and \ref{lem1} and H\"older's inequality that
$$
\E|(\beta_k^{\tr})^2\bbs_k^T\bbA_k^{-2}(z)\bbs_k\eta_k|^4\leq M\E|(\beta_k^{\tr})^2\eta_k^{(2)}\eta_k|^4
$$\begin{equation}\label{h34}+M|\E n^{-1}\tr\bbA_k^{-2}(z)|^4E|(\beta_k^{\tr})^2\eta_k|^4
+M\E|\beta_k^{\tr}\Gamma_k^{(2)}\eta_k|^4\leq M/(n^2v^{4}).
\end{equation}
Similarly we have
\begin{equation}
\label{h33}
\E|\beta_k(\beta_k^{\tr})^2\bbs_k^T\bbA_k^{-2}(z)\bbs_k\eta_k^2|^2\leq M/(n^2v^3).
\end{equation}
In view of (\ref{h34}) and (\ref{h33})
$$
\E q_5=O(n^{-1}v^{-2}),\quad \E q_2=O(n^{-1}v^{-2}).
$$
As in (\ref{h34}) by (\ref{f44}), (\ref{g1}), Lemma \ref{lem1} we have
$$
\E|(\beta_k^{\tr})^3\bbs_k^T\bbA_k^{-2}(z)\bbs_k\eta_k^2|^2\leq M/(n^2v^3).
$$
By (\ref{m17}) and (\ref{g1})  we obtain
\begin{equation}\label{h39}
\E|\beta_k(\beta_k^{\tr})^3\bbs_k^T\bbA_k^{-2}(z)\bbs_k\eta_k^3|^2\leq M/(n^3v^5).
\end{equation}
These two estimates, together with (\ref{m11}) and (\ref{h32}), ensure that
\begin{equation}\label{h40}
\E q_6=O(n^{-1}v^{-2}).
\end{equation}

Write
\begin{equation}\label{h35}
\beta_k^{\tr}=b_1-b_1^2\Gamma_k+\beta_k^{\tr}b_1^2(\Gamma_k)^2.
\end{equation}
In view of (\ref{i1}), (\ref{h35}), Lemmas \ref{lem7}, \ref{lem14} and \ref{lem1} we may write
$$
n^{-2}\sum\limits_{k=1}^n\E q_1
=n^{-4}b_1^2(z)\sum\limits_{k=1}^n\E\big[\tr\bbA_k^{-2}(z)\underline{\bbA}_k^{-2}(z)\big]
$$$$+2n^{-3}b_1^3(z) \sum\limits_{k=1}^n\E\big[\big(n^{-1}\tr\bbA_k^{-2}(z)\underline{\bbA}_k^{-2}(z)
-n^{-1}\E(\tr\bbA_k^{-2}(z)\underline{\bbA}_k^{-2}(z))\big)
\Gamma_k\big]+O(n^{-2}v^{-2})
$$$$=n^{-4}b_1^2(z)\sum\limits_{k=1}^n\E\big[\tr\bbA_k^{-2}(z)\underline{\bbA}_k^{-2}(z)\big]+O(n^{-2}v^{-2}).
$$
Likewise, by (\ref{i1}), (\ref{h35}), (\ref{m33}) and Lemmas
\ref{lem14}, \ref{lem7} and \ref{lem1} we have
$$
n^{-2}\sum\limits_{k=1}^n\E q_4=-2n^{-2}b_1^3\sum\limits_{k=1}^n\E\big[\E_k\big(\eta_k^{(2)}\big)\big(\E_k-\E_{k-1}\big)
\big(\bbs_k^T\bbA_k^{-2}(z)\bbs_k\eta_k\big)\big]+O(n^{-2}v^{-2})
$$
$$
=-2n^{-2}b_1^3\sum\limits_{k=1}^n\E\big[\E_k\big(\eta_k^{(2)}\big)\E_k\big(\eta_k^{(2)}\eta_k\big)\big]
+2n^{-3}b_1^3\sum\limits_{k=1}^n\E\big[n^{-1}\tr\bbA_k^{-2}(z)
n^{-1}\tr\bbA_k^{-1}(z)\underline{\bbA}_k^{-2}(z)\big]
$$$$+O(n^{-2}v^{-2})=O(n^{-2}v^{-2}).
$$
\end{proof}

\begin{lemma}\label{lem13}
$$
n^{-3}\big|\E\big(\tr\bbA^{-1}(z)-\E\tr\bbA^{-1}(z)\big)^3\big|\leq M/(n^2v^2).
$$
\end{lemma}

\begin{proof}
 A direct calculation indicates that
\begin{equation}\E\big[\tr\bbA^{-1}(z)-\E\tr\bbA^{-1}(z)\big]^3=\sum\limits_{k=1}^n\E(u_{k})^3
+3\sum\limits_{k_1\neq k_2}^n\E(u_{k_1}^2u_{k_2}).\label{m35}
\end{equation}
Referring to the expressions of $q_1,q_2,q_3$ in the last lemma we have
$$
n^{-3}\big|\sum\limits_{k=1}^n\E\big[\big(\E_k-\E_{k-1}\big)(\beta_k\bbs_k^T\bbA_k^{-2}(z)\bbs_k)\big]^3\big|\leq Mn^{-3}\sum\limits_{k=1}^n\big[\E\big|\beta_k^{\tr}\eta_k^{(2)}\big|^3
$$
$$
+\E\big|(\beta_k^{\tr})^2\bbs_k^T\bbA_k^{-2}(z)\bbs_k\eta_k
\big|^3
+
\E\big|\beta_k(\beta_k^{\tr})^2\bbs_k^T\bbA_k^{-2}(z)\bbs_k\eta_k^2
\big|^3\big]\leq M/(n^2v^{3/2}),
$$
where the estimates can be obtained as in (\ref{h32}), (\ref{h34}) and (\ref{h39}).

 When $k_1<k_2$,
$$
\E(u_{k_1}^2u_{k_2})=0.
$$
When $k_1> k_2$ we have
\begin{equation}\label{m41}
\E(u_{k_1}^2u_{k_2})=\E\big[u_{k_2}E_{k_2}(u_{k_1}^2)\big].
\end{equation}
Here we reminder that the term $ \E_{k_2}(u_{k_1}^2)$ is similar to $\E(u_k^2)$, given in Lemma \ref{lem11}, except that the former is the conditional expectation and the later is the expectation. By the expressions of $q_j,j=1,2,3$ in Lemma \ref{lem11} we may write
\begin{equation}\label{m42}
u_{k_2}=u_{k_21}+u_{k_22}+u_{k_23},
\end{equation} where
$$u_{k_21}=-\big(\E_{k_2}-\E_{{k_2}-1}\big)\big((\beta_{k_2}^{\tr})^2\bbs_{k_2}^T\bbA_{k_2}^{-2}(z)\bbs_{k_2}\eta_{k_2}\big),
$$ and $$u_{k_22}=\big(\E_{k_2}-\E_{{k_2}-1}\big)\big(\beta_{k_2}
(\beta_{k_2}^{\tr})^2\bbs_{k_2}^T\bbA_{k_2}^{-2}(z)\bbs_{k_2}\eta_{k_2}^2\big),\
u_{k_23}=\big(\E_{k_2}-\E_{{k_2}-1}\big)\big(\beta_{k_2}^{\tr}\gamma_{k_2}\big).
$$
We claim that
\begin{equation}\label{m38}
n^{-3}\sum\limits_{k_1\neq k_2}^n\E\big[u_{k_2j}E_{k_2}(u_{k_1}^2)\big]=O(n^{-2}v^{-2}),\ j=1,2.
\end{equation}
Indeed, the estimates for $q_2,q_3,q_5,q_6$ involved in $\E_{k_2}(u_{k_1}^2)$ are straightforward by the argument from (\ref{h32}) to (\ref{h40})
 in Lemma \ref{lem11} for $\E(u_k^2)$. Here and below, $q_j,j=1,\cdots,6$ are obtained, respectively, from $\{q_j\}$ in Lemma \ref{lem11} with $k$ replaced by $k_1$. To deal with $q_1$,
from (\ref{h32}) and (\ref{h33}) we see that
$$
n^{-3}\sum\limits_{k_1\neq k_2}^n\E\big[u_{k_22}\E_{k_2}(q_1)\big]=O(n^{-2}v^{-2}).
$$
As for $u_{k_21}$, we use (\ref{i1}) and Lemma \ref{lem15} first to obtain
\begin{equation}\label{h45}
\E_{k-1}\big((\beta_{k}^{\tr})^2\gamma_{k}\eta_{k}\big)=\E_{k-1}\big( (\beta_{k}^{\tr})^2n^{-2}\tr\bbA_{k}^{-3}(z)\big)+O(n^{-2}v^{-3}).
\end{equation} By (\ref{i1}), (\ref{f10}), (\ref{h45}), (\ref{h34}), Lemmas \ref{lem15}, \ref{lem1}, \ref{lem7} and \ref{lem14} we then have
$$
\E\big[u_{k_21}\E_{k_2}(q_1)\big]=n^{-1}\E\big[u_{k_21}\E_{k_2}(\beta_{k_1}^{\tr}\underline{\beta}_{k_1}^{\tr}
n^{-1}\tr\bbA_{k_1}^{-2}(z)\underline{\bbA}_{k_1}^{-2}(z))\big]+O(n^{-1}v^{-2})
$$
$$
=n^{-1}\E\big[u_{k_21}\E_{k_2}(
\beta_{k_1k_2}^{\tr}\underline{\beta}_{k_1k_2}^{\tr}n^{-1}\tr\bbA_{k_1k_2}^{-2}(z)\underline{\bbA}_{k_1k_2}^{-2}(z))\big]
+O(n^{-1}v^{-2})
$$
$$
=n^{-2}\E\big[n^{-1}\tr\bbA_{k_1k_2}^{-3}(z)(\beta_{k_2}^{\tr})^2\E_{k_2}(
\beta_{k_1k_2}^{\tr}\underline{\beta}_{k_1k_2}^{\tr}n^{-1}\tr\bbA_{k_1k_2}^{-2}(z)\underline{\bbA}_{k_1k_2}^{-2}(z))\big]
+O(n^{-1}v^{-2})
$$
$$
=O(n^{-2}v^{-4}),
$$
where we use the fact that via (\ref{f10})
\begin{equation}\label{h49}
|n^{-1}\tr\bbA_{k_1}^{-2}(z)\underline{\bbA}_{k_1}^{-2}(z)-n^{-1}\tr\bbA_{k_1k_2}^{-2}(z)
\underline{\bbA}_{k_1k_2}^{-2}(z)\Big|\leq M/(nv^4).
\end{equation}
 To handle $q_4$, by (\ref{h33}), (\ref{h32}) and (\ref{h34}), it is straightforward to check that
$$n^{-3}\sum\limits_{k_1\neq k_2}^n\E\big[u_{k_22}\E_{k_2}(q_4)\big]=O(n^{-2}v^{-2}).$$
As for $u_{k_21}$,
it follows from (\ref{h34}), (\ref{h45}), Lemmas \ref{lem14} and \ref{lem7} that
$$
n^{-1}\E\big[u_{k_21}\E_{k_2}(q_4)\big]
=n^{-1}b_1^3 \E\big[u_{k_21}\E_{k_2}\big(\E_{k_1}(\gamma_{k_1})\E_{k_1}(\gamma_{k_1}\eta_{k_1})\big)\big]
+O(n^{-2}v^{-2})=O(n^{-2}v^{-2}),
$$
where we use the fact that the arguments for (\ref{m33}) are also applicable to $\E_{k_2}\big(\E_{k_1}(\gamma_{k_1})$ $\E_{k_1}(\gamma_{k_1}\eta_{k_1})\big)$.

Next consider $\E\big[u_{k_23}\E_{k_2}(u_{k_1}^2)\big]$. The strategy is to remove $\bbs_{k_2}$ from $\E_{k_2}(u_{k_1}^2)$ so that we make use of the fact that
\begin{equation}\label{h17}
\E\big[u_{k_23}\E_{k_2}(u_{k_14}^2)\big]=0
\end{equation}
by the fact that $u_{k_14}$ is independent of $\bbs_{k_2}$ with $u_{k_14}=(\E_{k_1}-\E_{k_1-1})(\beta_{k_2k_1}\bbs_{k_1}^T\bbA_{k_1k_2}^{-2}(z)\bbs_{k_1}\big)
$. To this end, write
$$
u_{k_1}=u_{k_11}+u_{k_12}+u_{k_13}+u_{k_14},
$$
where
$$
u_{k_11}=(\E_{k_1}-\E_{k_1-1})\big((\beta_{k_1}-\beta_{k_2k_1})\bbs_{k_1}^T\bbA_{k_1k_2}^{-2}(z)\bbs_{k_1}\big]
$$
$$
u_{k_12}=(\E_{k_1}-\E_{k_1-1})\big[(\beta_{k_1}-\beta_{k_2k_1})\bbs_{k_1}^T(\bbA_{k_1}^{-2}(z)
-\bbA_{k_1k_2}^{-2}(z))\bbs_{k_1}\big]
$$
$$
u_{k_13}=(\E_{k_1}-\E_{k_1-1})\big(\beta_{k_2k_1}\bbs_{k_1}^T(\bbA_{k_1}^{-2}(z)-\bbA_{k_1k_2}^{-2}(z))\bbs_{k_1}\big).
$$

 We now substitute $u_{k_11}+u_{k_12}+u_{k_13}+u_{k_14}$ for $u_{k_1}$ in $\E\big[u_{k_23}\E_{k_2}(u_{k_1}^2)\big]$ and evaluate them one by one besides using (\ref{h17}). By (\ref{b22}), (\ref{m17}), Lemmas \ref{lem14} and \ref{lem1} we have
\begin{equation}\label{h21}
\E|\beta_{k_1}\bbs_{k_1}^T\bbA_{k_1}^{-2}(z)\bbs_{k_1}|^8 \leq M\E|\gamma_{k_1}|^8+M|n^{-1}\E\bbA_{k_1}^{-2}(z))|^8+Mv^{-8}\E|\xi_{k_1}|^8
+M\E|\Gamma_1^{(2)}|^8
\leq M/v^4,
\end{equation}
and via (\ref{f7})
\begin{eqnarray}\label{h21*}
 &\E|(\bbs_{k_1}^T\bbA_{k_1k_2}^{-1}(z)\bbs_{k_2})^2\beta_{k_2k_1}\beta_{k_1k_2}|^4\leq M(\E|\zeta_1(z)\beta_{k_1k_2}|^8E|\beta_{k_2k_1}|^8)^{1/2}\non
 &+Mn^{-4}v^{-4}\E|\beta_{k_2k_1}|^4
 \leq Mn^{-4}v^{-4},
\end{eqnarray}
where $\zeta_1(z)=\centre^{k_1}\big(\bbs_{k_1}^T\bbA_{k_1k_2}^{-1}(z)\bbs_{k_2}\bbs_{k_2}^T\bbA_{k_1k_2}^{-1}(z)\bbs_{k_1}\big)
$.
Combining (\ref{h21}), (\ref{h21*}) and (\ref{h32}) we obtain
$$
\E\big[u_{k_23}\E_{k_2}(u_{k_11}+u_{k_12})^2\big]=O(n^{-2}v^{-7/2}).
$$
As in (\ref{h21*}) one may verify that
\begin{equation}\label{h41}
\E|\bbs_{k_1}^T\bbA_{k_1k_2}^{-1}(z)\bbs_{k_2}\bbs_{k_2}^T\bbA_{k_1k_2}^{-2}(z)\bbs_{k_1}\beta_{k_1k_2}|^4
\leq M/(nv^2).
\end{equation}
Thus from (\ref{h21}), (\ref{h21*}), (\ref{h32}), (\ref{h41}), (\ref{i1}), (\ref{f7}) and Lemmas \ref{lem14}, \ref{lem7}, \ref{lem1} we obtain
$$
\E\big[u_{k_23}\E_{k_2}(u_{k_13})^2\big]
$$$$=2\E\big[u_{k_23}\E_{k_2}\big((\E_{k_1}-\E_{k_1-1})(\beta_{k_2k_1}\bbs_{k_1}^TG_{k_1k_2}(z)
\bbs_{k_1}\beta_{k_1k_2})\big)^2\big]+O(n^{-2}v^{-7/2})
$$
$$
=2b_{12}(z)\E\big[u_{k_23}E_{k_2}\big(\E_{k_1}\big((\bbs_{k_1}^TG_{k_1k_2}(z)\bbs_{k_1}-\frac{1}{n}\tr G_{k_1k_2}(z))\beta_{k_1k_2}\big)\big)^2\big]+O(n^{-2}v^{-7/2})
$$
$$
=2b_{12}(z)n^{-2}\E\big[u_{k_23}\E_{k_2}\big(\beta_{k_1k_2}\bbs_{k_2}^T\bbA_{k_1k_2}^{-2}(z)
\underline{\bbA}_{k_1k_2}^{-2}(z)\bbs_{k_2}\underline{\beta}_{k_1k_2}\bbs_{k_2}^T
\underline{\bbA}_{k_1k_2}^{-1}(z)\bbA_{k_1k_2}^{-1}(z)\bbs_{k_2}\big)\big]
$$$$+O(n^{-2}v^{-7/2})
=O(n^{-2}v^{-4}),
$$
where $G_{k_1k_2}(z)=\bbA_{k_1k_2}^{-1}(z)\bbs_{k_2}\bbs_{k_2}^T\bbA_{k_1k_2}^{-2}(z)$. In view of (\ref{h21}), (\ref{h21*}) and (\ref{h32}) we also conclude that
 $$
\E\big[u_{k_23}\E_{k_2}\big((u_{k_11}+u_{k_12})u_{k_13}\big)\big]=O(n^{-1}v^{-2}),
 $$
 because via (\ref{h41}), (\ref{h21*}) and (\ref{h21*})
 \begin{equation}\label{h44}
\E|\beta_{k_2k_1}\bbs_{k_1}^T(\bbA_{k_1}^{-2}(z)-\bbA_{k_1k_2}^{-2}(z))\bbs_{k_1}|^2\leq M.
 \end{equation}
Likewise, by (\ref{h21}), (\ref{h21*}), (\ref{h32}) and (\ref{h44}) we have
$$
\E\big[u_{k_23}E_{k_2}(u_{k_12}u_{k_14})\big]=O(n^{-1}v^{-2}).
$$

Moreover by (\ref{b22}), (\ref{h21}), (\ref{h21*}), (\ref{h32}), (\ref{i1}) and (\ref{m17}), we have
$$
\E\big[u_{k_23}\E_{k_2}(u_{k_11}u_{k_14})\big]
$$
$$
=b_1\E\big[u_{k_23}\E_{k_2}\big((\E_{k_1}-\E_{k_1-1})
(\beta_{k_2k_1}\rho_{k_1k_2}\bbs_{k_1}^T\bbA_{k_1k_2}^{-2}(z)\bbs_{k_1})u_{k_14}\big)\big]+O(n^{-1}v^{-2})
$$
$$
=n^{-2}b_1\E\big[(\E_{k_1}-\E_{k_1-1})\big(\beta_{k_2k_1}\bbs_{k_1}^T\bbA^{-1}_{k_1k_2}(z)\E_{k_2}
\big(\beta_{k_2}^{\tr}\bbA^{-1}_{k_1k_2}(z)\big)\bbA^{-1}_{k_1k_2}(z)\bbs_{k_1}\bbs_{k_1}^T
\bbA_{k_1k_2}^{-2}(z)\bbs_{k_1}\big)
$$$$\times(u_{k_14})\big]+O(n^{-2}v^{-4})=O(n^{-2}v^{-4}),
$$
where
$$
\rho_{k_1k_2}=\centre^{k_2}\big(\bbs_{k_2}^T\bbA^{-1}_{k_1k_2}(z)\bbs_{k_1}\bbs_{k_1}^T\bbA^{-1}_{k_1k_2}(z)
\bbs_{k_2}\big)
$$
and we also use the fact that
\begin{equation}\label{h22}
\E\big[u_{k_23}\E_{k_2}\big((\E_{k_1}-\E_{k_1-1})
(\beta_{k_2k_1}n^{-1}\bbs_{k_1}^T\bbA^{-2}_{k_1k_2}(z)\bbs_{k_1}\bbs_{k_1}^T
\bbA_{k_1k_2}^{-2}(z)\bbs_{k_1})u_{k_14}\big)\big]=0.
\end{equation}

As for $\E\big[u_{k_23}\E_{k_2}(u_{k_13}u_{k_14})\big]$, on the one hand, using  (\ref{i1}) twice and Lemma \ref{lem1} we obtain
$$
\E\big[u_{k_23}\E_{k_2}\big((\E_{k_1}-\E_{k_1-1})\big(b_{12}\bbs_{k_1}^T\bbA^{-1}_{k_1k_2}(z)\bbs_{k_2}\bbs_{k_2}^T
\bbA^{-2}_{k_1k_2}(z)\bbs_{k_1}\big)u_{k_14}\big)\big]
$$
$$
=\E\big[u_{k_23}\E_{k_2}\big((\E_{k_1}-\E_{k_1-1})\big(b_{12}\rho_{k_1k_2}^{(2)}\big)u_{k_14}\big)\big]
$$
$$
=n^{-2}b^2_{12}\E\big[(\E_{k_1}-\E_{k_1-1})\big(\bbs_{k_1}^T\bbA^{-1}_{k_1k_2}\E_{k_2}
(\bbA^{-2}_{k_1k_2})\bbA^{-2}_{k_1k_2}\bbs_{k_1}\big)u_{k_14}\big]+O(n^{-2}v^{-4})
$$
$$
=n^{-2}b^3_{12}\E\big[(\E_{k_1}-\E_{k_1-1})\big(\bbs_{k_1}^T\bbA^{-1}_{k_1k_2}\E_{k_2}(\bbA^{-2}_{k_1k_2})
\bbA^{-2}_{k_1k_2}\bbs_{k_1}\big)\E_{k_1}(\gamma_{k_2k_1})\big]+O(n^{-2}v^{-4})
$$
$$
=n^{-3}b^3_{12}\E\big[n^{-1}\tr\bbA^{-1}_{k_1k_2}\E_{k_2}(\bbA^{-2}_{k_1k_2})\bbA^{-2}_{k_1k_2}
\E_{k_1}(\bbA^{-2}_{k_1k_2})\big]+O(n^{-2}v^{-4})=O(n^{-2}v^{-4}),
$$
where we also use an equality similar to (\ref{h22}),
$$
\rho_{k_1k_2}^{(2)}=\centre^{k_2}\big(\bbs_{k_2}^T\bbA^{-2}_{k_1k_2}(z)\bbs_{k_1}\bbs_{k_1}^T
\bbA^{-1}_{k_1k_2}(z)\bbs_{k_2}\big)
$$
and
$$
\gamma_{k_2k_1}=\centre^{k_1}\big(\bbs_{k_1}^T\bbA^{-2}_{k_1k_2}(z)\bbs_{k_1}\big).
$$
Similarly one may verify that
$$
\E\big[u_{k_23}\E_{k_2}\big((\E_{k_1}-\E_{k_1-1})\big(\beta_{k_1k_2}\bbs_{k_1}^T
\bbA^{-1}_{k_1k_2}(z)\bbs_{k_2}\bbs_{k_2}^T\bbA^{-2}_{k_1k_2}(z)\bbs_{k_1}\big)
\E_{k_1}(\gamma_{k_2k_1}b_{12})\big)\big]=O(n^{-2}v^{-4}).
$$
Moreover by H\"older's inequality, (\ref{h32}), (\ref{h41}) and Lemma \ref{lem1}
$$
\E\big[u_{k_23}\E_{k_2}\big((\E_{k_1}-\E_{k_1-1})\big(b_{12}\beta_{k_1k_2}\xi_{k_1k_2}
\bbs_{k_1}^T\bbA^{-1}_{k_1k_2}(z)\bbs_{k_2}\bbs_{k_2}^T\bbA^{-2}_{k_1k_2}(z)\bbs_{k_1}\big)
$$$$\times(\E_{k_1}-\E_{k_1-1})(b_{12}\beta_{k_1k_2}\xi_{k_1k_2})\big]=O(n^{-2}v^{-4}).
$$
Via (\ref{f13}), these ensure that
$$
\E\big[u_{k_23}\E_{k_2}\big((\E_{k_1}-\E_{k_1-1})\big(\beta_{k_1k_2}\bbs_{k_1}^T
\bbA^{-1}_{k_1k_2}(z)\bbs_{k_2}\bbs_{k_2}^T\bbA^{-2}_{k_1k_2}(z)\bbs_{k_1}\big)u_{k_14}\big)\big]=O(n^{-2}v^{-4}).
$$

On the other hand, apparently from H\"older's inequality, (\ref{h21}), (\ref{h21*}), (\ref{h41}), (\ref{h32}) and (\ref{h1}) we obtain
$$
\E\big[u_{k_23}\E_{k_2}\big((\E_{k_1}-\E_{k_1-1})
\big(\beta^2_{k_1k_2}(\bbs_{k_1}^T\bbA^{-1}_{k_1k_2}(z)\bbs_{k_2})^2\bbs_{k_2}^T\bbA^{-2}_{k_1k_2}(z)\bbs_{k_1}\big)
$$
$$\times(\E_{k_1}-\E_{k_1-1})\big(\beta_{k_2k_1}\beta^{\tr}_{k_2k_1}\eta_{k_2k_1}\bbs_{k_1}^T
\bbA_{k_1k_2}^{-2}(z)\bbs_{k_1}\big)\big]=O(n^{-2}v^{-4}).
$$
Similar to (\ref{m24}) we obtain
\begin{equation}\label{h23}
\E_{k_2}\big[\E_{k_1}(\rho_{k_2k_1}\beta^2_{k_1k_2}\bbs_{k_2}^T\bbA_{k_1k_2}^{-2}(z)\bbs_{k_1})
\E_{k_1}(\beta^{\tr}_{k_1k_2}\gamma_{k_2k_1})\big]
\end{equation}
$$
=n^{-5/2}\big[\E[(X_{11}^2-1)^2X_{11}]\sum\limits_{i=1}^p\E_{k_2}\big((\bbB_1)_{ii}(\bbB_2)_{ii}y_i\big)
$$$$+2\E X_{11}^3\sum\limits_{i_1\neq j_1}\E_{k_2}\big((\bbB_1)_{i_1j_1}
[(\bbB_2)_{j_1j_1}y_{i_1}+(\bbB_2)_{i_1i_1}y_{j_1}]\big)
$$
$$
+2\E X_{11}^3\sum\limits_{i_1\neq j_1}\E_{k_2}\big((\bbB_1)_{i_1j_1}(\bbB_2)_{i_1j_1}(y_{i_1}+y_{j_1})\big)\big],
$$
where
$$
\bbB_1=\beta^2_{k_1k_2}\bbA_{k_1k_2}^{-1}\bbs_{k_2}\bbs_{k_2}^T\bbA_{k_1k_2}^{-1},\ \bbB_2=\bbA_{k_1k_2}^{-2},\ y_i=\E_{k_1}(\beta^{\tr}_{k_1k_2}\bbs_{k_2}^T\bbA_{k_1k_2}^{-2}\bbe_i).
$$
We obtain from Lemma \ref{lem1}, Lemma \ref{lem15} and Burkholder's inequality
$$
\E|\beta_{k_1k_2}(\bbs_{k}^T\bbA_{k_1k_2}^{-1}\bbe_i)^2|^4\leq
M(\E|\beta_{k_1k_2}|^8\E|\bbs_{k}^T\bbA_{k_1k_2}^{-1}\bbe_i|^{16})^{1/2}\leq M/(n^4v^7).
$$
Similarly
\begin{equation}\label{h43}
\E|\beta_{k_1k_2}y_i|^4\leq M(\E|\beta_{k_1k_2}|^8\E|\beta^{\tr}_{k_1k_2}\bbs_{k_2}^T\bbA_{k_1k_2}^{-2}\bbe_i|^8)^{1/2}\leq M/(nv^{11/2}).
\end{equation}
These, Lemma \ref{lem15} and (\ref{h32}) imply that
$$
\E\big[u_{k_23}n^{-5/2}\sum\limits_{i=1}^p\E_{k_2}\big((\bbB_1)_{ii}(\bbB_2)_{ii}y_i\big)\big]=O(n^{-2}v^{-4}).
$$
 Note that
$$
n^{-5/2}\sum\limits_{i_1\neq j_1}(\bbB_1)_{ij}(\bbB_2)_{j_1j_1}y_{i_1}
=n^{-5/2}\beta^2_{k_1k_2}\bbs_{k_2}^T\bbA_{k_1k_2}^{-1}
\E_{k_1}(\beta^{\tr}_{k_1k_2}\bbA_{k_1k_2}^{-2})\bbs_{k_2}
$$
$$\times\sum\limits_{j}\bbs_{k_2}^T\bbA_{k_1k_2}^{-1}\bbe_j(\bbB_2)_{jj}+n^{-5/2}\sum\limits_{i_1\neq j_1}(\bbB_1)_{ij}(\bbB_2)_{j_1j_1}y_{i_1}.
$$
By Lemma \ref{lem1}, Lemma \ref{lem15} and Burkholder's inequality we also have
\begin{equation}\label{h42}
\E|\beta_{k_1k_2}\bbs_{k}^T\bbA_{k_1k_2}^{-1}\bbe_i|^4\leq
M(\E|\beta_{k_1k_2}|^8\E|\bbs_{k}^T\bbA_{k_1k_2}^{-1}\bbe_i|^{8})^{1/2}\leq M/(n^4v^6)
\end{equation}
and via (\ref{f7}), (\ref{f44})
$$
\E|\beta_{k_1k_2}\bbs_{k_2}^T\bbA_{k_1k_2}^{-1}E_{k_1}(\beta^{\tr}_{k_1k_2}\bbA_{k_1k_2}^{-2})\bbs_{k_2}|^8\leq Mv^{-4}
\E\big(|\sqrt{|\beta_{k_1k_2}|}\|\E_{k_1}(\beta^{\tr}_{k_1k_2}\bbA_{k_1k_2}^{-2})\bbs_{k_2}\||^8\big)
$$$$\leq
Mv^{-4} \big( \E|\beta_{k_1k_2}|^8\E|(\beta^{\tr}_{k_1k_2})^2\bbs_{k_2}\bbA_{k_1k_2}^{-2}(z)\bbA_{k_1k_2}^{-2}(\bar z)\bbs_{k_2}|^8\big)^{1/2}\leq M/v^{16}
$$
These, together with (\ref{h32}),  Lemmas \ref{lem1} and \ref{lem15}, ensure that
$$
\E\big[u_{k_23}n^{-5/2}\sum\limits_{i_1\neq j_1}(\bbB_1)_{ij}(\bbB_2)_{j_1j_1}y_{i_1}\big]=O(n^{-2}v^{-4}).
$$
This argument also works for the remaining terms in (\ref{h23}) so that
$$
\E\big[u_{k_23}\times(\ref{h23})\big]=O(n^{-2}v^{-4}).
$$
As in (\ref{i1}), a direct calculation, together with (\ref{h42}), (\ref{h32}) and an estimate similar to (\ref{h43}), yields
$$\E\big[u_{k_23}\E_{k_2}\big(\E_{k_1-1}(\rho_{k_2k_1}\beta^2_{k_1k_2}\bbs_{k_2}^T\bbA_{k_1k_2}^{-2}(z)\bbs_{k_1})
\E_{k_1}(\beta^{\tr}_{k_1k_2}\gamma_{k_1k_2})\big)\big]$$
$$
=n^{-3/2} \E(X_{11}^3-X_{11}) \sum\limits_{i=1}^p\E\big[u_{k_23}E_{k_2}
\big(\E_{k_1-1}\big(\beta^2_{k_1k_2}(\bbs_{k_2}^T\bbA^{-1}_{k_1k_2}\bbe_i)^2\bbe_i^T\bbA^{-2}_{k_1k_2}\bbs_{k_2}\big)
\E_{k_1}(\beta^{\tr}_{k_1k_2}\gamma_{k_1k_2})\big)\big]
$$
$$
=O(n^{-2}v^{-4}).
$$These ensure
$$
\E\big[u_{k_23}\E_{k_2}\big((\E_{k_1}-\E_{k_1-1})
\big(\beta^2_{k_1k_2}(\bbs_{k_1}^T\bbA^{-1}_{k_1k_2}(z)\bbs_{k_2})^2\bbs_{k_2}^T\bbA^{-2}_{k_1k_2}(z)\bbs_{k_1}\big)
u_{k_14}\big)\big]=O(n^{-2}v^{-4}).
$$
Hence
$$
\E\big[u_{k_23}\E_{k_2}(u_{k_13}u_{k_14})\big]=O(n^{-2}v^{-4}).
$$
Thus the proof is completed.
\end{proof}

\section{The limit of mean function}\label{lim-mean}

The aim in the section is to find the limit of
$$
\frac{1}{2\pi i}\oint K(\frac{x-z}{h})n(\E m_n(z)-m_n^0(z))dz.
$$
It is thus sufficient to investigate the uniform convergence
$nh(\E\underline {m}_n(z)-\underline{m}_n^0(z))$ on the contour. In order to establish Theorem \ref{theo2} we instead apply the estimates in the last section to investigate
$n(\E\underline {m}_n(z)-\underline{m}_n^0(z))$.

Write
$$
\bbA-z\bbI=\sum_{j=1}^n\bbs_j\bbs_j^T-z\bbI.
$$
Multiplying both sides by $\bbA^{-1}(z)$, taking the trace and dividing by $n$ we obtain
\begin{equation}\label{m4}
c_n+zc_nm_n(z)=1-\frac{1}{n}\sum\limits_{j=1}^n\beta_j(z)
\end{equation}
(one may see the equality above (2.2) of \cite{s3}).
Taking expectation on both sides of the equality above and applying (\ref{b22}) we have
\begin{equation}
\label{m1}
c_n+zc_n \E m_n(z)=1-b(z)+b(z)D_n,
\end{equation}
where $$D_n=\E\big[\beta_1(z)(\bbs_1^T\bbA_1^{-1}(z)\bbs_1-\E n^{-1}\tr\bbA^{-1}(z))\big].$$
On the other hand, it follows from (\ref{m2}) and Lemma \ref{lem8} that
\begin{equation}
\label{m3}
c_n+zc_nm_n^0(z)=1-\big(1+c_nm_n^0(z)\big)^{-1}.
\end{equation}
 Taking the difference between (\ref{m1}) and (\ref{m3}), along with (\ref{m3}), yields
\begin{equation}\label{m8}
n\big(\E m_n(z)-m_n^0(z)\big)=nD_n/\big(c_ng(z)\big),
\end{equation}
where $g(z)$ is defined in Lemma \ref{lem11a}.

Considered $z\in\gamma_1\cup\gamma_2$ first. Applying (\ref{b22}) and (\ref{b23}) yields
\begin{eqnarray}\label{b28}
&&\E\big(\tr\bbA_1^{-1}(z)\big)-\E\big(\tr\bbA^{-1}(z)\big)=\E\big(\beta_1\bbs_1^T\bbA_1^{-2}(z)\bbs_1\big)
\non
&=&b_1\E\big([1-b_1\xi_1+b_1\beta_1\xi_1^2(z)]\bbs_1^T\bbA_1^{-2}(z)\bbs_1\big)\non
&=&b_1\E n^{-1}\tr
\bbA_1^{-2}(z)-d_{n1}+d_{n2}+d_{n3},
\end{eqnarray}
where
$$
d_{n1}=b_1^2\E\big[\eta_1\eta_1^{(2)}\big],\
d_{n2}= b_1^2\E\big[\Gamma_1\bbs_1^T\bbA_1^{-2}(z)\bbs_1\big]
=b_1^2\E\big[\Gamma_1\Gamma_1^{(2)}\big],\
d_{n3}=b_1\E\big[\beta_1\xi_1^2\bbs_1^T\bbA_1^{-2}\bbs_1\big].
$$
It follows from (\ref{f7}) and Lemma \ref{lem1} that
$$
|d_{nj}|\leq M/(nv^2),\quad j=1,2,3,
$$
which implies
\begin{equation}
\label{m14}
\E\big(\tr\bbA_1^{-1}(z)\big)-\E\big(\tr\bbA^{-1}(z)\big)=b_1\E n^{-1}\tr
\bbA_1^{-2}(z)+O(n^{-1}v^{-2}).
\end{equation}

Next by (\ref{b22}) \begin{eqnarray}
&n\E\big[\beta_1\bbs_1^T\bbA_1^{-1}(z)\bbs_1\big]-\E(\beta_1)\E\big(\tr\bbA_1^{-1}(z)\big)\label{f27}\\
&=-nb_1^2\E\big[\xi_1\bbs_1^T\bbA_1^{-1}(z)\bbs_1\big]+nb^2_1\E\big[\beta_1\xi_1^2\bbs_1^T\bbA_1^{-1}(z)\bbs_1\big]
-b^2_1\E(\beta_1\xi_1^2)\E\big[\tr\bbA_1^{-1}(z)\big]\non
&=f_{n1}+f_{n2}+f_{n3}+f_{n4},\nonumber
\end{eqnarray}
where
$$
f_{n1}=-nb_1^2\E\eta_1^2, \
f_{n2}=-nb_1^2\E\big(\Gamma_1\bbs_1^T\bbA_1^{-1}(z)\bbs_1\big)
=nb_1^2\E(\Gamma_1)^2, \ f_{n3}=nb^2_1\E\big(\beta_1\xi_1^2\eta_1\big),
$$
and
$$
f_{n4}=b_1^2\big[\E\big(\beta_1\xi_1^2\tr\bbA_1^{-1}\big)-\E(\beta_1\xi_1^2)\E\tr\bbA_1^{-1}\big].
$$

By (\ref{i1}) we have
\begin{equation}
f_{n1}=-nb_1^2\E\eta_1^2=-2b_1^2\E n^{-1}\tr\bbA_1^{-2}+O(n^{-1}v^{-2}).\label{g10}
\end{equation}
By Lemma \ref{lem11} and (\ref{h48**})
\begin{equation}\label{g11}
f_{n2}/g(z)=O(n^{-1}v^{-2}|z+c_n-1+2zc_nm_n^0(z)|^{-1}),
\end{equation}
where we use the fact that via (\ref{m12}) and (\ref{m13})
\begin{equation}\label{g12}
|g(z)|\geq M_2|z+c_n-1+2zc_nm_n^0(z)|, \ M_2>0.
\end{equation}

Consider $f_{n4}$ next. Apply (\ref{b22}) to further write $f_{n4}$ as
$$
f_{n4}=f_{n41}+f_{n42}+f_{n43},
$$
where
$$
f_{n41}=nb_1^3\E(\eta_1^2\Gamma_1),\
f_{n42}=nb_1^3\E(\Gamma_1)^3,
$$
and $$
f_{n43}=-b_1^3\big[\E\big(\beta_1\xi_1^3\tr\bbA_1^{-1}(z)\big)-\E(\beta_1\xi_1^3)\E\tr\bbA_1^{-1}(z)\big].
$$
By H\"older's inequality and Lemma \ref{lem1} we obtain
$$
|f_{n43}|\leq nM\big(\E|\beta_1|^2\E\big|\xi_1^3\Gamma_1\big|^2\big)^{1/2}\leq M/(nv^2).
$$
From (\ref{i1}), Lemmas \ref{lem15} and \ref{lem1} we conclude that
$$
f_{n41}=2b_1^3\E(\Gamma_1\Gamma_1^{(2)})+O(n^{-1}v^{-2})=O(n^{-1}v^{-2}).
$$
 In view of Lemma \ref{lem13} we also have $f_{n42}=O(n^{-1}v^{-2})$. Therefore $f_{n4}=O(n^{-1}v^{-2})$.

By (\ref{b22}) $f_{n3}$ may be further written as
$$
f_{n3}=f_{n31}+f_{n32}+f_{n33},
$$
where
$$
f_{n31}=nb^3_1\E (\eta_1^3 ),\ f_{n32}=2nb^3_1\E (\eta_1^2\Gamma_1 ),\
f_{n33}=nb^3_1\E (\beta_1\xi_1^3\eta_1 ).
$$
Note that $f_{n32}=2f_{n41}$. Lemmas \ref{lem1} and \ref{lem12} ensure, respectively, $f_{n33}=O(n^{-1}v^{-2})$ and
$f_{n31}=O(n^{-1}v^{-2}).$ We then conclude that $f_{n3}=O(n^{-1}v^{-2}).$

Summarizing the above argument (particularly (\ref{m8}), (\ref{m14}) and (\ref{g10}))  we obtain
\begin{equation}\label{f38}
nD_n/g(z)
=-\big(b_1^2/g(z)\big)\E n^{-1}\tr
\bbA_1^{-2}(z)+O(n^{-1}v^{-2}|z+c_n-1+2zc_nm_n^0(z)|^{-1}).
\end{equation}

We would point out that (\ref{f38}), Lemmas \ref{lem11a} and \ref{lem14} imply
\begin{proposition}\label{prop1}
For $v> M/\sqrt{n}$ and $u\in [a,b]$,
\begin{equation}\label{g14}
|\E m_n(z)-m(z)|\leq M/(nv).
\end{equation}
\end{proposition}

From (\ref{f33*}) we have
\begin{equation}\label{g13}
n^{-1}\E\big[\tr\bbA_1^{-2}(z)\big]
=\frac{c_n}{z^2(1+\underline{m}_n^0(z))^2}\big(1-\frac{c_n(\underline{m}_n^0(z))^2}{(1+\underline{m}_n^0(z))^2}\big)^{-1}
+O(n^{-1}v^{-2}).
\end{equation}
We then conclude from (\ref{m8}), (\ref{g14}), (\ref{f38}), (\ref{g16}), (\ref{g36}) and (\ref{g13}) that
\begin{eqnarray}
n(\E\underline{m}_n(z)-\underline{m}_n^0(z))&=&\frac{c_n(\underline{m}_n^0(z))^3}{(1+\underline{m}_n^0(z))^3}
\big(1-\frac{c_n(\underline{m}_n^0(z))^2}{(1+\underline{m}_n^0(z))^2}\big)^{-2}\label{g32} \\
&&+O(n^{-1}v^{-2}|z+c_n-1+2zc_nm_n^0(z)|^{-1}). \nonumber
\end{eqnarray}
The case where $z$ lies in the vertical lines on the contour can be
handled similarly as pointed out in the last section with the truncation version of $X_n(z)$.

In view of (\ref{h19}) it remains to find the limit of the following
\begin{equation}\label{f50}
\frac{1}{4\pi i}\oint
K(\frac{x-z}{h})\frac{c_n(\underline{m}_n^0(z))^3}{(1+\underline{m}_n^0(z))^3}
\big(1-\frac{c_n(\underline{m}_n^0(z))^2}{(1+\underline{m}_n^0(z))^2}\big)^{-2}dz,
\end{equation}
which is done in Appendix 3.

\section{The proof of Theorems \ref{rem2} and \ref{theo2}}\label{proof-31}

{\bf Proof of Theorem \ref{rem2}.} Let $x\in (a,b)$. We claim that
\begin{eqnarray*}
&&nh\Big[h^{-1}\int^b_a
K(\frac{x-y}{h})d\F_{c_n}(y)-f_{c_n}(x)\Big]\\
&=&nh\Big[\int^{\frac{x-a}{h}}_{\frac{x-b}{h}}
K(y)f_{c_n}(x-yh)dy-f_{c_n}(x)\Big]\\
&=&nh\Big[f_{c_n}(x)\int^{\frac{x-a}{2h}}_{\frac{x-b}{2h}}
K(y)dy-f_{c_n}(x)\Big] +\text{remainder},
\end{eqnarray*}
where
$$
|\text{remainder}| \leq 4nh^3\big((x-a)^{-2}+(b-x)^{-2}\big)(\|f\|+M)
\int y^2|K(y)| dy \to 0 \ \ \text{as} \ \ n\to \infty.
$$
Indeed, by Taylor's expansion
$$
f_{c_n}(x-yh)=f_{c_n}(x)-f'_{c_n}(x)yh+f''_{c_n}(x-\theta yh)(yh)^2,
$$
where $\theta \in [0,1]$. Moreover note that
$$
nh^2\Big|\int^{+\infty}_{\frac{x-a}{2h}}yK(y)dy+\int^{\frac{x-b}{2h}}_{-\infty}
yK(y)dy\Big|\leq 4nh^3\big((x-a)^{-2}+(b-x)^{-2}\big) \int
y^2|K(y)|dy\rightarrow 0,$$
$$  nh\Big[1-\int^{\frac{x-a}{h}}_{\frac{x-b}{h}} K(y)dy\Big]\leq 4nh^3\big((x-a)^{-2}+(b-x)^{-2}\big) \int y^2|K(y)|dy \rightarrow 0,
$$
and $f''_{c_n}(x-\theta yh)$ is bounded above by a finite constant
depending only on $x$ when $y\in \big((x-b)/(2h), (x-a)/(2h)\big)$.
Thus the proof is complete. \qed

{\bf Proof of Theorem \ref{theo2}.}
Following the truncation steps in \cite{b2} we could
truncate and re-normalize the random variables so that
\begin{equation}\label{f46}
|X_{ij}|\leq \tau_nn^{1/2},\ \E X_{ij}=0,\  \E X_{ij}^2=1,
\end{equation}
where $\tau_nn^{1/3}\rightarrow\infty$ and $\tau_n\rightarrow 0$. Based on this one may then
verify that
\begin{equation}\label{f49}
\E X_{11}^4=3+O(\frac{1}{n}).
\end{equation}

 For any finite constants
$l_1,\cdots,l_d$, by Cauchy's theorem and Fubini's theorem we write
\begin{equation}
\frac{n}{\sqrt{\ln h^{-1} }}\sum\limits_{j=1}^d
l_j\Big(F_n(x_j)-\int^{x_j}_{-\infty}\frac{1}{h}\int
K(\frac{t-y}{h})d\F_{c_n}(y)dt\Big)\label{g42}
\end{equation}
$$=\frac{n}{\sqrt{\ln  h^{-1}
}}\sum\limits_{j=1}^d
l_j\Big(\int^{x_j}_{-\infty}f_n(t)dt-\int^{x_j}_{-\infty}\frac{1}{h}\int
K(\frac{t-y}{h})d\F_{c_n}(y)dt\Big)
$$
$$
=-\frac{n}{2h\pi i\sqrt{\ln
 h^{-1}}}\sum\limits_{j=1}^d l_j(\int^{x_j}_{-\infty}\oint_{\mathcal{C}_1}
K(\frac{t-z}{h})(\tr\bbA^{-1}(z)-nm^0_n(z))dzdt
$$
$$
=-\frac{n}{2h\pi i\sqrt{\ln
 h^{-1}}}\sum\limits_{j=1}^d l_j\oint_{\mathcal{C}_1}
\int^{x_j}_{-\infty}K(\frac{t-z}{h})dt(\tr\bbA^{-1}(z)-nm^0_n(z))dz,
$$
where the contour ${\mathcal{C}_1}$ is defined as before.

Furthermore, we conclude from (\ref{g39}) and integration by parts
that
$$
\frac{1}{2h\pi i\sqrt{\ln  h^{-1}}}\oint_{\mathcal{C}_1}
\int^{x}_{-\infty}K(\frac{t-z}{h})dt(\tr\bbA^{-1}(z)-\E\tr\bbA^{-1}(z))dz
$$
$$
=-\frac{1}{2h\pi i\sqrt{\ln
 h^{-1}}}\sum\limits_{k=1}^n(\E_k-\E_{k-1})\oint_{\mathcal{C}_1}
\int^{x}_{-\infty}K(\frac{t-z}{h})dt\big[\ln \beta_k(z)\big]'dz
$$
\begin{equation}
=\frac{1}{2h\pi i\sqrt{\ln
 h^{-1}}}\sum\limits_{k=1}^n(\E_k-\E_{k-1})\oint_{\mathcal{C}_1}
K(\frac{x-z}{h})\ln
\Big(\frac{\beta_k^{\tr}(z)}{\beta_k(z)}\Big)dz,\label{g40}
\end{equation}
where in the last step one uses the fact that via (\ref{a25})
\begin{equation}\label{g47}
\Big[\int^{x}_{-\infty}K(\frac{t-z}{h})dt\Big]'=K(\frac{x-z}{h}).
\end{equation}
It is observed that the unique difference between (\ref{g40}) and
(\ref{b12}) is that the test function $K'(\frac{x-z}{h})$ there is
replaced by $K(\frac{x-z}{h})$. Therefore, repeating the arguments
in Section 2 we obtain that (\ref{g40}) is asymptotically normal
with covariance (see (\ref{g33}) and (\ref{f62}))
\begin{equation}\label{g43}
-\frac{1}{2h^2\pi^2\ln
 h^{-1}}\oint_{\mathcal{C}_1}\oint_{\mathcal{C}_2}
K(\frac{x_1-z_1}{h})K(\frac{x_2-z_2}{h})\ln(\underline{m}_n^0(z_1)-\underline{m}_n^0(z_2))dz_1dz_2+O(\frac{1}{nh^2}).
\end{equation}

Also, for the nonrandom part we have
\begin{equation}\label{g48}
\frac{1}{2h\pi i\sqrt{\ln  h^{-1}}}\oint_{\mathcal{C}_1}
\Big[\int^{x}_{-\infty}K(\frac{t-z}{h})dt\Big]n(\E\tr\bbA^{-1}(z)-nm_n^0(z))dz.
\end{equation}
Note that
$$
|h^{-1}\int^{x}_{-\infty}K(\frac{t-z}{h})dt|<\infty.
$$
Likewise, repeating the arguments in Section 3 we see that
(\ref{g48}) becomes (see (\ref{g32}))
\begin{equation}\label{g44}
\frac{1}{4h\pi i\sqrt{\ln  h^{-1}}}\oint
\Big[\int^{x}_{-\infty}K(\frac{t-z}{h})dt\Big]\frac{c_n(\underline{m}_n^0(z))^3}{(1+\underline{m}_n^0(z))^3}
\big(1-\frac{c_n(\underline{m}_n^0(z))^2}{(1+\underline{m}_n^0(z))^2}\big)^{-2}dz
+O(\frac{1}{nh^2\sqrt{\ln h^{-1}}}).
\end{equation}
The limits of (\ref{g43}) and (\ref{g44}) are derived in Appendix 3.

Applying a change of variables and Fubini's theorem  we obtain
\begin{eqnarray}\label{v1}
\int^x_{-\infty}\Big[\frac{1}{h}\int
K(\frac{t-y}{h})d\F_{c_n}(y)\Big]dt&=&\int^x_{-\infty}\Big(\int
K(y)f_{c_n}(t-hy)dy\Big)dt\non =\int\Big(
K(y)\int^x_{-\infty}f_{c_n}(t-hy)dt\Big)dy&=&\int
K(y)\F_{c_n}(x-hy)dy.
\end{eqnarray}
By Taylor's expansion we have
$$
\F_{c_n}(x-hy)=\F_{c_n}(x)+hyf_{c_n}(x)+2^{-1}h^2y^2f'_{c_n}(x-\theta
hy),
$$
where $\theta\in (0,1)$. This, together with (\ref{v1}), yields that
\begin{eqnarray*}
&&\int^x_{-\infty}\Big[\frac{1}{h}\int
K(\frac{t-y}{h})d\F_{c_n}(y)\Big]dt\\
&=&\int_{|y| \leq x_0/(2h)} K(y)\big(\F_{c_n}(x)+hyf_{c_n}(x)+2^{-1}h^2y^2f'_{c_n}(x-\theta hy)\big)dy\\
&& \quad +\int_{|y| > x_0/(2h)} K(y)\F_{c_n}(x-hy)dy,
\end{eqnarray*}
where $x_0=\min(x, x-a,b-x)$ is positive since $x\in (a,b)$. Note that
\begin{eqnarray*}
\big|\int_{|y| > x_0/(2h)} K(y)\F_{c_n}(x-hy)dy\big| &\leq& (4h^2/x_0^2)\int y^2|K(y)|dy, \\
\big|\int_{|y| \leq x_0/(2h)} yK(y) dy\big| =\big|\int_{|y| >
x_0/(2h)} yK(y) dy\big| &\leq&( 2h/x_0)\int y^2|K(y)|dy
\end{eqnarray*}
and $f'_{c_n}(x-\theta hy)$ is bounded above by a finite constant
depending only on $x$, and that
$$
\frac{n}{\sqrt{\ln
h^{-1}}}\Big(F_n(x)-\int^x_{-\infty}\Big[\frac{1}{h}\int
K(\frac{t-y}{h})d\F_{c_n}(y)\Big]dt\Big)
$$$$=\frac{n}{\sqrt{\ln h^{-1}}}\Big(F_n(x)-\F_{c_n}(x)\Big)+O\big(\frac{nh^2}{2\sqrt{\ln h^{-1}}}\big).
$$
Hence the proof of Theorem \ref{theo2} is complete. \qed

\section{The proof of Theorem \ref{theo3}} \label{proof-2}

For any $x$, write
$$
\P\Big(\frac{n}{\sqrt{\ln n}}\big(x_{n,\alpha}-x_\alpha\big)\leq x\Big)=\P\Big(F_n\big(x_{\alpha}+\frac{x\sqrt{\ln n}}{n}\big)\geq \alpha\Big)
=\P\big(\hat{F}_n(x)\geq g_n(x)\big),
$$
where
$$
\hat{F}_n(x)=\frac{n}{\sqrt{\ln n}}\big[F_n\big(x_{\alpha}+\frac{x\sqrt{\ln n}}{n}\big)-\F_{c_n}\big(x_{\alpha}+\frac{x\sqrt{\ln n}}{n}\big) \big]
$$
and
$$
g_n(x)=\frac{n}{\sqrt{\ln n}}\big[\alpha-\F_{c_n}\big(x_{\alpha}+\frac{x\sqrt{\ln n}}{n}\big)\big].
$$
By Taylor's expansion we have
$$
g_n(x)\rightarrow -xf_{c}(x_\alpha)
$$
and
$$
\hat{F}_n(x)=\frac{n}{\sqrt{\ln n}}\big[F_n(x_{\alpha})-\F_{c_n}(x_{\alpha}) \big]+o_p(1),
$$
where we use Theorem \ref{rem2} and the fact that $F_n(x)$ and $\F_c(x)$ are both continuous. Theorem \ref{theo3} then follows the above and Theorem \ref{theo2}.

\section{Appendix 1} \label{appen1}

This Appendix collects some frequently used Lemmas.

\begin{lemma}
\label{lem8} (Lemma 2.2 of \cite{b2}) Suppose that $X_1,\cdots,X_n$
are i.i.d real random variables with $\E X_1=0$ and $\E X_1^2=1$. Let
$\bbx=(X_1,\cdots,X_n)^T$ and $\bbD$ be any $n\times n$ complex
matrix. Then for any $p\geq2$
$$
\E|\bbx^T\bbD\bbx-\tr\bbD|^p\leq
M_p\big[(\E|X_1|^4\tr\bbD\bbD^*)^{p/2}+\E|X_1|^{2p}\tr(\bbD\bbD^*)^{p/2}\big].
$$
\end{lemma}

\begin{lemma} Assume that $v\geq M/\sqrt{n}$ and $u\in [a,b]$. Then
\label{lem1} $$|\underline{m}_n^0(z)|\leq M, \ |\E m_n(z)|\leq M,\
|b_1(z)|\leq M, \ \E|\beta_k^{\tr}(z)|^8\leq M,\
\E|\beta_k(z)|^8\leq M;$$
\begin{equation}\label{a39}
\E|\Gamma_k|^8 \leq M/(n^8v^{12}),\quad \E|\Gamma_k^{(2)}|^8\leq
M/(n^8v^{20}),\quad \E|\xi_k(z)|^8\leq
M/(n^4v^4);
\end{equation}
for $m_1=1,2$ and $m_2=0,1,2,m_3=0,1,2,$
\begin{equation}\label{b18}
\E\big|\centre^k\big(\bbs_k^T\bbA_k^{-m_1}(z)\underline{\bbA}_k^{-m_2}(z)\bbA_k^{-m_3}(z)\bbs_k\big)
\big|^8\leq M/(n^4v^{8m_1+8m_2+8m_3-4}),
\end{equation}
where $\centre^k$ is defined in (\ref{h24}) and $\underline{\bbA}_k^{-1}(z)$ defined right before Section 3.1;
\begin{equation}
|u_n(z)|=|z-(1-n^{-1})b_{12}(z)|^{-1}\leq M.\label{m44}
\end{equation}
\end{lemma}
\begin{remark}
Checking the argument of Lemma \ref{lem1} indicates
that all above estimates involving $\bbA_k^{-1}(z)$ (and $\underline{\bbA}_k^{-1}(z)$) still hold if replacing $\bbA_k^{-1}(z)$ (and
$\underline{\bbA}_k^{-1}(z)$) by $\bbA_{kj}^{-1}(z)$ (and
$\underline{\bbA}_{kj}^{-1}(z)$) respectively.
\end{remark}

\begin{proof} As pointed out in (6.1) in \cite{ker}, we obtain
 \begin{equation}\label{b15}
|\underline{m}_n^0(z)|\leq M,\quad |m_n^0(z)|\leq M.
\end{equation}
It was proved in \cite{g2} that
 \begin{equation}\label{b16}
|\E m_n(z)|\leq M, \quad  |\E\underline{m}_n(z)|\leq M,\quad  |b_1(z)|\leq M.
\end{equation}
See Lemma 6.2 of \cite {g2} for the first estimate of (\ref{a39}) and Cauchy's theorem ensures the second estimate of (\ref{a39}) via the first estimate of (\ref{a39}).

Write
\begin{equation}\label{b17}
\beta_1^{\tr}=b_1-\beta_1^{\tr}b_1\Gamma_1=b_1-b_1^2\Gamma_1+\beta_1^{\tr}b^2_1\Gamma_1^2.
\end{equation}
We then conclude from (\ref{h25}), (\ref{b17}) and Lemma 6.2 of \cite {g2} that
\begin{equation}
\E|\beta_1^{\tr}|^8\leq
M(1+v^{-8}E|\Gamma_1|^{16}) \leq M.\label{g15}
\end{equation}
Expand $\beta_1(z)$ as
\begin{equation}\label{g30}
\beta_1=\beta_1^{\tr}-\beta_1^{\tr}\beta_1\eta_1=\beta_1^{\tr}-(\beta_1^{\tr})^2\eta_1+(\beta_1^{\tr})^2\beta_1\eta^2_1.
\end{equation}
It follows from (\ref{g15}), (\ref{f44}) and Lemma \ref{lem8}
that
$$
\E|\beta_1(z)|^8\leq
M+ME|(\beta_1^{\tr})^2\eta_1|^8+ M v^{-8}\E|\eta_1(z)\beta_1^{\tr}(z)|^{16}\leq M.
$$

From (\ref{f10}) and (\ref{b16}) we have
\begin{equation}\label{c3}|n^{-1}\E\tr\bbA_k^{-1}(z)|\leq M.
\end{equation}
As for (\ref{b18}) by Lemma \ref{lem8}, (\ref{f44}), (\ref{a39}) and (\ref{c3}) we then obtain
$$
\E\big|\centre^k\big(\bbs_k^T\bbA_k^{-m_1}(z)\underline{\bbA}_k^{-m_2}(z)\bbA_k^{-m_3}(z)\bbs_k\big)\big|^8
$$$$\leq M n^{-4}\E( n^{-1}\tr\bbA_k^{-m_1}(z)\underline{\bbA}_k^{-m_2}(z)\bbA_k^{-m_3}(z)\bbA_k^{-m_3}(\bar z)\underline{\bbA}_k^{-m_2}(\bar z)\bbA_k^{-m_1}(z))^4
$$
$$
\leq M n^{-4}v^{-8m_1-8m_2-8m_3+4}\big(\E|\Gamma_k|^4+|\Im(n^{-1}
\E\tr\bbA_k^{-1}(z))|^4\big)\leq  M n^{-4}v^{-8m_1-8m_2-8m_3+4},
$$
where $\bbA_k^{-1}(\bar z)$ denotes the complex conjugate of
$\bbA_k^{-1}(z)$ and we also use the fact that
$$
n^{-1}\E\tr\bbA_k^{-1}(z)\bbA_k^{-1}(\bar{z})=v^{-1}\Im( n^{-1}\E\tr\bbA_k^{-1}(z)).
$$
This, together with (\ref{a39}), yields the estimate of $\xi_1(z)$.

Via (\ref{m45}) we have
\begin{equation}\label{m46}
|z+z\underline{m}_n^0(z)|^{-1}=|z+zc_nm_n^0(z)-1+c_n|^{-1}=|m_n^0(z)|\leq M
\end{equation}
which implies that $\big(z+(1-n^{-1})z\underline{m}_n^0(z)\big)^{-1}$ is bounded. By (\ref{m46}), (\ref{g36}) and the equality
$$
u_n(z)-\big(z+(1-n^{-1})z\underline{m}_n^0(z)\big)^{-1}=(1-n^{-1})u_n(z)(b_{12}(z)+z\underline{m}_n^0(z))\big(z+(1-n^{-1})z\underline{m}_n^0(z)\big)^{-1}
$$
we obtain
$$
|u_n(z)|\leq M(1-n^{-1}v^{-3/2})^{-1}|z+(1-n^{-1})z\underline{m}_n^0(z)|^{-1}\leq M.
$$
This implies (\ref{m44}).
\end{proof}

\begin{lemma}\label{lem7} Assume that $v\geq M_3/\sqrt{n}$ with $M_3$ being sufficiently large and $u\in [a,b]$. Then
\begin{equation}
n^{-2}\E\big|\tr\E_k\bbA_k^{-1}(z_1)
\underline{\bbA}_k^{-1}(z_2)-\E\tr\bbA_k^{-1}(z_1)
\underline{\bbA}_k^{-1}(z_2)\big|^2  \leq
  M/(n^2v^5),
\label{f5}
\end{equation}
\begin{equation}\label{f5a}
n^{-2}\E\big|\tr\E_k\bbA_k^{-1}(z)\underline{\bbA}_k^{-2}(z)-\E\tr\bbA_k^{-1}(z)\underline{\bbA}_k^{-2}(z)\big|^2
\leq M/(n^2v^7),
\end{equation}
\begin{equation}\label{f5b}
n^{-2}\E\big|\tr\bbA_{k}^{-2}(z)\underline{\bbA}_{k}^{-2}(z)-\E\tr\bbA_{k}^{-2}(z)\underline{\bbA}_{k}^{-2}(z)\big|^2
\leq M/(n^2v^9).
\end{equation}
and
\begin{equation}\label{g27}
|g(z)|^{-1}
\E|\Gamma_1|^8\leq M/(n^8v^{12}).
\end{equation}
\end{lemma}
\begin{remark}
Checking on the argument of (\ref{f5}) shows that (\ref{f5}) is still true when the notation $\E_k$ is removed.
\end{remark}

\begin{proof} We begin with a martingale decomposition of the random variable of
interest:
$$
n^{-1}\tr\bbA_k^{-1}(z_1)
\E_k\bbA_k^{-1}(z_2)-\E\big(n^{-1}\tr\bbA_k^{-1}(z_1)
\E_k \bbA_k^{-1}(z_2)\big)
$$
$$
=n^{-1}\sum\limits_{j\neq
k}^{n}(\E_j-\E_{j-1})\big[\tr\bbA_k^{-1}(z_1)
\E_k \bbA_k^{-1}(z_2)\big]
$$
$$
=n^{-1}\sum\limits_{j\neq
k}^{n}(\E_j-\E_{j-1})\big[\tr\bbA_k^{-1}(z_1)
\E_k\bbA_{k}^{-1}(z_2)-\tr\bbA_{kj}^{-1}(z_1)
\E_k\bbA_{kj}^{-1}(z_2)\big]
$$
$$
=n^{-1}\sum\limits_{j\neq
k}^{n}(\E_j-\E_{j-1})(\delta_1+\delta_2+\delta_3),
$$
where, via (\ref{f6}),
$$
\delta_1=\beta_{kj}(z_1)\bbs_j^T\bbA_{kj}^{-1}(z_1)
\E_k\big(\beta_{kj}(z_2)\bbA_{kj}^{-1}(z_2)\bbs_j\bbs_j^T\bbA_{kj}^{-1}(z_2)\big)\bbA_{kj}^{-1}(z_1)\bbs_j
$$
$$
\delta_2=-\beta_{kj}(z_1)\bbs_j^T\bbA_{kj}^{-1}(z_1)
\E_k\big(\bbA_{kj}^{-1}(z_2)\big)\bbA_{kj}^{-1}(z_1)\bbs_j
$$
and
$$
\delta_3=-\tr\bbA_{kj}^{-1}(z_1)
\E_k\big(\beta_{kj}(z_2)\bbA_{kj}^{-1}(z_2)\bbs_j\bbs_j^T\bbA_{kj}^{-1}(z_2)\big).
$$

It follows from (\ref{f7}) that
\begin{equation}\label{f12}
|\delta_1|\leq v^{-2}.
\end{equation}
This implies that when $j>k$,
$$
(\E_j-\E_{j-1})\delta_1=(\E_j-\E_{j-1})b_{12}(z_1)(\delta_{11}-\delta_{12}),
$$
where $\delta_{12}=\xi_{kj}(z_1)\delta_1$ and
$$
\delta_{11}=\bbs_j^T\bbA_{kj}^{-1}(z_1)
\E_k\big(\beta_{kj}(z_2)\bbG_k(z_2)\big)\bbA_{kj}^{-1}(z_1)\bbs_j
-n^{-1}\tr\bbA_{kj}^{-1}(z_1)
\E_k\big(\beta_{kj}(z_2)\bbG_k(z_2)\big)\bbA_{kj}^{-1}(z_1)
$$
 with
$\bbG_k(z_2)=\bbA_{kj}^{-1}(z_2)\bbs_j\bbs_j^T\bbA_{kj}^{-1}(z_2)$.
We conclude from (\ref{f7}), (\ref{f12}), (\ref{f15}) and Lemma \ref{lem8} that
$$
\E| n^{-1}\sum\limits_{j\neq
k}^{n}(\E_j-\E_{j-1})(\delta_{11}+\delta_{12})|^2\leq
n^{-2}\sum\limits_{j\neq
k}^{n}\E|\delta_{11}|^2+\E|\delta_{12}|^2\leq M/(n^2v^5).
$$

For handling the case $j<k$, let
$$\alpha_{k1}=\bbs_j^T\bbA_{kj}^{-1}(z_1)\underline{\bbA}_{kj}^{-1}(z_2)\bbs_j,\quad
\zeta_{kj1}=\alpha_{k1}-n^{-1}\tr\bbA_{kj}^{-1}(z_1)
\underline{\bbA}_{kj}^{-1}(z_2).
$$  Applying (\ref{f13}) and
the equality for $\underline{\beta}_{kj}(z_2)$ similar to
(\ref{f13}) yields
$$
(\E_j-\E_{j-1})\delta_1=(\E_j-\E_{j-1})\big[\beta_{kj}(z_1)\underline{\beta}_{kj}(z_2)\alpha_{k1}^2
\big]
$$
$$
=b_{12}(z_1)b_{12}(z_2)[\delta_{13}+2\delta_{14}+\delta_{15}+\delta_{16}+\delta_{17}],
$$
where
$$
\delta_{13}=(\E_j-\E_{j-1})\big(\zeta_{kj1}^2\big),
\delta_{14}=(\E_j-\E_{j-1})\big(\zeta_{kj1}n^{-1}\tr\bbA_{kj}^{-1}(z_1)
\underline{\bbA}_{kj}^{-1}(z_2)\big),
$$
$$
\delta_{15}=-(\E_j-\E_{j-1})\big[\beta_{kj}(z_1)\xi_{kj}(z_1)\alpha_{k1}^2\big],\
\delta_{16}=-(\E_j-\E_{j-1})\big[\underline{\beta}_{kj}(z_2)\underline{\xi}_{kj}(z_2)\alpha_{k1}^2\big]
$$
and
$$
\delta_{17}=(\E_j-\E_{j-1})\big[\beta_{kj}(z_1)\underline{\beta}_{kj}(z_2)\xi_{kj}(z_1)\underline{\xi}_{kj}(z_2)\alpha_{k1}^2\big].
$$

It follows from Lemma \ref{lem1} that
\begin{equation}\label{f16}
\E|\zeta_{kj1}|^4\leq
M/(n^2v^6).
\end{equation}
In view of (\ref{f16}) and (\ref{f3*}),
$$
\E|n^{-1}\sum\limits_{j\neq
k}^{n}(\E_j-\E_{j-1})b_{12}(z_1)b_{12}(z_2)(\delta_{13})|^2\leq
M/(n^3v^6).
$$
While (\ref{g8}) and (\ref{f3*}) yield
 $$
\E| n^{-1}\sum\limits_{j\neq
k}^{n}(\E_j-\E_{j-1})b_{12}(z_1)b_{12}(z_2)(\delta_{14})|^2\leq
 M/(n^2v^5).
$$
It follows from (\ref{f7}) that
$$
|\beta_{kj}(z_1)\alpha_{k1}\alpha_{k2}|\leq Mv^{-1}\|\underline{\bbA}_{kj}^{-1}(z_2)\bbs_j\|^2=Mv^{-1}\bbs_j^T\underline{\bbA}_{kj}^{-1}(\bar z_2)
\underline{\bbA}_{kj}^{-1}(z_2)\bbs_j.
$$
This, together with estimates similar to (\ref{g8}) and (\ref{f16}), ensures that
$$
\E|n^{-1}\sum\limits_{j\neq
k}^{n}(\E_j-\E_{j-1})b_{12}(z_1)b_{12}(z_2)(\delta_{15})|^2
\leq M/(n^2v^5).
$$
Obviously, this estimate applies to the term involving
$\delta_{16}$. From (\ref{f7}) and Lemma \ref{lem1} we obtain
$$
\E| n^{-1}\sum\limits_{j\neq
k}^{n}(\E_j-\E_{j-1})b_{12}(z_1)b_{12}(z_2)(\delta_{17})|^2\leq M/(n^3v^6).
$$
Summarizing the above we have
$$
\E|n^{-1}\sum\limits_{j\neq
k}^{n}(\E_j-\E_{j-1})b_{12}(z_1)b_{12}(z_2)(\delta_{1})|^2\leq M/(n^2v^5).
$$

Applying an argument similar to that for $\delta_1$ in the case of $j>k$ one may prove that
$$
\E| n^{-1}\sum\limits_{j\neq
k}^{n}(\E_j-\E_{j-1})b_{12}(z_1)b_{12}(z_2)(\delta_{2})|^2\leq M/(n^2v^5).
$$
When $j>k$
$$
n^{-1}\sum\limits_{j\neq
k}^{n}(\E_j-\E_{j-1})b_{12}(z_1)b_{12}(z_2)(\delta_{3})=0.
$$
When $j<k$, as in dealing with $\delta_1$ in the case of $j>k$ one may verify that
$$
\E| n^{-1}\sum\limits_{j\neq
k}^{n}(\E_j-\E_{j-1})b_{12}(z_1)b_{12}(z_2)(\delta_{3})|^2\leq M/(n^2v^5).
$$
Thus, the proof of (\ref{f5}) is complete.

As in (\ref{g19}), by Cauchy's theorem one may verify (\ref{f5a}). Following the proof of (\ref{f5}) one can prove (\ref{f5b}) and the details are omitted here.

Consider (\ref{g27}) next. Thanks to (\ref{f10}), it is enough to consider $\tr\bbA^{-1}(z)-\E\tr\bbA^{-1}(z)$ rather than $\Gamma_1$.  As in (\ref{g39}) write
$$
\tr\bbA^{-1}(z)-\E\tr\bbA^{-1}(z)=-\sum\limits_{k=1}^n (\E_k-\E_{k-1} )(\beta_k\bbs_k^T\bbA_k^{-2}(z)\bbs_k)
$$
$$
=-\sum\limits_{k=1}^n (\E_k-\E_{k-1} )(b_1\eta_k^{(2)})+\sum\limits_{k=1}^n (\E_k-\E_{k-1} )
(\beta_kb_1\xi_k\bbs_k^T\bbA_k^{-2}\bbs_k),
$$
where the last step uses (\ref{b22}). It follows from (\ref{g29}), Lemmas \ref{lem8}, \ref{lem1} and Burkholder's inequality that
$$
n^{-8}|g(z)|^{-1}\E\big|\sum\limits_{k=1}^n (\E_k-\E_{k-1} )(\eta_k^{(2)})\big|^8\leq n^{-12}v^{-8}|g(z)|^{-1}\E|M\tr\bbA_1^{-1}(z)\bbA_1^{-1}(\bar z)|^4\leq M/(n^8v^{12}).
$$
Similarly, via (\ref{m17}) we obtain
$$
n^{-8}|g(z)|^{-1}\E\big|\sum\limits_{k=1}^n (\E_k-\E_{k-1} )(\beta_k\eta_k\bbs_k^T\bbA_k^{-2}\bbs_k)\big|^8
\leq M/(n^8v^{12})
$$
and via (\ref{b23}), (\ref{m17}) and Lemma \ref{lem11a}
$$
n^{-8}|g(z)|^{-1}\E\big|\sum\limits_{k=1}^n (\E_k-\E_{k-1} )(\beta_k\Gamma_k\bbs_k^T\bbA_k^{-2}\bbs_k)\big|^8\leq Mn^{-4}v^{-8}|g(z)|^{-1}E|\Gamma_1|^8
$$$$\leq Mn^{-4}v^{-8}|g(z)|^{-1}\E|n^{-1}\tr\bbA^{-1}(z)-n^{-1}\E\tr\bbA^{-1}(z)|^8+M/(n^8v^{12}).
$$
Summarizing the above we have
$$
(1-Mn^{-4}v^{-8})\E|n^{-1}\tr\bbA^{-1}(z)-n^{-1}\E\tr\bbA^{-1}(z)|^8\leq M/(n^8v^{12}),
$$
which implies (\ref{g27}).
\end{proof}

\section{Appendix 2} \label{appen2}

The aim in this section is to develop the asymptotic means and variances in Theorem \ref{theo1} and Theorem \ref{theo2}. Consider (\ref{f62}) first. Note that
\begin{eqnarray*}
(\ref{f62})
&=&-\frac{1}{2h^2\pi^2}\oint_{\mathcal{C}_1}\oint_{\mathcal{C}_2}
K'(\frac{x_1-z_1}{h})K'(\frac{x_2-z_2}{h})\label{e1}\\
&&\qquad\quad\times\big[\ln\big|\underline{m}_n^0(z_1)-\underline{m}_n^0(z_2)\big|+i
\arg(\underline{m}_n^0(z_1)-\underline{m}_n^0(z_2))\big]dz_1dz_2,
\end{eqnarray*}
where the contours $\mathcal{C}_1$ and $\mathcal{C}_2$ are two
rectangles defined in (\ref{a3*}) and (\ref{g53}), respectively.

As in Section 5 of \cite{b2} one may prove that
\begin{equation}\label{f58}
\inf\limits_{z\in S,n}|\underline{m}_n^0(z)|>0,\quad
\Big|\frac{\underline{m}_n^0(z_1)-\underline{m}_n^0(z_2)}{z_1-z_2}\Big|\geq\frac{1}{2}|\underline{m}_n^0(z_1)\underline{m}_n^0(z_1)|,
\end{equation}
where $S$ is any bounded subset of $\mathbb{C}$.

To facilitate statements, denote the real parts of $z_j$ by
$u_j$,$j=1,2$. In what follows, let $n\rightarrow\infty$ first and
then $v_0\rightarrow 0$. Then, as argued in \cite{b2}, the integrals
in (\ref{e1}) involving the arg term and the vertical sides approach
zero.

Define
$$
K^{(1)}_{ri}=K_r'(\frac{x_1-z_1}{h})K_r'(\frac{x_2-z_2}{h})-K_i'(\frac{x_1-z_1}{h})K_i'(\frac{x_2-z_2}{h}),
$$
$$
K^{(2)}_{ri}=K_r'(\frac{x_1-z_1}{h})K_r'(\frac{x_2-z_2}{h})+K_i'(\frac{x_1-z_1}{h})K_i'(\frac{x_2-z_2}{h}).
$$
 Therefore it is enough to investigate the following integrals
$$
-\frac{1}{h^2\pi^2}\int_{a_l}^{a_r}\int_{a_l-\varepsilon}
^{a_r+\varepsilon}
[K^{(1)}_{ri}\ln|\underline{m}_n^0(z_1)-\underline{m}_n^0(z_2)|\non
-K^{(2)}_{ri}\ln|\underline{m}_n^0(z_1)-\overline{\underline{m}_n^0}(z_2)|]du_1du_2\non
$$
\begin{eqnarray}
&&=\frac{1}{h^2\pi^2}\int_{a_l}^{a_r}\int_{a_l-\varepsilon}
^{a_r+\varepsilon}
(K_r'(\frac{x_1-z_1}{h})K_r'(\frac{x_2-z_2}{h})\ln\Big|\frac{\underline{m}_n^0(z_1)-\overline{\underline{m}_n^0}(z_2)}{\underline{m}_n^0(z_1)-\underline{m}_n^0(z_2)}\Big|du_1du_2
\label{e4}\\
&&+\frac{1}{h^2\pi^2}\int_{a_l}^{a_r}\int_{a_l-\varepsilon}
^{a_r+\varepsilon}
(K_i'(\frac{x_1-z_1}{h})K_i'(\frac{x_2-z_2}{h})\label{e2}\\
&&\qquad\qquad\times\ln\Big|(\underline{m}_n^0(z_1)-\underline{m}_n^0(z_2))(\underline{m}_n^0(z_1)-\overline{\underline{m}_n^0}(z_2))\Big|du_1du_2\nonumber
,\end{eqnarray} where $K_r'(h^{-1}(x-z))$ and
$K_i'(h^{-1}(x-z))$, respectively, represent the real part and
imaginary part of $K'(h^{-1}(x-z))$,
$\overline{\underline{m}_n^0}(z)$ stands for the complex conjugate
of $\underline{m}_n^0(z)$.

We develop the limit of (\ref{e4}) and (\ref{e2}) below. To this
end, we list some facts below.  By (\ref{a25}) and (\ref{a26}) one
may verify that
\begin{equation}\label{a13*}
\int_{-\infty}^{+\infty} \int_{-\infty}^{+\infty}\Big|K'(u_1)
K'(u_2)\ln (u_1-u_2)^2\Big|du_1du_2<\infty.
\end{equation}
In addition, it follows from (\ref{a25}) that
$$
(\ln h^{-2})\int_{\frac{x-b}{h}}^{\frac{x-a}{h}}
K_r'(u_1)\int_{\frac{x-b-\varepsilon}{h}}
^{\frac{x-a+\varepsilon}{h}} K_r'(u_2)du_1du_2\rightarrow 0.
$$
This, together with (\ref{a13*}), implies that as
$n\rightarrow\infty$
\begin{eqnarray}
&&\frac{1}{h^2}\int_{x_1-a_r}^{x_1-a_l}\int_{x_2-a_r-\varepsilon}
^{x_2-a_l+\varepsilon} K'(\frac{u_1}{h})K'(\frac{u_2}{h})\ln
(u_1-u_2)^2du_1du_2\non
&=&\int_{\frac{x_1-a_r}{h}}^{\frac{x_1-a_l}{h}}\int_{\frac{x_2-a_r-\varepsilon}{h}}
^{\frac{x_2-a_l+\varepsilon}{h}} K'(u_1)K'(u_2)\Big[\ln
(u_1-u_2)^2-\ln \frac{1}{h^2}\Big]du_1du_2 \non
&\rightarrow&\int_{-\infty}^{+\infty}
\int_{-\infty}^{+\infty}K'(u_1) K'(u_2)\ln
(u_1-u_2)^2du_1du_2.\label{f75}
\end{eqnarray}
By (\ref{f63}) and the continuity property of $K''(u+iv_0)$ and
$K'(u+iv_0)$ in $u$ and $v_0$ it is not difficult to prove that
\begin{equation}
\label{f60}
\lim\limits_{v_0\rightarrow0}\int^{+\infty}_{-\infty}|K''(u+iv_0)|du=\int^{+\infty}_{-\infty}|K''(u)|du
\end{equation}
and
\begin{equation}\label{f72}
\lim\limits_{v_0\rightarrow0}\int^{+\infty}_{-\infty}K^{(j)}(u+iv_0)du=\int^{+\infty}_{-\infty}K^{(j)}(u)du,\ j=0,1,
\end{equation}
where $K^{(j)}$ is the $j$-th derivative of $K$.

By complex Roller's theorem
\begin{equation}\label{e5}
K_i'(\frac{x-z_1}{h})=K_i'(\frac{x-u_1}{h}+iv_0)=vK_r''(\frac{x-u}{h}+iv_1)
\end{equation}
because $K_i'(h^{-1}(x-z))=0$, where $v_1$ lies in $(0,v_0))$.
Thus we conclude from (\ref{f58}) and (\ref{f60}) that
$$
|\frac{1}{h}\int_{a_l}^{a_r}
(K_i'(\frac{x_1-z_1}{h})\ln\Big|(\underline{m}_n^0(z_1)-\underline{m}_n^0(z_2))(\underline{m}_n^0(z_1)-\overline{\underline{m}_n^0}(z_2))\Big|du_1|
$$
$$
\leq v_0h\ln
(v_0^{-1}h)\frac{1}{h}\int_{a}^b|K''(\frac{x-u}{h}+iv_1)|du_1\rightarrow
0,
$$
as $n\rightarrow\infty,\ v_0\rightarrow 0$, which implies that
(\ref{e2}) converges to zero.

Consider (\ref{e4}) next. As in (\ref{m2*}) from (\ref{h31}) one may solve
$$
\underline{m}(z)=\frac{-(z+1-c)+\sqrt{(z+1-c)^2-4z}}{2z}.
$$
Note that the above equality still holds when $c$ and $\underline{m}(z)$ are, respectively replaced by $c_n$ and $\underline{m}_n^0(z)$. Also, when $z\rightarrow x\in [a,b]$ we have
\begin{equation}\label{g31}
\underline{m}(x)=\frac{-(x+1-c)+i\sqrt{(x-a)(b-x)}}{2x}.
\end{equation}
It follows that for $u\in
[(x-b)/h,(x-a)/h]$, as $n\rightarrow \infty$,
\begin{equation}\label{f66}
|\underline{m}_n^0(z_n)-\underline{m}(u_n)|\rightarrow 0,
\end{equation}
where $z_n=u_n-iv_0h$ with $u_n=x-uh$.

Now, as in \cite{b2}, for (\ref{e4}) write
\begin{equation}\label{e8}
\ln\Big|\frac{\underline{m}_n^0(z_1)-\overline{\underline{m}_n^0}(z_2)}{\underline{m}_n^0(z_1)-\underline{m}_n^0(z_2)}\Big|
=\frac{1}{2}\ln\Big(1+\frac{4\underline{m}_{ni}^0(z_1)\underline{m}_{ni}^0(z_2)}{|\underline{m}_n^0(z_1)-\underline{m}_n^0(z_2)|^2}\Big),
\end{equation}
where $\underline{m}_{ni}^0(z)$ denotes the imaginary part of
$\underline{m}_{n}^0(z)$. By (\ref{f58})
\begin{equation}\label{f74}
\ln\Big(1+\frac{4\underline{m}_{ni}^0(z_1)\underline{m}_{ni}^0(z_2)}{|\underline{m}_n^0(z_1)-\underline{m}_n^0(z_2)|^2}\Big)\leq
\ln\Big(1+\frac{16\underline{m}_{ni}^0(z_1)\underline{m}_{ni}^0(z_2)}{(u_1-u_2)^2|\underline{m}_n^0(z_1)\underline{m}_n^0(z_2)|^2}\Big).
\end{equation}
In view of (\ref{f58}) and Lemma \ref{lem1}
\begin{equation}\label{f73}
\sup\limits_{u_1,u_2\in
[a,b],v_1,v_2\in[v_0h,1]}\Big|\frac{\underline{m}_{ni}^0(z_1)\underline{m}_{ni}^0(z_2)}{|\underline{m}_n^0(z_1)\underline{m}_n^0(z_2)|^2}\Big|<\infty.
\end{equation}
By the generalized dominated convergence theorem we then conclude
from (\ref{f75}), (\ref{f72}), (\ref{f66}), (\ref{f74}), (\ref{f73})
that as $n\rightarrow\infty$
$$\int_{\frac{x_1-a_r}{h}}^{\frac{x_1-a_l}{h}}\int_{\frac{x_2-a_r-\varepsilon}{h}}
^{\frac{x_2-a_l+\varepsilon}{h}}K_r'(z_1)
K_r'(z_2)\Big[\ln\Big|\frac{\underline{m}_n^0(u_{n1}-iv_0h)-\overline{\underline{m}_n^0}(u_{n2}-iv_0h/2)}
{\underline{m}_n^0(u_{n1}-iv_0h)-\underline{m}_n^0(u_{n2}-iv_0h/2)}\Big|$$$$-\ln\Big|\frac{\underline{m}(u_{n1})-\overline{\underline{m}}(u_{n2})}
{\underline{m}(u_{n1})-\underline{m}(u_{n2})}\Big|\Big]du_1du_2\longrightarrow
0,
$$
where $u_{nj}=x_j-u_jh$, $j=1,2$. In addition, it follows from
(\ref{f75}), (\ref{f72}), and inequalities similar to (\ref{f74})
and (\ref{f73}) that as $n\rightarrow\infty$ and then
$v_0\rightarrow0$
\begin{eqnarray*}
\int_{\frac{x_1-a_r}{h}}^{\frac{x_1-a_l}{h}}\int_{\frac{x_2-a_r-\varepsilon}{h}}
^{\frac{x_2-a_l+\varepsilon}{h}}(K_r'(z_1) K_r'(z_2)-K_r'(u_1)
K_r'(u_2))\ln\Big|\frac{\underline{m}(u_{n1})-\overline{\underline{m}}(u_{n2})}
{\underline{m}(u_{n1})-\underline{m}(u_{n2})}\Big|du_1du_2
\rightarrow
0.
\end{eqnarray*}
Therefore (\ref{e4}) can be reduced to the following
\begin{eqnarray}
\qquad \int_{\frac{x_1-a_r}{h}}^{\frac{x_1-a_l}{h}}\int_{\frac{x_2-a_r-\varepsilon}{h}}
^{\frac{x_2-a_l+\varepsilon}{h}}K'(u_1)
K'(u_2)\ln\Big|\frac{\underline{m}(u_{n1})-\overline{\underline{m}}(u_{n2})}
{\underline{m}(u_{n1})-\underline{m}(u_{n2})}\Big|du_1du_2+o(1),\label{f71}
\end{eqnarray}
which turns to be
$$\frac{1}{h^2}\int_{x_1-a_r}^{x_1-a_l}\int_{x_2-a_r-\varepsilon}
^{x_2-a_l+\varepsilon}
K'(\frac{u_1}{h})K'(\frac{u_2}{h})\ln\Big|\frac{\underline{m}(x_1-u_1)-\overline{\underline{m}}(x_2-u_2)}
{\underline{m}(x_1-u_1)-\underline{m}(x_2-u_2)}\Big|du_1du_2+o(1).
$$

To handle (\ref{f71}), we need one more lemma:
\begin{lemma}\label{lem4} Suppose that the function $g(x_1,x_2)$ is continuous in $x_1$ and $x_2$,
\begin{equation}\label{f64}
\int_{x_1-a_r}^{x_1-a_l}\int_{x_2-a_r}^{x_2-a_l}|g(x_1-u_1,x_2-u_2)|du_1du_2<\infty
\end{equation}
and
\begin{equation}\label{f64}
\int_{x_1-a_r}^{x_1-a_l}|g(x_1-u_1,x_2)|du_1<\infty,\quad
\int_{x_2-a_r}^{x_2-a_l}|g(x_1,x_2-u_2)|du_2<\infty.
\end{equation} Then, as $n\rightarrow\infty$
\begin{equation}\label{e7}
\frac{1}{h^2}\int_{x_1-a_r}^{x_1-a_l}\int_{x_2-a_r-\varepsilon}
^{x_2-a_l+\varepsilon}
K'(\frac{u_1}{h})K'(\frac{u_2}{h})g(x_1-u_1,x_2-u_2)du_1du_2\rightarrow
0,
\end{equation}
where $x_1\neq a_l, a_r$ and $x_2\neq a_l, a_r$.
\end{lemma}

\proof
 Define the sets $G_1=(|u_1|\leq\delta_1)\cap (|u_2|>\delta_2)$, $G_2=(|u_1|>\delta_1)\cap (|u_2|\leq\delta_2)$ and $G_3=(|u_1|>\delta_1)\cap (|u_2|>\delta_2)$. Splitting the region of integration into the union of the sets $(|u_1|\leq\delta_1)\cap (|u_2|\leq\delta_2)$, $G_1$, $G_2$ and $G_3$ gives
$$\Big|\frac{1}{h^2}\int_{x_1-a_r}^{x_1-a_l}\int_{x_2-a_r-\varepsilon}
^{x_2-a_l+\varepsilon}
K'(\frac{u_1}{h})K'(\frac{u_2}{h})\Big[g(x_1-u_1,x_2-u_2)-g(x_1,x_2)\Big]du_1du_2\Big|$$
\begin{eqnarray}
\leq I_1+I_2+I_3+I_4+I_5\label{f70},
\end{eqnarray}
where
$$I_1= \sup\limits_{|u_1|\leq\delta_1,|u_2|\leq\delta_2}\Big|g(x_1-u_1,x_2-u_2)-g(x_1,x_2)\Big|\int_{-\infty}^{+\infty}
|K'(u)|du\Big|^2,
$$
$$I_2= |g(x_1,x_2)|\Big|\frac{1}{h^2}\int_{x_1-a_r}^{x_1-a_l}\int_{x_2-a_r-\varepsilon}
^{x_2-a_l+\varepsilon}I(G_1\cup G_2\cup G_3)
K'(\frac{u_1}{h})K'(\frac{u_2}{h})du_1du_2\Big|,
$$$$I_3=\Big|\frac{1}{h^2}\int_{x_1-a_r}^{x_1-a_l}\int_{x_2-a_r-\varepsilon}
^{x_2-a_l+\varepsilon}I(G_1)
K'(\frac{u_1}{h})K'(\frac{u_2}{h})g(x_1-u_1,x_2-u_2)du_1du_2\Big|,
$$
$$I_4=\Big|\frac{1}{h^2}\int_{x_1-a_r}^{x_1-a_l}\int_{x_2-a_r-\varepsilon}
^{x_2-a_l+\varepsilon}I(G_2)
K'(\frac{u_1}{h})K'(\frac{u_2}{h})g(x_1-u_1,x_2-u_2)du_1du_2\Big|
$$
and
$$I_5=\Big|\int_{x_1-a_r}^{x_1-a_l}\int_{x_2-a_r-\varepsilon}
^{x_2-a_l+\varepsilon}I(G_3)
\frac{u_1u_2}{h^2}K'(\frac{u_1}{h})K'(\frac{u_2}{h})\frac{g(x_1-u_1,x_2-u_2)}{u_1u_2}du_1du_2\Big|.
$$
Evidently, $I_1\rightarrow 0$ due to the continuity property of $g(x_1,x_2)$ when $\delta_1$ and $\delta_2$ converge to zero. As $n\rightarrow\infty$, for $I_2$ we have
$$
I_2\leq M|g(x_1,x_2)|\int\limits_{|u|>\delta/h}|
K'(u)|du\int_{-\infty}^{+\infty}
|K'(u)|du\rightarrow 0,
$$
and for $I_5$ by (\ref{f64}) we obtain
\begin{eqnarray*}
&&I_5\leq(\delta_1\delta_2)^{-1}\sup\limits_{|u_1|>\delta_1/h}|u_1K'(u_1)|\sup\limits_{|u_2|>\delta_2/h}|u_2K'(u_2)|
\\
&&\qquad\qquad\times\int_{x_1-a_r}^{x_1-a_l}\int_{x_2-a_r-\varepsilon}
^{x_2-a_l+\varepsilon}|g(x_1-u_1,x_2-u_2)|du_1du_2\rightarrow 0.
\end{eqnarray*}
Consider $I_3$. Similar to $I_5$,
$$
I_3\leq \delta_2^{-1}\sup\limits_{|u_2|>\delta_2/h}|u_2K'(u_2)|\int_{|u_1|\leq\delta_1/h}\int_{x_2-a_r-\varepsilon}
^{x_2-a_l+\varepsilon}|K'(u_1)g(x_1-u_1h,x_2-u_2)|du_1du_2.
$$
While, as $n\rightarrow\infty$ and then $\delta_1\rightarrow 0$, by the dominated convergence theorem
$$
h^{-1}\int_{|u_1|\leq\delta_1}|K'(\frac{u_1}{h})|\int_{x_2-a_r-\varepsilon}
^{x_2-a_l+\varepsilon}|(g(x_1-u_1,x_2-u_2)-g(x_1,x_2-u_2))|du_1du_2\rightarrow 0.
$$
From (\ref{f64}) we then see that $I_3\rightarrow 0$. One may
similarly prove that $I_4$ converges to zero as well. We summarize
the above that (\ref{f70}) converges to zero as $n\rightarrow\infty$
first and then both $\delta_1\rightarrow0$ and
$\delta_2\rightarrow0$. In addition, apparently,
\begin{eqnarray}
g(x_1,x_2)h^{-1}\int_{x_1-a_r}^{x_1-a_l}
K'(\frac{u}{h})du&=&g(x_1,x_2)\int_{\frac{x_1-a_l}{h}}^{\frac{x_1-a_r}{h}}
K'(u)du\label{f78}\\
&=&g(x_1,x_2)K(u)\Big|_{\frac{x_1-a_l}{h}}^{\frac{x_1-a_r}{h}}\rightarrow 0. \nonumber
\end{eqnarray}
Thus (\ref{e7}) is proved. \qed

We are now in a position to apply Lemma \ref{lem4} to (\ref{f71}).
It follows from (\ref{g31}) that $\underline{m}(x_1)\neq
\underline{m}(x_2)$ and $\underline{m}(x_1)\neq
\overline{\underline{m}}(x_2)$ whenever $x_1\neq x_2$. Also, note
(5.1) in \cite{b2}. Therefore
$g(x_1,x_2)=\ln\big( |\underline{m}(x_1)-\overline{\underline{m}}(x_2)|
|\underline{m}(x_1)-\underline{m}(x_2)|^{-1}\big)$ is continuous in $x_1$
and $x_2$. Furthermore, it is straightforward to show that $ \ln
\big(1+M\big((x_1-x_2)-(u_1-u_2)\big)^{-1}\big) $ for $u_1,u_2\in
[a_l-\varepsilon,a_r+\varepsilon]$ is Lebesgue integrable and $\ln
\big(1+M\big((x_1-x_2)-u_1\big)^{-1} \big)$ for $u_2\in
[a_l-\varepsilon,a_r+\varepsilon]$ is Lebesgue integrable. Thus, in
view of inequalities similar to (\ref{e8})-(\ref{f73}) and applying
(\ref{e7}) we have
\begin{equation}\label{f67}
\frac{1}{h^2}\int_{x_1-a_r}^{x_1-a_l}\int_{x_2-a_r-\varepsilon}
^{x_2-a_l+\varepsilon}
K'(\frac{u_1}{h})K'(\frac{u_2}{h})\ln\Big|\frac{\underline{m}(x_1-u_1)-\overline{\underline{m}}(x_2-u_2)}
{\underline{m}(x_1-u_1)-\underline{m}(x_2-u_2)}\Big|du_1du_2\rightarrow
0,
\end{equation}
which is the limit of (\ref{e4}) due to (\ref{f71}) when $x_1\neq
x_2$.

 When $x_1=x_2=x$ taking
$g(x_1,x_2)=\ln |\underline{m}(x)-\overline{\underline{m}}(x) |$
and applying (\ref{e7}) we obtain
\begin{equation}\label{f67}
\frac{1}{h^2}\int_{x-a_r}^{x-a_l}\int_{x-a_r-\varepsilon}
^{x-a_l+\varepsilon}
K'(\frac{u_1}{h})K'(\frac{u_2}{h})\ln\big|\underline{m}(x-u_1)-\overline{\underline{m}}(x-u_2)\big|du_1du_2\rightarrow 0.
\end{equation}
Here we keep in mind that the boundary points are not considered when
investigating the case $x_1=x_2=x$. Consider next
\begin{equation}\label{a12}
\frac{1}{h^2}\int_{x-b}^{x-a}
K'(\frac{u_1}{h})\int_{x-b-\varepsilon} ^{x-a+\varepsilon}
K'(\frac{u_2}{h})\ln\big|\underline{m}(x-u_1)-\underline{m}(x-u_2)\big|du_1du_2.
\end{equation}
By complex Roller's theorem we have
\begin{eqnarray}
&&\ln\big|\underline{m}(x-u_1)-\underline{m}(x-u_2)\big|\non&=&2^{-1}\ln
\big((u_1-u_2)^2[|\underline{m}'_r(x-u_3)|^2+|\underline{m}'_i(x-u_4)|^2]\big)\non
&=&2^{-1}\ln (u_1-u_2)^2
+2^{-1}g_{ri}(x-u_1,x-u_2),\label{f65}
\end{eqnarray}
where $g_{r}(x-u_1,x-u_2)=\ln
\big(|\underline{m}'_r(t_1(x-u_1)+(1-t_1)(x-u_2))|^2+|\underline{m}'_i(t_2(x-u_1)+(1-t_2)(x-u_2))|^2\big),$
$u_3=t_1u_1+(1-t_1)u_2,\ u_4=t_2u_1+(1-t_2)u_2$ and $t_1,t_2\in
(0,1)$. It follows from inequalities for $\underline{m}(x)$ similar
to (\ref{f58}) that
$$
\big|\int_{x-b}^{x-a} \int_{x-b-\varepsilon} ^{x-a+\varepsilon}
\ln |\underline{m}(x-u_1)-\underline{m}(x-u_2) |du_1du_2\big|<\infty.
$$
This, together with (\ref{f65}), ensures that
$$
\big|\int_{x-b}^{x-a} \int_{x-b-\varepsilon} ^{x-a+\varepsilon}
g_{r}(x-u_1,x-u_2)du_1du_2\big|<\infty.
$$
Similarly, one may check the remaining conditions in Lemma
\ref{lem4}. Therefore, using Lemma \ref{lem4} with
$g(x_1,x_2)=\ln|\underline{m}'(x)|^2$ gives
\begin{equation}\label{f68}
\frac{1}{h^2}\int_{x-b}^{x-a}
K'(\frac{u_1}{h})\int_{x-b-\varepsilon} ^{x-a+\varepsilon}
K'(\frac{u_2}{h})g_{r}(x-u_1,x-u_2)du_1du_2\rightarrow 0.
\end{equation}
We then conclude from (\ref{f65}), (\ref{f68}) and (\ref{f75}) that
$$
(\ref{a12})=\frac{1}{2}\frac{1}{h^2}\int_{x-b}^{x-a}
K'(\frac{u_1}{h})\int_{x-b-\varepsilon} ^{x-a+\varepsilon}
K'(\frac{u_2}{h})\ln (u_1-u_2)^2du_1du_2+o(1)
$$
\begin{equation}\label{a13}
\rightarrow\frac{1}{2}\int_{-\infty}^{+\infty}
\int_{-\infty}^{+\infty}K'(u_1) K'(u_2)\ln (u_1-u_2)^2du_1du_2.
\end{equation}
which is minus the limit of (\ref{e4}) due to (\ref{f67}) and (\ref{f71})
when $x_1=x_2$.

\noindent
{\bf Limit of (\ref{f50})}.  From an expression similar to (\ref{h31})
we obtain
$$
\frac{d}{dz}\underline{m}_n^0(z)=(\underline{m}_n^0(z))^2\big(1-c\frac{(\underline{m}_n^0(z))^2}{(1+\underline{m}_n^0(z))^2}\big)^{-1}.
$$
It follows that (\ref{f50}) becomes
$$
\frac{1}{4\pi i}\oint K(\frac{x-z}{h})\frac{d}{dz}\ln
\Big[1-c\frac{(\underline{m}_n^0(z))^2}{(1+\underline{m}_n^0(z))^2}\Big]dz
$$
\begin{equation}\label{f52}
=\frac{1}{4\pi hi}\oint K'(\frac{x-z}{h})\ln
\Big[1-c\frac{(\underline{m}_n^0(z))^2}{(1+\underline{m}_n^0(z))^2}\Big]dz.
\end{equation}
In view of (\ref{g16}) and (\ref{m12}) we see that the integrals on the two vertical lines in
(\ref{f52}) are bounded by $Mv\ln v^{-1}$, which converges to zero as $v\to 0$.
The integrals on the two horizontal lines are equal to
\begin{eqnarray}
&&\frac{1}{2\pi h}\int K_i'(\frac{x-z}{h})\ln \Big
|1-c\frac{(\underline{m}_n^0(z))^2}{(1+\underline{m}_n^0(z))^2}\Big
|du\label{f54}
\\ &+& \frac{1}{2\pi h}\int K_r'(\frac{x-z}{h})\arg
\Big
[1-c\frac{(\underline{m}_n^0(z))^2}{(1+\underline{m}_n^0(z))^2}\Big
]du.\label{f55}
\end{eqnarray}
By (\ref{g16}), (\ref{m12}) and (\ref{e5}), (\ref{f54}) is bounded by $Mv\ln v^{-1}$, converging to zero. It follows from (\ref{f66})
that
$$
\frac{(\underline{m}_n^0(z_n))^2}{(1+\underline{m}_n^0(z_n))^2}-\frac{(\underline{m}(u_n))^2}{(1+\underline{m}(u_n))^2}\rightarrow0.
$$
We then conclude from the dominated convergence theorem that
$$
\int K_r'(z)\Big[\arg \Big
(1-c\frac{(\underline{m}_n^0(z_n))^2}{(1+\underline{m}_n^0(z_n))^2}\Big)-\arg \Big
(1-c\frac{(\underline{m}(u_n))^2}{(1+\underline{m}(u_n))^2}\Big)\Big]du\rightarrow0.
$$
Moreover, by (\ref{f72}) we obtain
$$
\int (K_r'(z)-K_r'(u))
\arg \Big
[1-c\frac{(\underline{m}(u_n))^2}{(1+\underline{m}(u_n))^2}\Big
]du\rightarrow0.
$$
By (\ref{f78}) and Theorem 1A in \cite{p1} (replacing $K(x)$ there by $K'(x)$) we see that
$$
\int K_r'(u)
\arg \Big
[1-c\frac{(\underline{m}(u_n))^2}{(1+\underline{m}(u_n))^2}\Big
]du\rightarrow0.
$$
Summarizing the above yields that (\ref{f50}) converges to zero.

\noindent
{\bf Limits of (\ref{g43}) and (\ref{g44})}. Repeating the argument leading to (\ref{f71}) yields that (\ref{g43}) becomes
\begin{equation}\label{g45}\frac{1}{h^2\ln h^{-1}}\int_{x_1-a_r}^{x_1-a_l}\int_{x_2-a_r-\varepsilon}
^{x_2-a_l+\varepsilon}
K(\frac{u_1}{h})K(\frac{u_2}{h})\ln\Big|\frac{\underline{m}(x_1-u_1)-\overline{\underline{m}}(x_2-u_2)}
{\underline{m}(x_1-u_1)-\underline{m}(x_2-u_2)}\Big|du_1du_2+o(1).
\end{equation}
The argument of (\ref{f70}) in Lemma \ref{lem4} indeed also, together with (\ref{a26}),  gives
\begin{equation}\label{g46}
\frac{1}{h^2}\int_{x_1-a_r}^{x_1-a_l}\int_{x_2-a_r-\varepsilon}
^{x_2-a_l+\varepsilon}
K(\frac{u_1}{h})K(\frac{u_2}{h})g(x_1-u_1,x_2-u_2)du_1du_2-g(x_1,x_2)\rightarrow 0.
\end{equation}
This ensures that (\ref{g45}) converges to zero when $x_1\neq x_2$. When $x_1= x_2=x$, by (\ref{g46}) we have
$$
\frac{1}{h^2\ln h^{-1}}\int_{x_1-a_r}^{x_1-a_l}\int_{x_2-a_r-\varepsilon}
^{x_2-a_l+\varepsilon}
K(\frac{u_1}{h})K(\frac{u_2}{h})\ln\Big|\underline{m}(x_1-u_1)-\overline{\underline{m}}(x_2-u_2)\Big|du_1du_2\rightarrow 0.
$$
Applying (\ref{g46}) and replacing $K'(x)$ in (\ref{f75}), (\ref{f65}), (\ref{f68}) and (\ref{a13}) by $K(x)$, we can prove that
$$
-\frac{1}{h^2\ln h^{-1}}\int_{x_1-a_r}^{x_1-a_l}\int_{x_2-a_r-\varepsilon}
^{x_2-a_l+\varepsilon}
K(\frac{u_1}{h})K(\frac{u_2}{h})\ln\Big|\underline{m}(x_1-u_1)-\underline{m}(x_2-u_2)\Big|du_1du_2\rightarrow 1.
$$
Moreover, from the conditions on $h$ one may show that
$$
\frac{\ln n}{\ln h^{-1}}\rightarrow 2.
$$

Checking on the argument of (\ref{f50}) and replacing $K'(x)$ there
with $K(x)$, along with(\ref{g47}), we have
$
(\ref{g44})\rightarrow 0.
$
Thus the proof is complete.

\end{document}